\documentclass[11pt]{amsart}
\usepackage{url, 
  amssymb,setspace, mathrsfs,fontenc, comment}
 \usepackage[alphabetic]{amsrefs}
\usepackage{fullpage} 
\usepackage{color}
\usepackage{tikz-cd}
\usepackage[all]{xy}
\usepackage{amsmath}

\usepackage{url, 
  amssymb,setspace, mathrsfs,fontenc}
\usepackage{graphicx}
\usepackage[mathcal]{euscript}
\usepackage{verbatim}
\usepackage{hyperref}
\hypersetup{backref=true}
\usepackage{mathtools}
\usepackage{tikz}
\usetikzlibrary{chains}

\tikzset{node distance=2em, ch/.style={circle,draw,on chain,inner sep=2pt},chj/.style={ch,join},every path/.style={shorten >=4pt,shorten <=4pt},line width=1pt,baseline=-1ex}

\newtheorem{thm}{Theorem}
\newtheorem{lem}[thm]{Lemma}
\newtheorem{prop}[thm]{Proposition}
\newtheorem{conj}[thm]{Conjecture}
\newtheorem{cor}[thm]{Corollary}
\newtheorem{defe}[thm]{Definition}

\theoremstyle{remark}
\newtheorem{rem}[thm]{Remark}
\newtheorem{exam}[thm]{Example}

\theoremstyle{definition}

\DefineSimpleKey{bib}{myurl}
\newcommand\myurl[1]{\url{#1}}
\BibSpec{webpage}{
  +{}{\PrintAuthors} {author}
  +{,}{ \textit} {title}
  +{}{ \parenthesize} {date}
  +{,}{ \myurl} {myurl}
}
\usepackage{arydshln}

\newcommand{\nc}{\newcommand}

\nc{\ssec}{\subsection}

\nc{\on}{\operatorname}

\nc{\sE}{\mathscr{E}}
\nc{\sF}{\mathscr{F}}
\nc{\sL}{\mathscr{L}}
\nc{\sD}{\mathscr{D}}
\nc{\sA}{\mathscr{A}}

\nc{\cC}{\mathcal{C}}
\nc{\cG}{\mathcal{G}}
\nc{\cV}{\mathcal{V}}
\nc{\cK}{{k(\!(s)\!)}}
\nc {\K}{\mathcal{K}}

\nc{\cE} {\mathcal{E}}
\nc{\Kl}{\mathrm{Kl}}
\nc{\cO}{\mathcal{O}}
\nc{\cF}{\mathcal{F}}
\nc{\cZ}{\mathcal{Z}}
\nc{\cD}{\mathcal{D}}
\nc{\cDt}{\mathcal{D}^\times}
\nc{\cH}{\mathcal{H}}
\nc{\bZ}{\mathbb{Z}}
\nc{\bQ}{\mathbb{Q}}
\nc{\bR}{\mathbb{R}}
\nc{\bC}{\mathbb{C}}
\nc{\bQl}{\overline{\mathbb{Q}}_\ell}
\nc{\bQlt}{\bQl^\times} 
\nc{\FG}{\mathrm{FG}}
\nc{\dR}{\mathrm{dR}}
\nc{\uG}{\underline{G}}
\nc{\cB}{\mathcal{B}}
\nc{\cU}{\mathcal{U}}
\nc{\rat}{\mathrm{rat}}
\nc{\Hyp}{\mathrm{Hyp}}
\nc{\Lie}{\mathrm{Lie}}
      
\nc{\fF}{\mathfrak{F}}
\nc{\fB}{\mathfrak{B}}
\nc{\fZ}{\mathfrak{Z}}
\nc{\fx}{\mathfrak{x}}
\nc{\fy}{\mathfrak{y}}
\nc{\fb}{\mathfrak{b}}
\nc{\fk}{\mathfrak{k}}
\nc{\fI}{\mathfrak{i}}
\nc{\fj}{\mathfrak{j}}
\nc{\fg}{\mathfrak{g}}
\nc{\fu}{\mathfrak{u}}
\nc{\fl}{\mathfrak{l}}
\nc{\fn}{\mathfrak{n}}
\nc{\cP}{\mathcal{P}}
\nc{\ft}{\mathfrak{t}}
\nc{\fz}{\mathfrak{z}}
\nc{\fc}{\mathfrak{c}}
\nc{\fh}{\mathfrak{h}}
\nc{\fp}{\mathfrak{p}}
\nc{\bone}{\mathbf{1}}
\nc{\tg}{\mathtt{g}}
\nc{\hfg}{\widehat{\fg}}
\nc{\cg}{\check{\fg}}
\nc{\ch}{\check{\fh}}
\nc{\hG}{\check{G}}
\nc{\hg}{\widehat{\mathfrak{g}}}
\nc{\Ug}{\widehat{U}(\mathfrak{g})}

\nc{\bGm}{\mathbb{G}_m}
\nc{\bGa}{\mathbb{G}_a}
\nc{\bL}{\mathbf{L}}
\nc{\bK}{\mathbf{K}}
\nc{\bJ}{\mathbf{J}}
\nc{\bI}{\mathbf{I}}
\nc{\bV}{\mathbb{V}}
\nc{\bP}{\mathbb{P}}
\nc{\bA}{\mathbb{A}}
\nc{\bN}{\mathbb{N}}

\nc {\Q}{\mathrm{Q}}
\nc{\diag}{\mathrm{diag}}
\nc{\ev}{\mathrm{ev}}
\nc{\Res}{\mathrm{Res}}
\nc{\Fl}{\mathcal{F}\ell}
\nc{\Ad}{\mathrm{Ad}}
\nc{\ad}{\mathrm{ad}}
\nc{\pr}{\mathrm{pr}}
\nc{\Sl}{\mathfrak{sl}}
\nc{\gl}{\mathfrak{gl}}
\nc{\ra}{\rightarrow}
\nc{\tra}{\twoheadrightarrow}
\nc{\hra}{\hookrightarrow}
\nc{\quo}{\mathopen{ /\!/}}
\nc{\GL}{\mathrm{GL}}
\nc{\SL}{\mathrm{SL}}
\nc{\Sp}{\mathrm{Sp}}
\nc{\SO}{\mathrm{SO}}
\nc{\so}{\mathfrak{so}}
\nc{\PGL}{\mathrm{PGL}}
\nc{\Bun}{\mathrm{Bun}}
\nc{\supp}{\mathrm{supp}}
\nc{\bgamma}{\bar{\gamma}}
\nc{\I}{\mathrm{I}}
\nc{\II}{\mathrm{II}}
\nc{\III}{\mathrm{III}}
\nc{\ab}{\mathrm{ab}}
\nc{\td}{\mathrm{d}}
\nc{\Ht}{\mathrm{ht}}

\nc         {\rar}[1]       {\stackrel{#1}{\longrightarrow}}

\nc{\fa}{\mathfrak{a}}
\nc{\Hitch}{\mathrm{Hitch}}

\nc{\RS}{\mathrm{RS}}

\nc{\tp}{\mathfrak{p}}
\nc{\cA}{\mathcal{A}}
\nc{\cN}{\mathcal{N}}
\nc{\cW}{\mathcal{W}}

\nc{\opp}{\mathrm{opp}}
\nc{\Ind}{\mathrm{Ind}}
\nc{\sAn}{\mathrm{can}}
\nc{\Vac}{\mathrm{Vac}}
\nc{\Op}{\mathrm{Op}}
\nc{\Lg}{\check{\fg}}
\nc{\LV}{\check{V}}
\nc{\Lh}{\check{h}}
\nc{\LG}{\check{G}}
\nc{\Spec}{\mathrm{Spec}}
\nc{\End}{\mathrm{End}}
\nc{\rX}{\mathring{X}}
\nc{\ru}{\mathring{u}}

\nc{\sW}{\mathscr{W}}
\nc{\sH}{\mathscr{H}}
\nc{\sV}{\mathscr{V}}
\nc{\geom}{\mathrm{geom}}
\nc{\Irr}{\mathrm{Irr}}
\nc{\fm}{\mathfrak{m}}
\nc{\aff}{\mathrm{aff}}
\nc{\Aut}{\mathrm{Aut}}
\nc{\cJ}{\mathcal{J}}
\nc{\fs}{\mathfrak{s}}
\nc{\Stab}{\mathrm{Stab}}
\nc{\tw}{{\widetilde{w}}}
\nc{\gen}{\mathrm{gen}}
\nc{\genn}{\mathrm{genn}}
\nc{\sss}{\mathrm{ss}}
\nc{\spp}{\mathfrak{sp}}
\nc{\Hom}{\mathrm{Hom}}
\nc{\bm}{\mathbf{m}}
\nc{\HG}{\mathcal{HG}}
\nc{\Gal}{\mathrm{Gal}}
\nc{\Sym}{\mathrm{Sym}}
\nc{\rank}{\mathrm{rank}}

\nc{\tP}{\mathtt{P}}
\nc{\tL}{\mathtt{L}}
\nc{\tU}{\mathtt{U}}

\nc{\tW}{\widetilde{W}}
\nc{\Hk}{\on{Hk}}
\nc{\cL}{\mathcal{L}}
\nc{\talpha}{\widetilde{\alpha}}
\nc{\tQ}{{\widetilde{Q}}}
\nc{\ochi}{\overline{\chi}}
\nc{\tdelta}{\widetilde{\Delta}}
\nc{\wt}{\mathrm{wt}}
\nc{\fQ}{\mathfrak{Q}}

\nc{\Rep}{\mathrm{Rep}}
\nc{\Conn}{\mathrm{Conn}}
\nc{\Hecke}{\mathrm{Hecke}}
\nc{\Gr}{\mathrm{Gr}}
\nc{\GR}{\mathrm{GR}}
\nc{\IC}{\mathrm{IC}}
\nc{\Std}{\mathrm{Std}} 
\nc{\Db}{\mathrm{D}^{\mathrm{b}}}
\nc{\tr}{\mathrm{tr}}
\nc{\Hit}{\mathrm{Hit}}
\nc{\gr}{\mathrm{gr}}
\nc{\Fun}{\mathrm{Fun}~}

\newcommand{\quash}[1]{}

\AtEndDocument{\bigskip{\footnotesize

\textsc{Masoud Kamgarpour, School of Mathematics and Physics, The University of Queensland.} \par
\textit{E-mail address}: \texttt{masoud@uq.edu.au} \par
\textsc{Daxin Xu, Morningside Center of Mathematics and Hua Loo-Keng Key Laboratory of Mathematics, Academy of Mathematics and Systems Science, Chinese Academy of Sciences, Beijing 100190, China;}\par
\textsc{School of Mathematical Sciences, University of Chinese Academy of Sciences, Beijing, 100049, China.} \par
\textit{E-mail address}: \texttt{daxin.xu@amss.ac.cn} \par
\textsc{Lingfei Yi, School of Mathematics, University of Minnesota, Twin cities, Minneapolis, MN 55455 } \par
\textit{E-mail address}: \texttt{lyi@umn.edu} \par
}}

\title{Hypergeometric sheaves for classical groups via geometric Langlands} 
\author{Masoud Kamgarpour, Daxin Xu, Lingfei Yi} 
\date{\today}

\begin{document} 

\maketitle

\begin{abstract}
	In a previous paper, the first and third authors gave an explicit realization of the geometric Langlands correspondence for hypergeometric sheaves, considered as $\GL_n$-local systems.
	Certain hypergeometric local systems admit a symplectic or orthogonal structure, which can be viewed as $\hG$-local systems, for a classical group $\hG$. 
	This article aims to realize the geometric Langlands correspondence for these $\hG$-local systems. 

	We study this problem from two aspects. 
	In the first approach, we define the hypergeometric automorphic data for a classical group $G$ in the framework of Yun, 
	one of whose local components is a new class of euphotic representations in the sense of Jakob--Yun. 
	We prove the rigidity of hypergeometric automorphic data under natural assumptions, which allows us to define $\hG$-local systems $\cE_{\hG}$ on $\bGm$ as Hecke eigenvalues (in both $\ell$-adic and de Rham setting).
	In the second approach (which works only in the de Rham setting), we quantize an enhanced ramified Hitchin system, following Beilinson--Drinfeld and Zhu, and identify $\cE_{\hG}$ with certain $\hG$-opers on $\bGm$. 
	Finally, we compare these $\hG$-opers with hypergeometric local systems. 
\end{abstract}

\tableofcontents

\section{Introduction}
\subsection{Overview}

Hypergeometric functions have a long history, going back to the works of Wallis, Euler, Gauss and etc. 
The geometry underpinning hypergeometric functions emerged from Riemann's study of the local system of solutions of the Euler–Gauss hypergeometric differential equation.
In modern time, Katz \cite{KatzBook} has systematically studied the (generalized) hypergeometric local systems on the torus $\mathbb{G}_m$ (or $\mathbb{G}_m-\{1\}$) over a finite field $k$ (resp. field of complex numbers $\bC$). 
They are $\ell$-adic realizations of \textit{finite hypergeometric functions} over $k$ (resp. expressed as \textit{hypergeometric differential equations} over $\mathbb{C}$).

Let $n> m$ be two non-negative integers, $\lambda\in \mathbb{C}^{\times}$, and $\underline{\alpha}=\{\alpha_1,\dots,\alpha_n\}$, $\underline{\beta}=(\beta_1,\cdots,\beta_m)$ two sequences of complex numbers. 
We denote by $\Hyp(\underline{\alpha},\underline{\beta})$ the associated (irregular) hypergeometric differential equation (of type $(n,m)$) on $\bGm$: 
\begin{equation}\label{eq:hyp}
	\Hyp_{\lambda}(\underline{\alpha};\underline{\beta}):=\prod_{i=1}^n(t \frac{d}{dt}-\alpha_i)-(-1)^{n+m}\lambda t\prod_{j=1}^{m}(t\frac{d}{dt}-\beta_j), 
\end{equation}
where $t$ denotes a coordinate of $\bGm$.  
We abusively use this notation to denote the corresponding connection on the rank $n$ trivial bundle over $\bGm$. 
In the definition of $\ell$-adic hypergeometric local system, $\alpha_i$ and $\beta_j$'s are replaced by $\ell$-adic mutiplicative characters $\delta_i$ and $\rho_j$ of $k^{\times}$. 

Certain hypergeometric local systems admit a symplectic or orthogonal symmetry structure \cite[\S~3.4, \S~8.8]{KatzBook}. 
Roughly speaking, the present paper aims to realize the geometric Langlands correspondence for hypergeometric $\hG$-local systems, with a classical group $\hG$. 

Some examples have been explored in the work of Heinloth-Ng\^o-Yun on generalized Kloosterman sheaves \cite{HNY} and Frenkel-Gross' rigid connection for reductive groups \cite{FG, Zhu} (see \cite[\S 5.2]{XuZhu} for their relationship with hypergeometric local systems). 
When $\hG=\GL_n$, this problem is completely solved in a recent work of the first and third authors \cite{KY}. Note that in \textit{loc.cit}, they also constructed the geometric Langlands correspondence for \emph{tame} hypergeometric local systems, i.e. one may assume $n=m$ in \eqref{eq:hyp}. 
While in the current paper, we only work with \emph{wild} hypergeometric local systems, where $n> m$.

In this article, we adopt two methods for studying this problem. 

In the first approach, we introduce a class of \textit{euphotic representations} in the sense of Jakob-Yun \cite{JY} for a classical group $G$ and use them to define the \textit{hypergeometic automorphic data of $G$} in the framework of Yun \cite{YunCDM}. 
Under appropriate genericness conditions, we show the rigidity of hypergeometric automorphic data, generalizing the rigidity results in \cite{HNY,KY}. 
After \cite{YunCDM}, we can associate to hypergeometric automorphic data $\hG$-local systems on $\bGm$, as resulting Hecke eigenvalues, where $\hG$ denotes the Langlands dual group of $G$.
In \cite{KY}, they proved that, when $\hG=\GL_n$, these Hecke eigenvalues are hypergeometric local systems by comparing Frobenius traces. And we expect similar results hold for general $\hG$. 
However, obtaining an explicit expression for the trace function of these $\hG$-local systems, for general $\hG$, seems infeasible, since the geometry of the intervening moduli stack of bundles is more complicated. 

This brings us to our second method in the de Rham setting, via quantization of an appropriate ramified Hitchin system. 
The hypergeometric automorphic data gives rise to a group scheme $\cG$ over $\bP^1$ with generic fiber $G$. 
We study the associated Hitchin map on $T^*\Bun_\cG$ and determine its image $\Hitch(\bP^1)_\cG$. 
In the Galois side, an explicit space $\Op_{\cg}(\bP^1)_{\cG}$ of $\cg$-opers on $\bP^1$ with suitable singularities at $\{0,\infty\}$ associated to $\cG$ quantizes the Hitchin base $\Hitch(\bP^1)_\cG$. 
Following \cite{BD, Zhu}, we then quantize the enhanced Hitchin map to obtain Hecke eigensheaves on $\Bun_\cG$ from $\Op_{\cg}(\bP^1)_{\cG}$. 
We compare two constructions and identify opers of $\Op_{\cg}(\bP^1)_{\cG}$ with (de Rham) $\hG$-local systems produced in the previous paragraph. 
In some cases, we can relate these opers to hypergeometric local systems, which partially achieve our goal in the de Rham setting. 

\subsection{Hypergeometric automorphic data}\label{intro ss: def of hyp auto data}
The notion of automorphic data and its rigidity were developed by Yun \cite{YunCDM}, building on earlier works of Heinloth-Ng\^o-Yun and Yun \cite{HNY,YunGalois,YunEpipelagic}. 
A rigid automorphic data is a collection of level structures such that the space of corresponding equivariant automorphic forms has a small dimension, see \cite[\S~2]{YunCDM} for the precise formulation. 
We briefly explain the automorphic data considered in this paper. 

Let $k$ be a finite field of characteristic $p$, $F=k(t)$ the rational function field over $k$ where $t$ is a coordinate at $0\in|\bP^1|$, and $F_x$ the local fields of $F$ at $x\in |\bP^1|$. 
Let $G$ be a split classical group over $k$\footnote{In this article, a classical group means $\GL_n$ or an almost simple group whose root system is of type A, B, C, or D.}, $d$ a fundamental degree of $G$ (different to $n$ when $G$ is type $D_{n}$ and $n$ is odd\footnote{We exclude this case for the sake of the construction of linear functional $\phi$.}). 
We denote by $P$ the parahoric subgroup of $G(F_\infty)$ associated to $\check{\rho}/d$, where $\check{\rho}$ is the half-sum of positive coroots. 
Consider the Moy-Prasad filtration of $P$, where we denote the first three terms as $P\supset P(1)\supset P(2)$.

The notion of \textit{euphotic representations} was introduced by Jakob-Yun \cite[\S 2.2]{JY}. 
We consider euphotic representations of the $t^{-1}$-adic group $G(F_\infty)$ defined by $P$ and a triple $\mu=(\phi,\psi,\rho)$:
\begin{itemize}
\item $\phi$ is a linear function on the vector space $V=P(1)/P(2)$. We require that $\phi$ is semistable (in the sense that its orbit under the action of $L=P/P(1)$ is closed) and that there is only one invariant polynomial of degree $d$ taking non-vanishing value on $\phi$. In \ref{sss:generic functional}, \ref{sss:toric pairs}, we give an explicit formula for $\phi$ as a generic linear function associated to certain Levi subgroup $M$, whose root system $\Phi_M$ is irreducible and Coxeter number equals to $d$.

\item $\psi:k\to \overline{\mathbb{Q}}_{\ell}^{\times}$ is a non-trivial additive character. 

\item Let $L_{\phi}$ be the connected stablizer of $\phi$ under the action of $L$, and $B_\phi$ a Borel of $L_\phi$ with quotient torus $T_\phi$. 
	Then $\rho$ is a character $\rho:T_\phi(k)\to \overline{\mathbb{Q}}_{\ell}^{\times}$.  
\end{itemize}

We set $J=B_\phi(k) P(1)$ and the above data defines a character $\mu=(\rho,\psi\phi):J\to \overline{\mathbb{Q}}_{\ell}^{\times}$. 
	
Let $B_0$ be a Borel subgroup of $L$ containing $B_\phi$, then $I=B_0P(1)$ is an Iwahori subgroup. Let $I^{\opp}\subset G(F_0)$ be the Iwahori subgroup opposite to $I$. We take a character $\delta: T(k)\to \overline{\mathbb{Q}}_\ell^{\times}$, viewed as a character of $I^{\opp}$ via quotient by $I^\opp(1)$. 
The data $(I^{\opp},\delta),(J,\mu)$ are called \textit{hypergeometric automorphic data}, see Definition \ref{d:hyp auto data}.

We consider the vector space of compact supported $\overline{\mathbb{Q}}_{\ell}$-valued functions: 
\[
\mathscr{C}(\delta,\mu)= \textnormal{Fun} \biggl(G(F)  \backslash G(\mathbb{A}_F)/\prod_{x\neq 0,\infty} G(\mathscr{O}_{x}),\  \overline{\mathbb{Q}}_{\ell}\biggr)^{(I^{\opp},\delta),(J,\mu)}.
\]
Assume that the Coxeter number $h_G$ of $G$ is less than $p$. Our first main result is the following: 

\begin{thm}[Theorem \ref{t:main rigid}, Proposition \ref{p:eigenvalue}] \label{t:intro main thm}Assume the characters $\delta$, $\rho$ are in general position as in Definition \ref{d:general position}.

\textnormal{(i)} The hypergeometric automorphic data is rigid in the sense of \cite{YunCDM}.
 
\textnormal{(ii)} The space $\mathscr{C}(\delta,\mu)$ has finite dimension and is decomposed into a direct sum of one-dimensional subspace parametrized by the set $Z_{G,\phi}^*$ of characters of $Z_{G,\phi}=T_\phi\cap Z_G$. 
\end{thm}

For each character $\sigma$ of $Z_{G,\phi}$, the machinery of \cite{YunCDM} allows us to produce an $\ell$-adic $\hG$-local system $\cE_{\hG}(\delta,\mu)_{\sigma}$ on $\mathbb{G}_{m,k}$. It is defined as the Hecke eigenvalue of the Hecke eigensheaf $\cA(\delta,\mu)_{\sigma}$, on the moduli of $G$-bundles on $\bP^1$ with certain level structures, associated to the function of $\mathscr{C}(\delta,\mu)$ indexed by $\sigma$. 
When $\sigma=1$ is the trivial character, we set $\cE_{\hG}(\delta,\mu)=\cE_{\hG}(\delta,\mu)_1$ for simplicity. 

There is a parallel story when $\ell$-adic local systems on varieties over $k$ are replaced by connections on algebraic vector bundles on varieties over $\mathbb{C}$. Similarly we choose a generic linear function $\phi: V\to \bC$ and linear functions $\rho:\Lie(T_\phi)\to \bC$ and $\delta:\Lie(T)\to \bC$. 
We replace the Artin-Schreier sheaf associated to $\psi$ by the exponential $\cD$-module $\bC\langle x,\partial_x\rangle/(\partial_x-1)$ on $\bA^1_\bC$. If $(\delta,\mu)$ forms a rigid hypergeometric automorphic data, then the same construction as in the $\ell$-adic setting allows us to define a Hecke eigensheaf $\cA^{\dR}(\delta,\mu)_{\sigma}$ and its Hecke eigenvalue $\cE_{\hG}^{\dR}(\delta,\mu)_{\sigma}$, a de Rham $\hG$-local system on $\mathbb{G}_{m,\mathbb{C}}$.

\subsection{Comparisons with previous works} 
When $d$ is equal to the Coxeter number $h_G$ of $G$, the above result is due to Heinloth-Ng\^o-Yun \cite{HNY}. 
When $G=\GL_n$, the above result is obtained in \cite{KY}. 
In these works, the Hecke eigenvalues $\cE_{\hG}(\delta,\mu)$ are well-understood: 
\begin{itemize}
\item When $d=h_G$, the $\ell$-adic Hecke eigenvalues generalize the classical Kloosterman sheaf introduced by Deligne \cite{Del77}. 
The parallel story for de Rham local systems is studied in the work of Frenkel-Gross \cite{FG} and of Zhu \cite{Zhu}. 
In a recent work of the second author with Zhu \cite{XuZhu}, they unified previous works and show that generalized Kloosterman sheaves for classical groups come from certain hypergeometric local systems. 
		
\item When $G=\GL_n$, the Hecke eigenvalues $\cE_{\GL_n}(\delta,\mu)$ are the $\ell$-adic hypergeometric local systems defined by characters $\delta$ and $\mu=(\phi,\psi,\rho)$ \cite[Proposition 25]{KY}. 
In this situation, they can treat the case $d=1$ and tame hypergeometric local systems as well. 
\end{itemize}

\subsection{Study of $\cE^{\dR}_{\hG}(0,\mu)$ via quantization of Hitchin's integrable system} \label{ss:EdRquantization}
The second part of the paper is devoted to the study of the Hecke eigenvalue $\cE^{\dR}_{\hG}(0,\mu)$ in the de Rham setting, where we take $\delta,\sigma$ to be trivial characters. 
We follow Beilinson-Drinfeld's approach \cite{BD} and a variant due to Zhu \cite{Zhu}, where they constructed the ``Galois-to-automorphic'' direction of the geometric Langlands correspondence over $\mathbb{C}$ for certain opers. 

We keep the above notations and assume moreover that $G$ is simply connected (and $\hG$ is therefore of adjoint type). 
Let $\cg$ denote the Lie algebra of $\hG$. 
Suppose $d>\frac{h_G}{2}$. We consider the space $\Op_{\cg}(\bGm)_{(0,\varpi(0)),(\infty,\frac{1}{d})}$ of $\hG$-opers on $\bGm$ that
\begin{itemize}
\item have a regular singularity at $0$ with residue equals to $0$;
		
\item have a possibly irregular singularity of maximal formal slope $\le 1/d$ at $\infty$. 
\end{itemize}

There is a simple description of these opers. 
Let $n$ be the rank of $G$, $\cg=\fn^-\oplus \ft \oplus \fn^+$ the Cartan decomposition and $p_{-1}\in \fn^-$ a principal nilpotent element. 
Consider the unique principal $\mathfrak{sl}_2$-triple $\{p_{-1},2\check{\rho},p_1\}$ and the $\ad_{p_1}$-invariant subspace $\fn^{p_1}\subset \fn$. 
The adjoint action of $\check{\rho}$ gives rise to a grading on $\fn^{p_1}$. 
Let $d_1\le d_2 \le\dots \le d_n$ be the fundamental degrees of $G$ and $p_1,\dots,p_n$ be a homogeneous basis of $\fn^{p_1}$ with $\deg(p_i)=d_i-1$.  
If $t$ denotes a coordinate of $\bGm$ with a single pole at $\infty$, then an oper of $\Op_{\cg}(\bGm)_{(0,\varpi(0)),(\infty,\frac{1}{d})}$ can be uniquely expressed as the following $\cg$-connection on the trivial $\hG$-bundle over $\bGm$:
\begin{equation} \label{eq:oper intro}
\nabla=\partial_t + p_{-1}t^{-1}+\sum_{d_i\ge d} \lambda_i p_i,\quad \lambda_i \in \mathbb{C}.
\end{equation}
In particular, $\Op_{\cg}(\bGm)_{(0,\varpi(0)),(\infty,\frac{1}{d})}$ is isomorphic to an affine space. 

The character $\mu$ in the automorphic data factors through an abelian quotient $\overline{J}$ of $J$ (see Proposition \ref{l:J}) and is parametrized by the dual of its Lie algebra $\overline{\fj}^*$. 
The conditions that $\phi$ being generic (Definition \ref{d:phi generic}) and $\rho$ being in the general position from $\delta=0$ (Definition \ref{d:general position}) cut out a generic open subset $U$ of $\overline{\fj}^*$ parametrizing characters $\mu$ such that $\{(I^{\opp},0),(J,\mu)\}$ forms a rigid hypergeometric automorphic data.
Our second main result can be summarized as follows: 

\begin{thm}[Theorem \ref{th:oper Hecke eigensheaf}] \label{t:intro Hecke eigenvalues}
Suppose $d>h_G/2$ in the case of type A, C \footnote{We restrict ourself to the case $d>h_G/2$ for the simplicity of the statement.} and $d=h_G$ or $h_G-2$ in the case of type B, D. 
There exists a canonical dominant map $\pi: \bar{\fj}^*\to \Op_{\cg}(\bGm)_{(0,\varpi(0)),(\infty,\frac{1}{d})}$ such that, 
if $\mu\in U$, 
then the underlying connection of the $\hG$-oper $\pi(\mu)$ is isomorphic to the de Rham $\hG$-local system $\cE^{\dR}_{\hG}(0,\mu)$. 
\end{thm}

\subsubsection{Strategy of the proof} \label{sss:intro proof quantization}
We prove the above result following Zhu's strategy \cite{Zhu}, where he treats the case $d=h_G$ (for a simple simply-connected almost simple group $G$). 
Let $\cG$ be the group scheme over $\bP^1$, defined by \eqref{eq:group cG}
\[
\cG|_{\bP^1-\{0,\infty\}} =G\times(\bP^1-\{0,\infty\}),\qquad \cG(\cO_0)=I^{\opp},\qquad \cG(\cO_\infty)=\ker(J\to \overline{J}). 
\]
We study the global Hitchin map $H^{cl}$ from the cotangent bundle $T^*\Bun_\cG$ of the moduli stack of $\cG$-bundles to the Hitchin base over $\bGm$ and its local variant $h^{cl}$ at $\infty$. 
Then we determine the schematic image $\Hitch(\bP^1)_\cG$ of $H^{cl}$, which is quantized by the space $\Op_{\cg}(\bP^1)_{(0,\varpi(0)),(\infty,\frac{1}{d})}$. 
We also quantize the global and local Hitchin maps $H^{cl}$ and $h^{cl}$ (Propositions \ref{p:j^+ image} and \ref{p:global Hitchin diagram}).
Compared to Zhu's setting, our proof involves more complicated computation of images of Hichin maps, and the action of Segal-Sugawara operators on certain induced module of affine Kac-Moody algebra. 

These quantization results allow us to apply a variant of the machinery of Beilinson-Drinfeld, dealing with suitable level structures, due to Zhu \cite[Corollary 9]{Zhu} to produce the automorphic $\cD$-modules on $\Bun_{\cG}$, with opers of $\Op_{\cg}(\bP^1)_{(0,\varpi(0)),(\infty,\frac{1}{d})}$ as their Hecke eigenvalues. 
Finally, we conclude the theorem by comparing these automorphic $\cD$-modules with $\cA^{\dR}(0,\mu)$ constructed by hypergeometric automorphic data \eqref{intro ss: def of hyp auto data}. 

\subsubsection{Functoriality of opers spaces} \label{sss:intro functoriality}
In view of equation \eqref{eq:oper intro}, we obtain the functorial relationship between $\Op_{\cg}(\bGm)_{(0,\varpi(0)),(\infty,\frac{1}{d})}$ via pushout by $\spp_{2n}\to\Sl_{2n}$, $\so_{2n+1}\to\Sl_{2n+1}$ when $d>n$, and $\so_{2n+1}\to \so_{2n+2}$ when $d>n+1$ (see Proposition \ref{p:oper functoriality} for details). 
This suggests that there should be certain functoriality between eigenvalues $\cE^{\dR}_{\hG}(0,\mu)$ for different groups via the same pushout. 
In the following, we conjecture an explicit description of Hecke eigenvalues $\cE^{\dR}_{\hG}(0,\mu)$, which implies the functoriality between $\cE^{\dR}_{\hG}(0,\mu)$.

\subsection{Hypergeometric differential equation and $\cE_{\hG}^{\dR}(0,\mu)$}
Opers \eqref{eq:oper intro} are closely related to hypergeometric differential equations \eqref{eq:hyp}. 
Let $G$ be one of the classical groups $\GL_n$, $\SO_{2n+1}$, $\Sp_{2n}$, $\SO_{2n}$, $\Std$ the standard representation of $G$ and $\{(I^{\opp},\delta),(J,\mu)\}$ a rigid hypergeometric automorphic data such that $(\delta,\rho)$.
In the de Rham setting, $\rho$ (resp. $\delta$) is defined as a linear function on $\Lie(T_\phi)$ (resp. $\Lie(T)$) (\S~\ref{intro ss: def of hyp auto data}). 
There exist decompositions of subtori (see \S~\ref{ss:classical}, \eqref{eq:T_phi})
\[
\Lie(T)\simeq \oplus_{i=1}^n \Lie(T_i)\quad \text{and} \quad
\Lie(T_\phi)\simeq \oplus_{j=1}^{n-m} \Lie(T_j)\oplus\Lie(Z)
\]
and we set $\delta_i=\delta(1_i)$, $\rho_j =\rho(1_j) \in \mathbb{C}$, where $1_i$ is a basis of $\Lie(T_i)\simeq\bC$. 
%We set $\underline{\delta}=(\delta_i)_{i=1}^n, \underline{\rho}=(\rho_j)_{j=1}^{n-m}$. 
When each $\delta_i=0$, the rigidity of hypergeometric automorphic data is satisfied by requiring each $\rho_j\notin \mathbb{Z}$. 

When $G=\GL_n$, we take $\phi$ as in \cite[\S 7.2.2]{KY}. 
A variant of \cite[Proposition 25]{KY} in the de Rham setting shows the following isomorphism: 
\[
\cE^{\dR}_{\GL_{n}}(\delta,\mu)\simeq \Hyp_1(\delta_1,\dots,\delta_n;\rho_1,\dots,\rho_{n-m}). 
\]
Inspired by this fact and \S~\ref{sss:intro functoriality}, we make the following conjecture on the description of the de Rham local system $\cE_{\hG}^{\dR}(0,\mu)(\Std)$ associated to the standard representation of $\hG$. 

\begin{conj} \label{c:conjecture opers}
For $\hG=\SO_{n},\Sp_{2n}$ and a rigid hypergeometric automorphic data $(\delta=0,\mu=\rho\times \phi)$ of $G$, there exists a $\lambda\in \mathbb{C}^{\times}$ (depending on $\phi$), such that the following isomorphisms hold:

\textnormal{(i)} For $\hG=\SO_{2n+1}$, we have 
$\cE^{\dR}_{\SO_{2n+1}}(0,\mu)(\Std) \simeq 
\Hyp_\lambda(\underline{0};\rho_1,\dots,\rho_{n-m},-\rho_1,\dots,-\rho_{n-m},1/2)$.

\textnormal{(ii)} For $\hG=\Sp_{2n}$, we have 
$\cE^{\dR}_{\Sp_{2n}}(0,\mu)(\Std)
\simeq 
\Hyp_\lambda(\underline{0};\rho_1,\dots,\rho_{n-m},-\rho_1,\dots,-\rho_{n-m}).$

\textnormal{(iii)} For $\hG=\SO_{2n+2}$ and $d=2m>n+1$, we have 
$\cE^{\dR}_{\SO_{2n+2}}(0,\mu)(\Std) \simeq
\cE^{\dR}_{\SO_{2n+1}}(0,\mu)(\Std)
\oplus (\mathscr{O}_{\mathbb{G}_m},d).$
\end{conj}

Note that above hypergeometric differential equations are equipped with an orthogonal or symplectic symmetric structure. 
When $d=h_G$ (i.e. $n=m$), the above result was obtained by Zhu \cite{Zhu} (c.f. \cite[\S 5.2]{XuZhu} for details). 

As discussed in \S~\ref{ss:EdRquantization}, Conjecture \ref{c:conjecture opers} implies functorial relationships between Hecke eigenvalues $\cE_{\hG}^{\dR}(0,\mu)$ of different groups via pushout $\Sp_{2n} \to \SL_{2n}$, $\SO_{2n+1}\to \SL_{2n+1}$ and $\SO_{2n+1}\to \SO_{2n+2}$ respectively. 

Moreover, one can deduce the $p$-adic variants (and therefore the $\ell$-adic variants) of Conjecture \ref{c:conjecture opers} and the functorial relationships from their de Rham versions, by investigating the Frobenius structure on these de Rham local systems as in \cite{XuZhu}. 

We obtain some results towards to this conjecture:

\begin{prop}\label{p:opers hyp}
Under the assumption of Theorem \ref{t:intro Hecke eigenvalues}, for $\hG=\Sp_{2n}$ (resp. $\SO_{2n+1}$), $\cE_{\hG}^{\dR}(0,\mu)(\Std)$ is isomorphic to a hypergeometric differential equation.	
\end{prop}

Above statement follows from Theorem \ref{th:oper Hecke eigensheaf} and Corollary \ref{c:opers hyp}.

\subsection{Structure of the article.}
Sections \ref{s:hyp data} and \ref{s:eigen via quant} contains the main results and constructions of this article, while their proofs are postponed to sections \ref{s:toric}-\ref{s:opers proof}. 
In section \ref{s:hyp data}, we first briefly review the notion of rigid automorphic data following Yun \cite{YunCDM}. Then we define the hypergeometric automorphic data using a family of Levi subgroups, that we call \textit{admissible Levi}. 
We also specify the general position condition for the hypergeometric automorphic data to be rigid and give an explicit description of the relevant orbits (Theorem \ref{t:rigid}), which refines Theorem \ref{t:intro main thm}. 
The proofs of these theorems are given in \S~\ref{s:rigid}, based on the study of admissible Levi subgroups in \S~\ref{s:toric}. 
In section \ref{s:eigen via quant}, we study the de Rham local systems $\cE^{\dR}_{\hG}(0,\mu)$ produced by hypergeometric automorphic data and prove Theorem \ref{t:intro Hecke eigenvalues}, Proposition \ref{p:opers hyp}. 
Our proofs consist of two steps (see \S~\ref{sss:intro proof quantization}): we first study the associated local and global Hitchin maps in \S~\ref{s:pfs II}; we quantize these maps to establish the geometric Langlands correspondence for opers of $\Op_{\cg}(\bP^1)_{\cG}$, and compare $\cE^{\dR}_{\hG}(0,\mu)$, $\Op_{\cg}(\bP^1)_{\cG}$ and hypergeometric local systems in \S~\ref{s:opers proof}.

\textbf{Acknowledgement.} 
We would like to thank Alexander Molev and Xinwen Zhu for valuable discussions.  M. K. is supported by an Australian Research Council Discovery Grant. 
D. X. is partially supported by National Natural Science Foundation of China Grant No. 11688101.

\subsection{Notations}\label{ss:classical} We first fix some notations for the root systems in classical types. 

\subsubsection{Type A: $\GL_{n+1}$} 
We take a maximal torus $T=\{t=\diag(t_1,t_2,...,t_{n+1})|t_i\neq 0\}$ and define subtori $T_i=\{(1,...,1,t_i,1,...,1)|t_i\neq 0\}$, $1\leq i\leq n+1$. A basis of $X^*(T)$ is $\chi_i:t\mapsto t_i$, $1\leq i\leq n+1$. 
Let $\lambda_i:t_i\mapsto(1,...,1,t_i,1,...,1)$, where $t_i$ sits at $i$-th diagonal entry, $1\leq i\leq n+1$. The roots, positive roots, simple roots, half sum of positive coroots are
\begin{align*}
\Phi_G&=\{\chi_i-\chi_j|1\leq i\neq j\leq n+1 \}; \quad &
\Delta_G&=\{\alpha_i=\chi_i-\chi_{i+1}|1\leq i\leq n \};\\
\Phi_G^+&=\{\chi_i-\chi_j|1\leq i<j\leq n+1 \}; \quad &
\check{\rho}_G&=\frac{1}{2}\sum_{i=1}^{n+1}(n-2i+2)\lambda_i.
\end{align*}

\subsubsection{Type B: $\SO_{2n+1}$}
Let $J$ be the anti-diagonal matrix with $J_{ij}=(-1)^i\delta_{i,2n+2-j}$ and 
\[
\SO_{2n+1}=\{ A\in \SL_{2n+1} \mid AJA^T=J\}.
\]
We take a maximal torus $T=\{t=\mathrm{diag}(t_1,t_2,...,t_n,t_1^{-1},t_2^{-1},...,t_n^{-1},1)| t_i\neq0 \}$ and define subtori $T_i=\{(1,...,1,t_i,1,...,1,t_{i+n}^{-1}, 1,...,1)|t_i\neq 0\}$, $1\leq i\leq n$. A basis of characters $X^*(T)$ is $\chi_i:t\mapsto t_i$, $1\leq i\leq n$. A basis of cocharacters $X_*(T)$ is $\lambda_i:t_i\mapsto(1,...,t_i,...,t_i^{-1},...,1)$, where $t_i$ and $t_i^{-1}$ are $i$-th and $i+n$-th diagonal entries, $1\leq i\leq n$. The roots, positive roots, simple roots and half sum of positive coroots are
\begin{align*}
\Phi_G&=\{\pm\chi_i|1\leq i\leq n) \}\cup\{\pm\chi_i\pm\chi_j)|1\leq i<j\leq n \};\\
\Phi_G^+&=\{\chi_i-\chi_{i+k}=\sum_{j=i}^{i+k-1}\alpha_j|1\leq i<i+k\leq n\}\cup\{\chi_i=\sum_{j=i}^{n-1}\alpha_j+\alpha_n|1\leq i\leq n \};\\
&\qquad\cup\{\chi_i+\chi_{i+k}=\sum_{j=i}^{i+k-1}\alpha_j+2\sum_{j=i+k}^{n-1}\alpha_j+2\alpha_n|1\leq i<i+k\leq n \};\\
\Delta_G&=\{\alpha_i=\chi_i-\chi_{i+1}|1\leq i\leq n-1 \}\cup\{\alpha_n=\chi_n \}; \quad
\check{\rho}_G=\sum_{i=1}^n(n-i+1)\lambda_i.
\end{align*}

\subsubsection{Type C: $\Sp_{2n}$}
Let $J$ be the anti-diagonal matrix with $J_{ij}=(-1)^i\delta_{i,2n+1-j}$ and 
\[
\Sp_{2n}=\{ A\in \SL_{2n} \mid AJA^T=J\},
\]
We take a maximal torus $T=\{t=\mathrm{diag}(t_1,t_2,...,t_n,t_n^{-1},t_{n-1}^{-1},...,t_1^{-1})| t_i\neq0 \}$ and define subtori $T_i=\{(1,...,1,t_i,1,...,1,t_{i}^{-1}, 1,...,1)|t_i\neq 0\}$, $1\leq i\leq n$. 
A basis of characters $X^*(T)$ is $\chi_i:t\mapsto t_i$, $1\leq i\leq n$. A basis of cocharacters $X_*(T)$ is $\lambda_i:t_i\mapsto(1,...,t_i,...,t_i^{-1},...,1)$, where $t_i$ and $t_i^{-1}$ are $i$-th and $(2n+1-i)$-th diagonal entries, $1\leq i\leq n$. The roots, positive roots, simple roots and half sum of positive coroots are 
\begin{align*}
\Phi_G&=\{\pm(\chi_i-\chi_j)|1\leq i<j\leq n) \}\cup\{\pm(\chi_i+\chi_j)|1\leq i\leq j\leq n \};\\
\Phi_G^+&=\{\chi_i-\chi_{i+k}=\sum_{j=i}^{i+k-1}\alpha_j|1\leq i<i+k\leq n\}\cup\{\chi_i+\chi_{i+k}=\sum_{j=i}^{i+k-1}\alpha_j+2\sum_{j=i+k}^{n-1}\alpha_j+\alpha_n|1\leq i\leq i+k\leq n \};\\
\Delta_G&=\{\alpha_i=\chi_i-\chi_{i+1}|1\leq i\leq n-1 \}\cup\{\alpha_n=2\chi_n \}; \quad
\check{\rho}_G=\frac{1}{2}\sum_{i=1}^n(2(n-i)+1)\lambda_i.
\end{align*}

\subsubsection{Type D: $\SO_{2n}$}
Let $J$ be the anti-diagonal matrix with $J_{ij}=(-1)^{\max\{i,j\}}\delta_{i,2n+1-j}$ and 
\[
\SO_{2n}=\{A\in \SL_{2n}, AJA^T=J\},	
\]
We take a maximal torus $T=\{t=\mathrm{diag}(t_1,t_2,...,t_n,t_n^{-1},t_{n-1}^{-1},...,t_1^{-1})| t_i\neq0 \}$ and define subtori $T_i=\{(1,...,1,t_i,1,...,1,t_{i}^{-1}, 1,...,1)|t_i\neq 0\}$, $1\leq i\leq n$. A basis of characters $X^*(T)$ is $\chi_i:t\mapsto t_i$, $1\leq i\leq n$. A basis of cocharacters $X_*(T)$ is $\lambda_i:t_i\mapsto(1,...,t_i,...,t_i^{-1},...,1)$, where $t_i$ and $t_i^{-1}$ are $i$-th and $(2n+1-i)$-th diagonal entries, $1\leq i\leq n$. The roots, positive roots, simple roots and half sum of positive coroots are
\begin{align*}
\Phi_G&=\{\pm\chi_i\pm\chi_j|1\leq i<j\leq n \};\\
\Phi_G^+&=\{\chi_i-\chi_{i+k}=\sum_{j=i}^{i+k-1}\alpha_j|1\leq i<i+k\leq n\}\\
&\qquad\cup\{\chi_i+\chi_{i+k}=\sum_{j=i}^{i+k-1}\alpha_j+2\sum_{j=i+k}^{n-2}\alpha_j+\alpha_{n-1}+\alpha_n|1\leq i<i+k\leq n-1 \}\\
&\qquad\cup\{\chi_i+\chi_n=\sum_{j=i}^{n-2}\alpha_j+\alpha_n|1\leq i<n \};\\
\Delta_G&=\{\alpha_i=\chi_i-\chi_{i+1}|1\leq i\leq n-1 \}\cup\{\alpha_n=\chi_{n-1}+\chi_n \}; \quad
\check{\rho}_G=\sum_{i=1}^{n-1}(n-i)\lambda_i.
\end{align*}

\section{Hypergeometric automorphic data}\label{s:hyp data}
Let $k$ be a finite field of characteristic $p$, $\overline{k}$ an algebraic closure of $k$ and $\ell\neq p$ a prime.
Let $X$ be the projective line $\bP^1$ over $k$ and $F=k(X)$ its function field . For any point $x\in |X|$, let $F_x$ be the local field at $x$, $\cO_x$ the ring of integers and $t_x\in\cO_x$ a uniformizer. 
Let $G$ be a connected split reductive group over $k$, $h_G$ its Coxeter number. 
We assume $p>h_G$\footnote{One reason for this assumption is to define the grading given by $\check{\rho}/d$ on $\fg$; another reason is that we will later use the exponential of nilpotent elements, which is given by Springer isomorphism assuming $p>h_G$.}. 
We fix a Borel subgroup $B\subset G$ and a maximal torus $T\subset B$. 
Our goal is to define the notion of hypergeometric automorphic data for $G$ on $X$ and formulate our main theorems about rigidity of such automorphic data.

\subsection{Recollections on rigid automorphic data} 
In this subsection, we briefly recall a simplified version of Yun's notion of rigid automorphic data \cite{YunCDM} ignoring the central characters, which is sufficient for our purpose.
For a more detailed exposition of Yun's theory, cf. \cite[\S~6]{KY}. 
\begin{defe}\label{d:auto data}
Let $S$ be a finite subset of $|X|$. 
An \emph{automorphic data} (on $(X,S)$) is a collection of data $(K_x,\gamma_x)_{x\in S}$:
\begin{itemize}
\item For each $x\in S$, $K_x\subset L_xG$ is a pro-algebraic subgroup of loop group, contained in some parahoric subgroup with finite codimension;

\item For each $x\in S$, $\gamma_x$ is a (rank one) character sheaf on $K_x$, which is pullback from a finite dimensional quotient $K_x\tra L_x:=K_x/K_x^+$, where $K_x^+\subset K_x$ is a normal subgroup. 
\end{itemize}
\end{defe}

We set $K_S=\prod_{x\in S} K_x$ and $L_S=\prod_{x\in S} L_x$. 
Then $\gamma_S=\prod_{x\in S} \gamma_x$ defines a character sheaf on $L_S$. 

The set $S$ is called the set of ramification points of the automorphic data. For each $x\in X$, the pair $(K_x,\gamma_x)$ (and the subgroup $K_x^+$) is called the automorphic data at $x$. 

\begin{rem} In practice, when constructing rigid automorphic data, the groups $K_x$ need to be chosen carefully and are usually unique within each context. In contrast, there is some flexibility in the choice of the groups $K_x^+$. 
\end{rem}

\subsubsection{Associated moduli stack of bundles} \label{sss:def of BuncG}
To every automorphic data, one associates group schemes $\cG$ and $\cG'$ over $X$ satisfying
\begin{align*}
\cG|_{X-S}=G\times(X-S)&,\quad \cG|_{\cO_x}=K_x^+,\ \  \forall x\in S;\\
\cG'|_{X-S}=G\times(X-S)&,\quad \cG'|_{\cO_x}=K_x,\ \  \forall x\in S.
\end{align*}

Let $\Bun_\cG$ (resp. $\Bun_{\cG'}$) be the algebraic stack of $\cG$-bundles (resp. $\cG'$-bundles) on $X$. 
The action of $L_S$ on $\Bun_\cG$ makes $\Bun_\cG$ into a $L_S$-torsor over $\Bun_{\cG'}$. 
Therefore, for a $\cG$-bundle $\cE\in\Bun_\cG(\overline{k})$, its stabilizer 
$$
\Stab(\cE):=\{(l,\eta)|l\in L_S, \eta\in\mathrm{Isom}(\cE,l\cdot\cE) \}
$$
has a canonical morphism to $L_S$. 
Let
$$
p:\Bun_\cG\rightarrow\Bun_{\cG'}
$$
be the quotient by $L_S$. For $\cE'=p(\cE)$, we have
$\Aut_{\cG'}(\cE')\simeq\Stab(\cE)$. 
The restriction of an automorphism to formal disks around $x\in S$ induces a morphism 
$$
\Aut_{\cG'}(\cE')\rightarrow\prod_{x\in S}\Aut(\cE'|_{\cO_x})\simeq K_S.
$$ 
The morphism $\Stab(\cE)\to L_S$ is compatible with the composition of above morphism and $K_S\twoheadrightarrow L_S$. 

\begin{defe}\label{d:relevant}
\textnormal{(i)} A $\cG$-bundle $\cE\in\Bun_\cG(\overline{k})$ is \emph{relevant} if the pullback of $\gamma_S$ along $\Stab(\cE)\rightarrow L_S$ is constant over the neutral component $\Stab(\cE)^\circ$.

\textnormal{(ii)} A $\cG'$-bundle $\cE'\in\Bun_{\cG'}(\overline{k})$ is \emph{relevant} if the pullback of $\gamma_S$ along $\Aut_{\cG'}(\cE')\to K_S \twoheadrightarrow L_S$ is constant over the neutral component $\Aut(\cE')^\circ$. 
\end{defe}

By above discussion, we see that $\cE$ is relevant if and only if $p(\cE)$ is relevant, and that $\cE$ is relevant if and only if every bundle in its $L_S$-orbit is relevant. Such an $L_S$-orbit is called \emph{relevant orbit}.

\subsubsection{Rigidity} Let $\hG$ be the dual group of $G$ over $\overline{\mathbb{Q}}_{\ell}$ with center $Z_{\hG}$. 
The Kottwitz morphism
$$
\kappa:\Bun_\cG\rightarrow X^*(Z_{\hG}) 
$$
exhibits $X^*(Z_{\check{G}})$ as the set of connected components of $\Bun_{\cG}$ (resp. $\Bun_{\cG'}$). 
For $\alpha\in X^*(Z_{\hG})$, we denote the associated component by $\Bun_\cG^\alpha$ (resp. $\Bun_{\cG'}^\alpha$).

\begin{defe}\label{d:rigid}
An automorphic data is called \emph{rigid} if on each component $\Bun_\cG^\alpha$ of $\Bun_\cG$, there is a unique relevant orbit. 
(Equivalently, there is a unique relevant element on each component $\Bun_{\cG'}^\alpha$.)
\end{defe}

The main result of \cite{YunCDM} states that in favourable situations, rigid automorphic data give rise to Hecke eigensheaves on $\Bun_{\cG}$. 
The corresponding Hecke eigenvalues are expected to be rigid $\hG$-local systems.

\subsection{Toric pairs} \label{ss:toric} 
Let $M$ be a standard Levi subgroup of $G$ such that its root system $\Phi_M$ is irreducible. The aim of this subsection is to define the notion of \textit{toric pair} $(G,M)$. This notion will be used to define  hypergeometric automorphic data in the next subsection. We assume that $M\neq T$ since the case  $M=T$ behaves differently, cf. \S \ref{sss:M=T}.

\subsubsection{Generic functionals of Levi subalgebras}\label{sss:generic functional}
For a root $\alpha\in\Phi_G$, we denote its height by $\Ht(\alpha)=\langle\check{\rho}_G,\alpha\rangle$. Note that since $M$ is a standard Levi, $\langle\check{\rho}_G,\alpha\rangle=\langle\check{\rho}_M,\alpha\rangle$ for $\alpha\in\Phi_M$. Denote the Coxeter number of $M$ by $d:=h_M$. This is the same $d$ as in \S\ref{intro ss: def of hyp auto data}. Consider the Coxeter grading of $\fm$ defined by $\check{\rho}_M/d$: 
\[
\fm=\bigoplus_{i\in \bZ/d \bZ} \fm_i, \qquad  \qquad 
\fm_i:=\bigoplus_{\alpha \in \Phi_M, \,\, \Ht(\alpha) \equiv i ~\textnormal{mod}~ d} 
\fm_\alpha.
\]
We recall a definition from \cite{HNY}: 
\begin{defe}\label{d:phi generic}
An element $\phi\in \fm_1^*$ is called \emph{generic} if its restriction to each root subspace of $\fm_1$ is non-trivial. 
\end{defe}

A generic functional on $\fm_1$ can be written as
$ \lambda_{-\theta} E^*_{-\theta_M} + \sum_{\alpha \in \Delta_M} \lambda_\alpha E^*_\alpha$, 
where $E^*_\alpha$ is a basis of $\fg_\alpha^*$ and $\lambda_\alpha, \lambda_{-\theta} \in k^\times$. Such a functional is $T$-conjugate to 
\begin{equation}\label{eq:generic phi}
	\phi = \lambda E^*_{-\theta_M} + \sum_{\alpha \in \Delta_M} E^*_\alpha,\quad \textnormal{for}~ \lambda \in k^\times. 
\end{equation}

\subsubsection{Grading on $\fg$} \label{sss:gradingG}
Consider the grading on $\fg$ defined by 
 $\check{\rho}_G/d$:
\begin{equation}
\fg=\bigoplus_{i\in \bZ/d \bZ} \fg_i, \qquad  \qquad 
\fg_i:=\bigoplus_{\alpha \in \Phi_G, \,\, \Ht(\alpha) \equiv i ~\textnormal{mod}~ d} 
\fg_\alpha.
\end{equation}
	
Let $G_0$ be the connected subgroup of $G$ with Lie algebra $\fg_0$.  Let $B_0:=B\cap G_0$ and $U_0:=U\cap G_0$. Then $G_0$ is a reductive group containing $T$ and $B_0$ is a Borel subgroup of $G_0$ with Levi decomposition $B_0=T U_0$.  
The adjoint action of $G$ on $\fg$ restricts to an action of $G_0$ on $\fg_1$, called the \emph{Vinberg representation}. The pair $(G_0, \fg_1)$ is known as a \emph{Vinberg $\theta$-group} \cite{Vinberg}.

\subsubsection{} \label{sss:stabiliser} 
Let $\phi \in \fm_1^*$ be a generic functional. We have a canonical $T$-equivariant projection $\fg_1\ra \fm_1$. Thus, we can think of $\phi$ as an element of $\fg_1^*$. 	
Let $L_\phi$, $B_\phi$, $U_\phi$, and $T_\phi$ be the connected stabilizer of $\phi$ in $G_0$, $B_0$, $U_0$, and $T$, respectively. As $\phi$ is semisimple, $L_\phi$ is a connected reductive group with solvable subgroup $B_\phi$, unipotent subgroup $U_\phi$, and torus $T_\phi$. 
We set $Z_{G,\phi}=T_\phi\cap Z_G$. 

\begin{prop} \label{p:stabiliser}
	\textnormal{(i)} $T_\phi=\left(\bigcap_{\alpha \in \Delta_M} \ker(\alpha)\right)^\circ$.
	
	\textnormal{(ii)} $T_\phi$ is a maximal torus of $L_\phi$. 

	\textnormal{(iii)} $C_{G}(T_\phi)=M$. 

	\textnormal{(iv)} $B_\phi$ is a Borel subgroup of $L_\phi$ with Levi decomposition  $B_\phi=T_\phi U_\phi$. 
\end{prop} 
\begin{proof}
	
Assertion (i) follows immediately from the definition of $\phi$. 
For (ii-iii), let $S$ be a maximal torus of $L_\phi$ containing $T_\phi$. 
Then we have $S\subseteq C_G(T_{\phi})$. 
By (i) and \cite[Proof of Proposition 6.8]{ND}, we have $M=C_G(T_{\phi})$ and therefore $S\subseteq M$. 
We deduce that $S\subseteq G_0\cap M=T$ and $S\subseteq(T\cap L_\phi)^\circ=T_{\phi}$. We conclude that $T_\phi=S$ is a maximal torus of $L_\phi$.

We postpone the proof assertion (iv) to \S \ref{ss:stabiliser}.
\end{proof}

\subsubsection{Definition of toric pairs}  Observe that the torus $T_\phi$ acts on $U_0$ by conjugation, preserving the subgroup $U_\phi$. Thus, $T_\phi$ acts on the affine variety $U_0/U_\phi$.

\begin{defe} \label{d:toric}
We say that the pair $(G,M)$ is \emph{toric} if the action of $T_\phi$ on $U_0/U_\phi$ has an open dense orbit whose stabilizer is finite modulo center.
\end{defe} 

If $(G,M)$ is toric, then $U_0/U_\phi$ is a toric variety. In particular, the action of $T_\phi$ on $U_0/U_\phi$ has finitely many orbits. In this case, we have 
\begin{equation} \label{eq:numerical} 
\dim(T_\phi)=\dim(U_0/U_\phi)+\dim Z_G.
\end{equation}  
We regard this equality as a numerical requirement for a pair $(G,M)$ being toric.

\subsubsection{Classification of toric pairs}\label{sss:toric pairs} 
The classification of toric pairs is an interesting open problem. One readily verifies that the pair $(G,G)$ is toric for every connected reductive group $G$. If $G=\GL_n$, then one can show that $(G,M)$ is toric if and only if $M$ is conjugate to $\GL_{m}\times \GL_1\times \cdots \times \GL_1$. The if direction was established in \cite{KY} and the only if direction can be proved using the numerical requirement and a quiver description of the grading\footnote{We thank K. Jakob for explaining the argument to us.}. The main result of this section, stated below, concerns toric pairs for classical groups, i.e. the almost simple groups whose root systems are of type A, B, C, or D. 

Consider the following family of Levi subgroups:

\begin{defe} \label{d:admissible}
Let $G$ be a classical group and $\Delta_G=\{\alpha_1,\cdots, \alpha_n\}$ the standard set of simple roots in its usual ordering (see \S \ref{ss:classical}). A standard Levi subgroup  $M\subseteq G$ is said to be \emph{admissible} if $\Delta_M=\{\alpha_{n-m+1},...,\alpha_n\}$ where $1\leq m\leq n$ when $G$ is of type $A_n$, $B_n$,  or $C_n$, and $3\leq m\leq n$ when $G$ is of type $D_n$. 
\end{defe} 

The Coxeter number $d=h_M$ of $M$ and $m$ are related by
\begin{equation} \label{eq:d and m} 
	m= 
\begin{cases} 
d-1, & \textrm{Type A}; \\
d/2, & \textrm{Type B and C};\\
d/2+1, & \textrm{Type D}.
\end{cases}
\end{equation}

\begin{thm}\label{t:toric}
 Let $G$ be a classical group and $M\subseteq G$ an admissible Levi. Then $(G,M)$ is toric. 
\end{thm} 

This theorem is proved in \S \ref{s:toric} via a detailed case by case analysis.

\subsubsection{The case $M=T$} \label{sss:M=T}
Observe that every functional $\phi \in \ft^*$ is generic, since $T$ has no roots. An easy case by case analysis shows that the numerical requirement \eqref{eq:numerical} holds if and only if $G$ is of type $A_n$ and $L_\phi$ is of type $A_{n-1}$. In \cite{KY}, it is proved that in this case, the action of $T_\phi$ on $U_0/U_\phi$ has a dense open orbit with finite stabilizers.

\subsection{Hypergeometric representations}  
In this subsection, we define a new class of euphotic representations in the sense of \cite[\S 2.2]{JY}, which we call (wild) hypergeometric representations. These representations  are closely related to the notion of hypergeometric automorphic data, introduced in the next subsection.

\subsubsection{}  
Let $M$ be a Levi subgroup of $G$ such that $M\neq T$ and $(G,M)$ is toric (Definition \ref{d:toric}). In previous section, we study the grading of $\fg$ defined by $\check{\rho}_G/d$.
Now observe that  $\check{\rho}_G/d$ also defines a (principal) parahoric subgroup $P\subset L_\infty G$ equipped with a Moy-Prasad filtration 
\[
P\supset P(1)\supset P(2)\supset \cdots.
\]
We set $L=P/P(1)$ and $V=P(1)/P(2)$. Then $L$ is a reductive group over $k$ equipped with a canonical embedding into $P$. Thus, we can (and do) think of $L$ as a subgroup of $P$. Moreover, we have canonical isomorphisms 
\begin{equation} 
L\simeq G_0,\qquad V \simeq \fg_1. 
\end{equation}

\subsubsection{} 
Let $\phi\in \fg_1^*$ be a functional induced by a generic functional of $\fm^*_1$ via projection $\fg_1\ra\fm_1$, and $\psi: k\ra \bQlt$ a nontrivial additive character. Then we can view $\psi\phi$ as a character $P(1)\ra \bQlt$, which factors through $V$. 
We can extend this character to a larger subgroup of $P$ as follows. 
Recall that stabilizer $L_\phi \subset G_0$ and its Borel subgroup $B_\phi=T_\phi U_\phi$ studied in the previous section.
We take again notations of \S~\ref{sss:stabiliser}. 
Let $\rho:B_\phi\ra B_\phi/U_\phi\simeq T_\phi\ra\bQlt$ be a multiplicative character. Using the isomorphism $L\simeq G_0$, we can think of these as subgroups of $L$ (and therefore as subgroups of $P$). 
We denote by $J$ the following subgroup of $P$:
\begin{equation}
	J=B_\phi \ltimes P(1) \subset P.
	\label{eq:defJ}
\end{equation}
As $B_{\phi}$ stabilizes $\phi$, the following morphism is a character:
\begin{equation} \label{eq:defmu} 
\mu: J\ra \bQlt, \qquad \mu(bp) = \rho(b) \psi\phi(p),\qquad b\in B_\phi, p\in P(1). 
\end{equation}  

\begin{defe} 
The irreducible subquotients of the (compactly) induced representation $\mathrm{ind}_{J}^{G(K)}\mu$ are called \emph{(wild) hypergeometric representations}.  
\end{defe} 
By definition, hypergeometric representations are the euphotic representations associated to the functional $\phi$, cf. \cite[\S 2.2]{JY}. 
If $M=G$, then $J=P$ equals to the Iwahori subgroup $I$, and wild hypergeometric representations are the simple supercuspidal representations \cite{GR}. Note, however, that hypergeometric representations associated to toric pairs $(G,M)$ with $M\neq G$ are not expected to be supercuspidal.

\subsection{Hypergeometric automorphic data} 
In this subsection, we give the main definition of this paper introducing hypergeometric automorphic data and state our main theorem regarding these automorphic data. We continue with the notations of the previous subsection. Thus, $G$ is a connected split reductive group over a finite field $k$, $M\neq T$ is a Levi subgroup of $G$ such that $(G,M)$ is toric. After choosing a generic functional $\phi \in \fm_1^*$, a character $\rho: T_\phi \ra \bQlt$, and an additive character $\psi$ of the ground field, we obtain a character $\mu$ on $J=B_\phi \ltimes P(1)\subset L_\infty G$ \eqref{eq:defmu}. 

Let  $\cL_\mu=\cL_\rho\boxtimes\phi^*\cL_\psi$ be  the character sheaf on the proalgebraic group $J$ whose Frobenius trace equals to $\mu$. Let $I^\opp$ be the Iwahori subgroup of the loop group $L_0G$ at $0$, which is opposite to the Iwahori $I=B_0 P(1)\subset L_\infty G$. Choose a character $\delta:T\ra\bQlt$, regarded as a character of $I^\opp$ via $I^\opp\ra I^\opp/I^\opp(1)\simeq T$. Let $\cL_\delta$ be the character sheaf on $I^\opp$ whose Frobenius trace equals to $\delta$.  

\begin{defe}\label{d:hyp auto data}
The (wild) hypergeometric automorphic data on $(\bP^1,\{0,\infty\})$ associated to $(G, M, \mu, \delta)$ is defined by 
\begin{equation}\label{eq:auto data}
(K_x,\gamma_x):=
\begin{cases}
(I^\opp,\cL_\delta), \quad &x=0;\\
(J,\cL_\mu), \quad &x=\infty.
\end{cases}
\end{equation}
\end{defe} 

For the results of this section, we can take any suitable subgroups $K_x^+$ of $K_x$ in Definition \ref{d:auto data} for hypergeometric automorphic data, e.g. $K_0^+=I^\opp(1)$, $K_\infty^+=P(2)$. 
We will discuss certain particular choices of $K_x^+$ in \S~\ref{s:eigen via quant}.

\subsubsection{Numerical condition for rigidity} 
We denote by $\cG'$ the group scheme over $\bP^1$ defined by
\begin{equation}\label{eq:gp schemes}
\begin{split}
\cG'|_{\bP^1-\{0,\infty\}}&=G\times(\bP^1-\{0,\infty\}),\quad \cG'(\cO_0)=I^\opp,\quad \cG'(\cO_\infty)=J=B_\phi P(1). 
\end{split}
\end{equation}

\begin{lem}\label{r:numerical criteria}
For a toric pair $(G,M)$, we have $\dim\Bun_{\cG'}=0$. 	
\end{lem}
\begin{proof}
Let $[\fg(\cO_x):\Lie(K_x)]$ denotes relative dimension of $\cO_x$-lattices in $\fg(F_x)$. It suffices to show 
$$
\dim\Bun_{\cG'}=\dim\Bun_G-[\fg(\cO_0):\Lie(I^\opp)]-[\fg(\cO_{\infty}):\Lie(J)]=0.
$$
When $(K_0,K_{\infty})=(I^{\opp},I(1))$ (i.e. replacing $J$ by $I(1)$ in the above formula), it is proved in \cite[Corollary 1.3]{HNY}. 
By \eqref{eq:numerical}, the relative dimension between $J=B_\phi P(1)$ and $I(1)=Z_GU_0 P(1)$ is zero. Then the lemma follows. 
\end{proof}

If the stabilizers of relevant orbits in an automorphic data are finite, then above fact is an necessary condition for the data to be rigid, cf. \cite[Lemma 2.7.12.(2)]{YunCDM}. This suggests that hypergeometric automorphic data should be ``generically" rigid. In the following subsection, we make this intuition precise.

\subsection{Rigidity} 
Let $(G,M)$ be an admissible Levi as in Definition \ref{d:admissible}. Recall from Theorem \ref{t:toric} that such Levi is toric. We specify a condition on the characters $\delta$, $\rho$ in the associated hypergeometric automorphic data \eqref{eq:auto data} under which we prove that the data is rigid.

\subsubsection{General position between $\delta$ and $\rho$}
Note that since $U_0/U_\phi$ is a toric variety for the action of $T_\phi$, there are finitely many $T_\phi$-orbits on it. 
Define a finite set $\mathcal{S}$ consisting of subtori of $T_\phi$ by 
\begin{equation}\label{eq:S subtori}
\mathcal{S}:=\{(\mathrm{Stab}_{T_\phi}(x))^\circ\, | \, x\in T_\phi \backslash (U_0/U_\phi)\}. 
\end{equation} 

Let $W$ be the Weyl group of $G$. We denote by $\delta^w$ the conjugation of $\delta:T\ra\bQlt$ by a Weyl group element $w\in W$.

\begin{defe}\label{d:general position}
We say $\delta$ and $\rho$ are \emph{in general position} if $\rho|_{S}\neq \delta^w|_S$ for all $S\in \mathcal{S}$ and $w\in W$. 
\end{defe} 

The above condition is the same as requiring the character sheaves are different: $\cL_{\rho}|_S\neq \cL_{\delta^{w}}|_S$.

For $G=\GL_{n+1}, \Sp_{2n}, \SO_{2n+1}, \SO_{2n}$, we can make above condition more explicit. Let $n$ be the rank of the derived group of $G$, $Z_G$ the center of $G$. We take the standard splitting of the diagonal $T$ into diagonal entries $T=\prod_i T_i$ (see \S\ref{ss:classical}). 

Recall from Definition \ref{d:admissible} that $M$ has $m$ simple roots. We obtain from Proposition \ref{p:stabiliser}(ii) that
\begin{equation}\label{eq:T_phi}
T_\phi=
\begin{cases}
Z\prod_{i=1}^{n-m}T_i, \quad G=\GL_{n+1};\\
\prod_{i=1}^{n-m} T_i, \quad G=\SO_{2n+1}, \Sp_{2n}, \SO_{2n}.
\end{cases}
\end{equation}
In this case, the condition that $\delta$, $\rho$ are in general position can be simplified to the following:
\begin{equation}\label{eq:general position by T_i}
\rho|_{T_i}\neq \delta^w|_{T_i},\quad\quad \forall w\in W,\ 1\leq i\leq n-m.
\end{equation}
For $G=\GL_{n+1}$, above becomes the assumption $\chi_i\neq\rho_j$ in the statement of \cite[Theorem 10]{KY}.

\begin{thm} \label{t:main rigid} 
Suppose $M$ is an admissible Levi of $G$, and $\delta$, $\rho$ are in general position. Then the hypergeometric automorphic data is rigid. 
\end{thm} 

For toric pairs $(G,G)$, this theorem is proved in \cite{HNY}. When $G=\GL_n$, this theorem is obtained in \cite{KY}. We first give a refinement of the above theorem in the following, which we will prove later.

\subsubsection{The relevant orbits} 
We give an explicit description of the relevant orbits of $\Bun_{\cG'}$ \eqref{eq:gp schemes} for the hypergeometric automorphic data in Theorem \ref{t:main rigid}. Let $\tW=N_{G(F_\infty)}(T(F_\infty))/T(\cO_\infty)$ be the Iwahori-Weyl group and $\Omega\subset \tW$  the normalizer of the Iwahori $I$. Let $s$ be a coordinate around $\infty$ and 
$I^-:=I^\opp\cap G(k[s,s^{-1}])$. Using Birkhoff decomposition, we have (see \eqref{eq:Bun_cG' decomp} for details)  
\begin{equation}\label{eq:Bun_cG' stratification}
\Bun_{\cG'}(\overline{k})=\bigsqcup_{\tw\in\tW}I^-\backslash I^-\tw U_0 J/J = \bigsqcup_{\tw\in\tW}\bigcup_{[u]\in T_\phi\backslash (U_0/U_\phi)} I^-\backslash I^-\tw u J/J.
\end{equation} 
Let $\ru\in U_0/U_\phi$ be an element whose $T_\phi$-orbit is dense and open. 
We deduce Theorem \ref{t:main rigid} from the following theorem, which is proved in \S \ref{s:rigid}.  

\begin{thm}\label{t:rigid}
Let $\cG'$ be the group scheme associated to a hypergeometric automorphic data for an admissible Levi $M$ \eqref{sss:def of BuncG}.   

\textnormal{(i)} The automorphism group of a $\cG'$-bundle is finite if and only if it is associated to the double coset $I^-\backslash I^-\tw \ru J/J$, for some $\tw\in \Omega$. 

\textnormal{(ii)} Suppose $\delta$ and $\rho$ are in general position. Then 
the relevant bundles are
$
\{I^-\tw\ru J\, | \, \tw \in \Omega\}$. 
Moreover, the automorphism group of a relevant bundle is isomorphic to $Z_{G,\phi}$. 
\end{thm}

\begin{rem}
We expect that above rigidity theorem holds for any toric pairs $(G,M)$. However, we do not know if there exist toric pairs other than the admissible Levi subgroups. 
\end{rem}

\subsection{Hecke eigensheaves and eigenvalues}
We now discuss the eigensheaves and eigenvalues arising from these rigid data. Let $D(\delta,\mu)$ be the derived category of $\overline{\mathbb{Q}}_{\ell}$-complexes on $\Bun_{\mathcal{G}}$ with $(T\times J/P(2), \mathcal{L}_\delta\boxtimes \mathcal{L}_{\mu})$-equivariant structures, and  $P(\delta,\mu)$ the full subcategory of perverse sheaves.
By applying the argument of \cite[\S 4.1-4.2]{JY}, we deduce from Theorem \ref{t:rigid} the following result about Hecke eigenvalues associated to rigid hypergeometric automorphic data. 

\begin{prop} \label{p:eigenvalue} 
We keep the assumption of Theorem \ref{t:rigid}. 
Let $Z_{G,\phi}^*$ be the set of characters of $Z_{G,\phi}=T_\phi\cap Z_G$. Then there exists a decomposition
\[
P(\delta,\mu)=\bigoplus_{\sigma\in Z_{G,\phi}^*} P_\sigma,
\]
such that each $P_{\sigma}$ contains a unique simple perverse sheaf $\mathcal{A}(\delta,\mu)_{\sigma}$ on $\Bun_{\cG}$, which is a Hecke eigensheaf with eigenvalue $\cE_{\hG}(\delta,\mu)_\sigma$.  Moreover, $\cE_{\hG}(\delta,\mu)_\sigma$ is a semisimple $\hG$-local system on $\bGm$. 
\end{prop}

When $\sigma$ is the trivial character $1$, we ignore the subscript $1$ from $\cE(\delta,\mu)_1, \cA(\delta,\mu)_1$. 
When $G=\PGL_{n}$, $\Sp_{2n}$, or $\SO_n$, the group $Z_{G,\phi}$ is trivial, and we obtain a unique Hecke eigenvalue $\cE(\delta,\mu)=\cE(\delta,\mu)_{1}$. 
However, when $G=\SL_{n}$, the group $Z_{G,\phi}$ equals the center $Z_G$ (c.f. \S \ref{ss:subtori}).

\subsubsection{Trivial functoriality} \label{sss:trivial functoriality}
Let $H\to G$ be a central isogeny of classical groups. 
Then it induces an isomorphism $\fm_{H,1}\xrightarrow{\sim} \fm_{G,1}$, and we can abusively use the notation $\phi$ to denote the ``same'' generic linear functional on these spaces. 
Let $\delta$, $\rho$, and $\sigma$ be characters of $T_G$, $T_{G,\phi}$, and $\sigma\in Z_{G,\phi}^*$, respectively.   
We denote the resulting characters of $T_H$, $T_{H,\phi}$,  $Z_{H,\phi}$  by the same notations. 
Then the $\check{H}$-eigenvalue $\cE_{\check{H}}(\delta,\mu)_{\sigma}$ is the pushout of $\cE_{\hG}(\delta,\mu)_{\sigma}$ along $\hG \to \check{H}$. For instance, if $G$ is adjoint, then the unique Hecke eigenvalue $\cE_{\check{G}}(\delta,\mu)$ maps to $\cE_{\check{H}}(\delta,\mu)$.

\subsubsection{Relationship to Katz's hypergeometric sheaves} In \cite{KY}, a variant of Theorem \ref{t:rigid} is proved for $G=\GL_n$, and it is shown that the resulting Hecke eigenvalues are precisely the $\ell$-adic hypergeometric local systems introduced in \cite{KatzBook}. 
We expect that the Hecke eigenvalues $\cE_{\hG}(\delta, \mu)$ are related to those $\ell$-adic hypergeometric local systems that admit an $\hG$-symmetry (c.f. Conjecture \ref{c:conjecture opers} for a precise statement in the de Rham setting). 
We will study this relationship in the de Rham setting (see subsection below) in the next section.

\subsubsection{De Rham variant}\label{r:dR setting}
As explained in \cite[\S 2.6]{HNY}, there is a variant of the above construction using $\cD$-modules in stead of $\ell$-adic sheaves. 
We take the base field to be $k=\bC$. 
We choose a generic $\bC$-linear function $\phi: \fm_1\to \bC$ and $\bC$-linear functions $\rho:\Lie(T_\phi)\to \bC$ and $\delta:\Lie(T)\to \bC$. 
We replace the Artin-Scherier sheaf by the exponential $\cD$-module $\bC\langle x,\partial_x\rangle/(\partial_x-1)$ on $\bA^1_\bC$, and the Kummer sheaves $\mathcal{L}_{\rho}, \mathcal{L}_{\delta}$ by $\cD$-modules
\begin{displaymath}
\cL_{\rho}^{\dR}=\boxtimes_{i=1}^{n-m} \bigl(\bC \langle x,x^{-1},\partial_x\rangle/(x\partial_x -\rho(\mathbf{1}_{i})) \bigr),\quad 
\cL_{\delta}^{\dR}=\boxtimes_{i=1}^{n} \bigl(\bC \langle x,x^{-1},\partial_x\rangle/(x\partial_x -\delta(\mathbf{1}_{i})) \bigr),
\end{displaymath}
via isomorphisms $T\simeq \mathbb{G}_{m,\bC}^n$, $T_\phi\simeq \mathbb{G}_{m,\bC}^{n-m}$.
Moreover, these $\cD$-modules are independent of the choices of decompositions. 

If $(\delta,\mu)$ are in general positions (defined in a similar way as in Definition \ref{d:general position} by $\cL_{\rho}^{\dR}$ and $\cL_{\delta}^{\dR}$), then the same construction in the de Rham setting allows us to define Hecke eigensheaves $\cA^{\dR}(\delta,\mu)_\sigma$ on $\Bun_{\cG}$, parametrized by $\sigma\in Z_{G,\phi}^*$ and associated Hecke eigenvalues:
\begin{equation}\label{eq:dR Hecke eigenvalue}
\cE^{\dR}_{\hG}(\delta,\mu)_\sigma: \Rep(\hG) \to \Conn(\bGm). 
\end{equation}

\section{Hecke eigenvalues via a quantization of Hitchin map}\label{s:eigen via quant}
Let $G$ be a simply-connected classical group of rank $n$ over the field of complex numbers $\mathbb{C}$, and $M$ an admissible Levi-subgroup of $G$ \eqref{d:admissible}. 
We keep the notations from the previous section. 
When $G$ is of type B or D, we will make certain restrictions on $d=h_M$. 

The goal of this section is to study the Hecke eigenvalues $\cE^{\dR}_{\hG}(0,\mu)$ in the de Rham setting, where we take $\delta,\sigma$ to be trivial characters.

We first define a group scheme $\cG$ over $\bP^1$ with generic fiber $G$ associated to the hypergeometric data (\S~\ref{ss:J+}), then study the associated local and global Hitchin maps (\S~\ref{ss:local Hitchin}-\ref{ss:global Hitchin}). 
The Hitchin base is quantized by a space of opers $\Op_{\cg}(\bP^1)_{\cG}$ (\S~\ref{ss:local opers}). 
Then we use a quantization of the global Hitchin map to construct the Hecke eigensheaves, with eigenvalues in $\Op_{\cg}(\bP^1)_{\cG}$, which allows us to compare these opers with $\cE^{\dR}_{\hG}(0,\mu)$ (\S~\ref{ss:globalopers} Theorem \ref{th:oper Hecke eigensheaf}). 
Finally, we compare opers of $\Op_{\cg}(\bP^1)_{\cG}$ with hypergeometric local systems (\S~\ref{ss:hyp opers} Corollary \ref{c:opers hyp}).

\subsection{An abelian quotient of $J$} \label{ss:J+}
In subsections \ref{ss:J+}-\ref{ss:local opers}
, we suppose moreover that $d>\frac{2}{3}n$ (resp. $d>n$) when $G$ is of type B (resp. type D). 
To start with, we specify some subgroups $K_x^+$ in the definition of hypergeometric automorphic data \eqref{d:auto data}. 
We set
\[
\fk:=\bigoplus_{\alpha\in\Phi(\fg_1)-\Phi(\fm_1)}\fg_\alpha\subseteq \fg_1.
\]
The functional $\phi\in \fg_1^*$ (which is extended by zero from a functional on $\fm_1\subset \fg_1$) vanishes on $\fk$. Recall that we have an identification of vector spaces $\fg_1\simeq V=P(1)/P(2)$. Recall \eqref{eq:defJ} that
\[
J/P(2) \simeq B_\phi \ltimes V. 
\]
Let $J^+$ be the preimage of $U_\phi \cdot\fk\subset B_\phi\ltimes V$ in $J$.

We prove the following proposition in \S \ref{ss:J}. 

\begin{prop} \label{l:J}
\textnormal{(i)} $J^+\subseteq J$ is a normal subgroup, $J^+/P(2)\simeq U_\phi\ltimes\fk$. 

\textnormal{(ii)} The quotient $\overline{J}:=J/J^+$ is isomorphic to $T_\phi\times \fm_1$ and is commutative.

\textnormal{(iii)} The character $\mu$ \eqref{eq:defmu} is trivial on $J^+$ and factors through the quotient $\overline{J}$. 
\end{prop} 

Hence, the character sheaf $\cL_\mu$ is trivial on $J^+$. On the other hand, the trivial character $\cL_0$ at $0\in |\bP^1|$ \eqref{d:hyp auto data}. 
In summary, we can define subgroups $K_x^+$ of $K_x$ \eqref{d:auto data}: 
\begin{equation} 
K_x^+:=
\begin{cases}
I^\opp, \quad &x=0;\\
J^+, \quad &x=\infty.
\end{cases}
\end{equation}

Let $\cG$ be the group scheme over $X=\bP^1$ associated to these level structures; i.e. 
\begin{equation} \label{eq:group cG} 
\cG|_{\bP^1-\{0,\infty\}} =G\times(\bP^1-\{0,\infty\}),\qquad \cG(\cO_0)=I^\opp,\qquad \cG(\cO_\infty)=J^+. 
\end{equation} 
Let $\Bun_\cG$ denote the moduli stack of $\cG$-torsors on $X$. The following proposition is proved in \S \ref{ss:good stack}: 
\begin{prop}\label{p:good stack}
The stack $\Bun_\cG$ is good; i.e. $\dim(T^*\Bun_\cG)=2\dim(\Bun_\cG)$ \cite[\S1.1.1]{BD}.
\end{prop}

Next, let $\omega_{\Bun_{\cG}}$ be the canonical line bundle on $\Bun_{\cG}$ and set 
\[
\omega_{\Bun_{\cG}}^\sharp:= \omega_{\Bun_\cG} \otimes \det(\mathrm{Lie}(\cG))^{-1}. 
\]
\begin{lem} There exists a square root $\omega_{\Bun_{\cG}}^{1/2}$ of $\omega_{\Bun_{\cG}}^{\sharp}$.
\end{lem} 
\begin{proof} As $J^+$ is pro-unipotent, this follows from  \cite[\S 4.1]{Zhu14} and \cite[Proposition 6]{Zhu}.
\end{proof}

\subsection{Local Hitchin map} \label{ss:local Hitchin}
Let $\fc=\fg/\!\!/G$ be the GIT quotient of $\fg$ by adjoint action, and similarly $\fc^*=\fg^*/\!\!/G$. 
We denote by $\chi:\fg\to \fc$ the Chevalley morphism. 
For an $N\times N$ matrix $X$, we write
\[
\det(\lambda I+X)=\lambda^N+c_1(X)\lambda^{N-1}+\dots+c_N(X),
\]
and denote by $p(X)$ the Pfaffian of $X$ if $X$ moreover belongs to $\so_{2n}$. 
In practice, there exists an isomorphism $\fc\simeq \bA^n$ and the map $\chi:\fg \to \fc\simeq \bA^n$ sends a matrix $X\in \fg$ to the coefficients $\{c_{d_i}(X)\}_{i=1}^n$ (and the Pfaffian $p(X)$ in the type D case).

Let $s=t^{-1}$ be a coordinate around $\infty\in\bP^1$, $K=\bC(\!(s)\!)$ and $\cO=\bC[\![s]\!]$. 
We denote by $D^{\times}$ (resp. $D$) the punctured disc $\Spec(K)$ (resp. the disc $\Spec (\cO)$). 
We denote by $\Hit(D^{\times})$ the local Hitchin base $\fc^*\times^{\bGm}\omega_K^\times$ over $D^{\times}$, which is an ind-scheme over $\bC$. 
In the following, we fix a basis $\td s$ of $\omega_K$ and take an isomorphism $\omega_K\simeq K\td s$.
The Chevalley morphism $\chi:\fg^*\to \fc^*$ induces the \textit{local Hitchin map} of ind-schemes \cite[\S 4.2]{Zhu}:
\[
h^{cl}: \fg^*\otimes\omega_K  \ra \fc^*\times^{\bGm}\omega_K^\times\simeq\bigoplus_{i=1}^n\omega_K^{d_i}.  
\]
Via the isomorphism $\fg\simeq\fg^*$ given by Killing form, it sends $X\td s$ to $(c_{d_i}(X)(\td s)^{d_i})_{i=1}^n$ in the case of type A, B, C, and additionally to $p(X)(\td s)^n$ in type D.

\subsubsection{} 
Given an $\cO$-lattice $\mathfrak{l}\subset\fg_K$, we define another $\cO$-lattice of $\fg^*\otimes\omega_K$ by 
\begin{equation} 
\fl^\perp=\{\xi \in\fg^*\otimes\omega_K \,|\,\Res_K(\xi,X)=0,\ \forall\  X\in\fl \},
\end{equation} 
where $\Res_K:\fg^*\otimes\omega_K\times\fg\ra\bC$ denotes the residue pairing, i.e. coefficient of $\xi(X)$ at $\frac{\td s}{s}$. We denote by $\cA_{\fl}$ the Zariski closure of image $h^{cl}(\fl^\perp)$ in $\bigoplus_{i}\omega_K^{d_i}$. 
\begin{exam}\label{Hitchin map I}
(i) For the hyperspecial subalgebra $\fg(\cO)$, we set $\Hit_{\fg(\cO)}:=\bigoplus_{i=1}^n \cO (ds)^{d_i}$. 
The image of unramified local Hitchin map is surjective, i.e. we have $\cA_{\fg(\cO)}=\Hit_{\fg(\cO)}$. 
This space is isomorphic to $\fc^*\times^{\bGm}\omega^{\times}_{\cO}$ and is also denoted by $\Hit(D)$. 

(ii) Let $\fI$ denote the Iwahori subalgebra and set $\Hit_{\fI}:=\bigoplus_{i=1}^n s^{-d_i+1} \cO(\td s)^{d_i}$. 
Then an easy explicit computation shows that $\cA_\fI=\Hit_\fI$.
We refer the reader to \cite{BK, BKV} for more examples. 
\end{exam} 

\subsubsection{} 
Now let us consider the hypergeometric data. 
We denote by $\fj^+\subseteq \fj \subseteq \fg(\cO)$ the Lie algebras of $J^+\subset J$ and by $\overline{\fj}$ the Lie algebra of $\overline{J}$. Our main aim is to understand $\cA_{\fj^+}$. We set 
\begin{equation} \label{eq:def of Z}
\Hit_{\fj^+}:=\bigoplus_{d_i<d} s^{-d_i} \cO(\td s)^{d_i} \oplus\bigoplus_{d_i\geq d} s^{-d_{i}-1} \cO(\td s)^{d_i}, 
\quad \textnormal{and} \quad 
Z:=\bigoplus_{d_i\geq d}s^{-d_i-1} \bC(\td s)^{d_i}.  
\end{equation}

We can now formulate an analogue of \cite[Proposition 14]{Zhu}, which will be proved in \S \ref{ss:ProofHitchin}:
\begin{prop}\label{p:j^+ image}
\textnormal{(i)} 
We have $\cA_{\fj^+}=\Hit_{\fj^+}$ in $\Hit(D^{\times})$. 

\textnormal{(ii)} We have a commutative diagram
\begin{equation}\label{eq:local Hitchin diagram}
\begin{tikzcd}
(\fj^{+})^{\perp} \arrow{r} \arrow{d}{h^{cl}} &  \overline{\fj}^* \arrow{d}\\
\Hit_{\fj^+} \arrow{r}{p} &Z
\end{tikzcd}
\end{equation}
where $p$ is the natural projection and every morphism in this diagram is dominant. 
\end{prop}

\subsection{Global Hitchin map}\label{ss:global Hitchin} 
We denote the global Hitchin base by $\Hit(\bGm)$:
\[
\Hit(\bGm):= \Gamma(\bGm,\fc^*\times^{\bGm}\omega_{\bGm}^\times),
\]
and the global Hitchin map associated to $\cG$ (cf. \cite[\S 2.2]{BD} and \cite{Zhu}) by
\begin{equation} \label{eq:global Hitchin}
H^{cl}:T^*\Bun_{\cG}\ra\Hit(\bGm). 
\end{equation}

Let $D_0^\times$ and $D_\infty^\times$ be the punctured formal disks around $0$ and $\infty$, respectively and set
\begin{equation}\label{eq:HitchinBase}
\Hit(\bP^1)_{\cG}:=\Hit(\bGm)\times_{\Hit(D_\infty^{\times})} \Hit_{\fj^+}\times_{\Hit(D_0^{\times})} \Hit_{\mathfrak{i}}.
\end{equation}
The following is a global version of the previous proposition (and an analogue of \cite[Lemma19]{Zhu}), which will be proved in \S \ref{ss:pf global Hitchin}:
\begin{prop} \label{p:global Hitchin diagram}
\textnormal{(i)} 
The Zariski closure of the image of $H^{cl}$ is $\Hit(\bP^1)_{\cG}$. 

\textnormal{(ii)} We have a commutative diagram
\begin{equation}\label{eq:global Hitchin diagram}
\begin{tikzcd}
T^*\Bun_\cG \arrow{r} \arrow{d}{H^{cl}} &{\overline{\fj}}^* \arrow{d}\\
\Hit(\bP^1)_\cG \arrow{r}{\sim} &Z 
\end{tikzcd}
\end{equation}
Every morphism in this diagram is dominant and the bottom arrow is an isomorphism.
\end{prop}

Our next goal is to quantize the above propositions using the center of affine Kac--Moody algebra at the critical level, following the framework of \cite{BD, Zhu}.

\subsection{Local quantization}
\subsubsection{}
Let $\hg=\fg(K)\bigoplus \bC\cdot\mathbf{1}$ be the affine Kac-Moody algebra\footnote{We ignore the derivation usually present in affine Kac--Moody algebras.} at the  critical level associated to $\fg$. Let $\Ug$ its completed universal enveloping algebra at the critical level. 
Let \[
 \Vac:= \Ind^{\hg}_{\fg(\cO)+\bC\bone}\bC
 \]
 be the universal affine vertex algebra. 
As a vector space, $\Vac$ is isomorphic to the universal enveloping algebra
$U(\fg_-)$, where $\fg_-=\fg\otimes s^{-1}\mathbb{C}[ [s^{-1}]]$. We have a canonical filtration on $\Ug$ wich induces a filtration on $U(\fg_-)$. The associated graded algebra $\gr(\Vac)$ identifies with $\Sym (\fg_-) = \Fun \fg^*(\cO)$.

Let $\fz$ and $\fZ$ denote the center of the vertex algebra $\Vac$ and the associative algebra $\Ug$ respectively. 
The algebras $\fZ$ and $\fz$ are equipped with canonical filtrations coming from that of $\hat{U}(\fg)$. 
By \cite[Theorem 3.7.8]{BD}, there exists canonical isomorphisms 
\begin{equation} 
\gr\fz\simeq\Fun\Hit(D),\qquad\quad \gr\fZ\simeq\Fun\Hit(D^\times).
\end{equation} 
Thus, $\mathrm{Spec}(\fz)$ and $\mathrm{Spec}(\fZ)$ quantize $\Hit(D)$ and $\Hit(D^\times)$, respectively. Given $S\in \fZ$, we write $\overline{S}$ for its image in $\gr\fZ$. 
Recall that we have an isomorphism 
\[
\Hit(D^\times)\simeq\bigoplus_{i=1}^n\omega^{d_i}_K.
\] 
Let $h_{ij}$ denote the coefficient of $s^{-j-1}(\td s)^{d_i}$ in $\bC s^{-j-1}(\td s)^{d_i}$. 
Then $\{h_{ij}\,|\, 1\leq i\leq n, j\in\bZ\}$ provide a set of topological generators of $\Fun\Hit(D^\times)$.

\subsubsection{}
A complete set of \textit{Segal-Sugawara vectors} is a set of elements
\[
S_1,S_2,\dots,S_n\in\fz,\quad n=\rank \fg,
\]
such that images of $S_1,\dots,S_n$ in $\gr \fz$ coincide with the images of some algebraically independent generators of the algebra of invariants $\Sym(\fg)^{\fg}$ under the embedding $\Sym(\fg)\to \Sym(\fg_-)$, cf. \cite{Frenkel,MolevBook}.	

Let $\{S_1,\dots,S_n\}$ be a complete set of Segal--Sugawara vectors. A fundamental theorem of Feigin and Frenkel \cite{FF, Frenkel} states that the center $\fz$ is isomorphic to the polynomial algebra generated by $S_i$ and their derivations $S_{i,j}, j\ge 1$. 
Moreover, the coefficients $S_{i,j}, j\in\bZ$, obtained under state-field correspondence of the vertex algebra $\Vac$, are called \textit{Segal--Sugawara operators} and form a set of topological generators of the center $\fZ$. Note that a complete set of Segal--Sugawara vectors (and therefore operators) is by no means unique.

\subsubsection{} 
Now we apply the above considerations to our hypergeometric setting. Define the vacuum module associated to $\fj^+$ by 
\[
\Vac_{\fj^+}:=\Ind^{\hg}_{\fj^+\oplus \bC\bone}(\bC)=\Ug/\Ug(\fj^+\oplus \bC\bone).
\]
We have a canonical isomorphism $\End\Vac_{\fj^+} \simeq \Vac_{\fj^+}^{J^+}$. 
Let $\fZ_{\fj^+}$ denote the image of $\fZ$ in $\End\Vac_{\fj^+}$: 
\[
\fZ\tra\fZ_{\fj^+}\hra\End\Vac_{\fj^+} \simeq \Vac_{\fj^+}^{J^+}\subset\Vac_{\fj^+}.
\]
Since $\Vac_{\fj^+}=\Ind_{\fj\oplus \bC \bone}^{\hg}(U(\overline{\fj}))$, we obtain an injective homomorphism
\[
U(\overline{\fj})\simeq \End_{\fj}~ U(\fj/\fj^+)\subset\End\Vac_{\fj^+}.
\]
Taking the associated graded with respect to the filtration induced from $\Ug$, we have
$$
\gr U(\overline{\fj})\simeq\Fun(\overline{\fj})^*,\qquad \gr\Vac_{\fj^+}\simeq\Fun(\fj^{+})^{\perp}.
$$
By Proposition \ref{p:j^+ image}, we have an isomorphism
\begin{equation} \label{eq:gr Zj+}
\gr\fZ_{\fj^+}\simeq\Fun\Hit_{\fj^+}.\footnote{We give an alternative proof of this isomorphism in the proof of Lemma \ref{l:ker(fZ to fZ_j^+)}.}
\end{equation}
Thus we see that $\Vac_{\fj^+}$ and $\fZ_{\fj^+}$ quantize $(\fj^+)^\perp$ and $\Hit_{\fj^+}$, respectively.

\subsubsection{} \label{sss:SS Molev}
Now we use a particular set of Segal-Sugawara vectors to define a subalgebra $A\subset\fZ$. In the case of type A, we let $S_i$ be those constructed in \cite[Theorem 3.1]{MolevA}. 
In the case of type C, we let $S_i$ be those constructed in \cite[Theorem 4.4]{Yakimova}. As explained at end of \cite[\S2.3]{MolevBCDnew}, these Segal-Sugawara vectors coincide with the ones constructed earlier in \cite{MolevBCD}. 
Since these $S_i$ are constructed using characteristic polynomials, then symbols of associated Segal-Sugawara operators $\overline{S_{i,j}}$ would correspond to $h_{ij}$ under the isomorphism $\gr\fZ\simeq\Fun\Hit(D^\times)$, c.f. \cite[Chapter 6]{MolevBook}.

In the case of type B or D, we take a set of Segal-Sugawara vectors $\{S_i\}$ such that $\overline{S_{i,j}}$ correspond to $h_{ij}$ under the isomorphism $\gr\fZ\simeq\Fun\Hit(D^\times)$. 

Recall that $d=h_M$ is a fundamental degree \eqref{sss:generic functional}. We consider the following subalgebra of $\fZ$: 
\begin{equation} \label{eq:A} 
A:=\bC[S_{i,d_i}\,|\,d_i\geq d]\subset \fZ
\end{equation} 
Since $\overline{S_{i,d_i}}$ corresponds to $h_{i,d_i}$ in \eqref{eq:gr Zj+},
we have $\gr A\simeq\Fun Z$ \eqref{eq:def of Z}.
The associated graded of the composition $A\to \fZ\twoheadrightarrow \fZ_{\fj^+}$ is given by 
\[
\gr A\simeq \Fun Z\ra\gr \fZ_{\fj^+}\simeq \Fun\Hit_{\fj^+}, 
\]
and is injective by Proposition \ref{p:j^+ image}. Thus, $A\to \fZ_{\fj^+}$ is also injective. We can now state a quantum analogue of Proposition \ref{p:j^+ image}, which is proved in \S \ref{ss:pf local quant}.  

\begin{prop}\label{p:local quant}
Assume moreover that $d=h_G$ or $h_G-2$ when $G$ is of type B or D. 

\textnormal{(i)} The image of $A\to \fZ_{\fj^+}$ is contained in $U(\overline{\fj})$ and we have the following commutative diagram:
\begin{equation}\label{eq:local quant diagram}
\begin{tikzcd}
A \arrow[r,hook] \arrow[d,hook] &\fZ_{\fj^+} \arrow[d,hook]\\
U(\overline{\fj}) \arrow[r,hook] &\End(\Vac_{\fj^+}) \arrow[r,hook] &\Vac_{\fj^+}
\end{tikzcd}
\end{equation}

\textnormal{(ii)} The morphisms in above diagram are strictly compatible with filtrations.
Taking the associated graded, we obtain diagram \eqref{eq:local Hitchin diagram}. 
\end{prop}

\subsection{Quantizing the Hitchin base via opers}\label{ss:local opers}
\subsubsection{} 
\label{sss:SS general}
We consider another set of Segal-Sugawara operators defined via the Feigin-Frenkel isomorphism. Let $\Op_{\cg}(D^{\times})$ be the space of $\cg$-opers (aka $\check{G}$-opers) over the punctured disc $D^{\times}$. Then the Feigin-Frenkel isomorphism \cite{FF} can be rephrased as a canonical isomorphism
\begin{equation} 
\Fun\Op_{\cg}(D^{\times}) \simeq \fZ. 
\end{equation} 
This formulation has the advantage of being coordinate independent. 

Let $n$ be the rank of $\hG$, $\cg=\fn^-\oplus \ft \oplus \fn^+$ the Cartan decomposition and $p_{-1}\in \fn^-$ a principal nilpotent element. 
Consider the unique principal $\mathfrak{sl}_2$-triple $\{p_{-1},2\check{\rho},p_1\}$ and the $\ad_{p_1}$-invariant subspace $\fn^{p_1}\subset \fn$. 
The adjoint action of $\check{\rho}$ gives rise to a grading on $\fn^{p_1}=\bigoplus_{i=1}^n \fn^{p_1}_{d_i}$, where $d_1\le d_2 \le\dots \le d_n$ are the fundamental degrees of $\hG$. 
Let $p_i$ be a homogeneous basis of $\fn^{p_1}_{d_i}$ with $\deg(p_i)=d_i-1$. We also denote $\fn^{p_1}_{d_i}$ by $V_{\cg,d_i}$ to distinguish these spaces for different groups.  

If we choose a local coordinator $s$ of $D^{\times}$, then 
\[
\Op_{\cg}(D^\times)\simeq \{\nabla=\partial_s+p_{-1}+\sum_{i=1}^n v_i(s)p_i,\ v_i(s)\in\bC(\!(s)\!)\}, \quad v_i(s)=\sum_{j}v_{ij}s^{-j-1}. 
\]

Now let $S_{i,j}\in\fZ$ be the image of $v_{ij}$ via the above isomorphism. This gives a set of Segal-Sugawara operators, c.f. \cite[\S4.3.1., \S4.3.2.]{Frenkel}. Consider
$$
A_1:=\{S_{i,j}\,|\,j\geq d_i+\lfloor\frac{d_i}{d}\rfloor,\forall d_i\}, \quad A_2:=\{S_{i,j}\,|\,j\geq d_i+1,\forall d_i\}.
$$

The following lemma follows from the discussions in \S~\ref{l:S_ij vanish} and \S~\ref{sss:proof local quant}:
\begin{lem}\label{l:ker(fZ to fZ_j^+)}
\textnormal{(i)} For Segal-Sugawara operators in \S~\ref{sss:SS general}, we have $A_1\subset\ker(\fZ\tra\fZ_{\fj^+})$.

\textnormal{(ii)} In the case of type A or C with Segal-Sugawara operators considered in \S~\ref{sss:SS Molev}, we have $A_2\subset \ker(\fZ\tra\fZ_{\fj^+}) $.  
\end{lem}

\subsubsection{} 
Let $\Op_{\fj^+}(D^{\times}):=\Spec(\fZ_{\fj^+})$, which is a closed subscheme of $\Op_{\cg}(D^{\times})$.
Lemma \ref{l:ker(fZ to fZ_j^+)}(i) implies 
\begin{equation}\label{eq:Op_j^+}
\Op_{\fj^+}(D^{\times})\subset
\{\nabla=\partial_s+p_{-1}+\sum_{d_i<d}s^{-d_i}\bC[\![s]\!]p_i+\sum_{d_i\geq d}s^{-d_i-\lfloor d_i/d\rfloor}\bC[\![s]\!]p_i\}.
\end{equation}
In addition, when $d>\frac{h_G}{2}$, the description of this space simplifies:
\begin{equation}\label{eq:Op_j^+ d>h/2}
\Op_{\fj^+}(D^{\times})\simeq \{\nabla=\partial_s+p_{-1}+\sum_{d_i<d}s^{-d_i}\bC[\![s]\!]p_i+\sum_{d_i\geq d}s^{-d_i-1}\bC[\![s]\!]p_i\}.
\end{equation}
In this case, $\Op_{\fj^+}(D^{\times})$ equals to the space of $\cg$-opers  of slope $\le \frac{1}{d}$, cf. \cite[Proposition 3]{CK}.

\subsubsection{} 
Let $\Op_{\cg,\fI}(D^{\times})=\Spec(\fZ_{\fI})$ denotes the scheme classifying local $\cg$-opers on $D^{\times}$ with regular singularity and  principal unipotent monodromy (c.f. \cite[\S 9]{Frenkel}). 
We define a closed subscheme $\Op_{\cg}(\bP^1)_{\cG}$ of $\Op_{\cg}(\bGm)$ classifying $\cg$-opers on $\bP^1$ associated to $\cG$ as in \cite[\S2]{Zhu}. 
By \cite[Lemma 5]{Zhu}, we have an isomorphism:
$$
\Op_{\cg}(\bP^1)_{\cG}\simeq
\Op_{\cg}(\bGm)\times_{\Op_{\cg}(D_0^\times)}\Op_{\cg,\fI}(D_0^\times)\times_{\Op_{\cg}(D_\infty^\times)}\Op_{\cg,\fj^+}(D_\infty^\times).
$$
Let $t=s^{-1}$ be a coordinate around $0\in\bP^1$. 
We have the following explicit description of $\Op_{\cg}(\bP^1)_{\cG}$: 

The followin proposition is proved in \S \ref{ss:global oper}. 

\begin{prop}\label{l:global oper}
An oper of $\Op_{\cg}(\bP^1)_{\cG}$ can be written in terms of coordinates $\lambda_m,\dots,\lambda_n \in \bC$ as: 
\begin{equation}\label{eq:global oper}
\nabla=
\partial_t+t^{-1}p_{-1}+
\sum_{d_i\geq d}(\lambda_i+\sum_{j=1}^{\lfloor d_i/d\rfloor-1}t^j P_{ij})p_i,
\end{equation}
where $P_{ij}$ are polynomials in $\lambda_m,\dots,\lambda_n$ with fixed coefficients.
In particular, we obtain an (non-canonical) isomorphism  
\begin{equation} 
\Op_{\cg}(\bP^1)_{\cG}\simeq \bA^{n-m+1},\qquad m:=\sharp\{d_i~|~ d_i\le d\}.
\end{equation}  
\end{prop}

\subsection{Quantizing the Hitchin map and geometric Langlands correspondence} \label{ss:globalopers} 
In this subsection, we keep the assumption of Proposition \ref{p:local quant}. 
We first quantize the global Hitchin map 
\[
H^{cl}: T^*\Bun_{\cG}\ra\Hit(\bP^1)_\cG
\]
introduced in \S \ref{ss:global Hitchin}. We have already seen that $\Op_{\check{\fg}}(\bP^1)_\cG$ quantizes the base $\Hit(\bP^1)_\cG$. On the other hand, it is well-known that differential operators, or their twisted variants,  quantize the cotangent bundle. More precisely, let $\cD'_{\Bun_{\cG}}$ be the sheaf of $\omega_{\Bun_{\cG}}^{1/2}$-twisted differential operators on the smooth site $(\Bun_{\cG})_{sm}$. Let $D'=(\underline{\mathrm{End}} \cD'_{\Bun_{\cG}})^{op}$ be the sheaf of endomorphisms of $\cD'_{\Bun_{\cG}}$. Then we have a natural filtration on $\Gamma(\Bun_{\cG}, D')$ whose associated graded is the cotangent sheaf, cf. \cite[\S 1]{BD} for more details. By  \cite[(3.3)]{Zhu} (which relies heavily on \cite{BD}), there exists a canonical homomorphism
\[
H^q: \Fun \Op_{\cg}(\bP^1)_{\cG} \to \Gamma(\Bun_{\cG}, D'), 
\]
which is a quantization of the map $H^{cl}$. We are now in a position to quantize Proposition \ref{p:global Hitchin diagram}. 
	
\begin{prop}
We have a commutative diagram
\begin{equation}\label{eq:global quant diagram}
\begin{tikzcd}
A \arrow{r}{\sim} \arrow[d,hook] &\Fun\Op_{\cg}(\bP^1)_\cG \arrow{d}\\
U({\overline{\fj}}) \arrow{r} &\Gamma(\Bun_\cG,D')
\end{tikzcd}
\end{equation}
whose associated graded coincides with diagram \eqref{eq:global Hitchin diagram}. 
Moreover, the top arrow is an isomorphism. 
\end{prop}
\begin{proof}
The diagram is obtained by embedding global opers into local opers using diagram \eqref{eq:local quant diagram}. It follows from construction that taking the associated graded gives diagram \eqref{eq:global Hitchin diagram}. Thus, we conclude that the top arrow is an isomorphism.
\end{proof}

\subsubsection{Hecke eigensheaves associated to $\Op_{\cg}(\bP^1)_{\cG}$ and comparison with $\cA^{\dR}(0,\mu)$}

\begin{thm} \label{th:oper Hecke eigensheaf}
\textnormal{(i)} Let $\mu:U({\overline{\fj}})\to \bC$ be a closed point of ${\overline{\fj}}^{*}$ and $\nabla_{\cg,\mu}\in \Op_{\cg}(\bP^1)_{\cG}(\bC)$ the associated $\cg$-oper via \eqref{eq:global quant diagram}. 
Then $\nabla_{\cg,\mu}$ is the Hecke eigenvalue of the following Hecke eigensheaf on $\Bun_{\cG}$: 
\begin{equation}\label{eq:Aut_cE}
\Aut_{\nabla_{\cg,\mu}}=\omega_{\Bun_\cG}^{-1/2}\otimes(\cD'_{\Bun_\cG}\otimes_{U({\overline{\fj}}),\mu}\bC). 
\end{equation}

\textnormal{(ii)} Suppose $\delta=0$ and $\rho$ are in general position. 
Then $\Aut_{\nabla_{\cg,\mu}}$ is isomorphic to the associated Hecke eigensheaf $\cA^{\dR}(0,\mu)$. Thus the underlying connection of the $\hG$-oper $\nabla_{\cg,\mu}$ is isomorphic to the Hecke eigenvalue $\cE^{\dR}_{\hG}(0,\mu)$ \eqref{r:dR setting}.  
\end{thm}
\begin{proof}
(i) Recall that $\Bun_{\cG}\to \Bun_{\cG'}$ is an $\overline{J}$-torsor. 
Since $\Bun_{\cG'}$ is good of dimension zero, the argument as in \cite[Lemma 18]{Zhu} shows that the moment map 
\[
\pi: T^*\Bun_{\cG} \ra {\overline{\fj}}^*
\]
defined by the action of $\overline{J}$ on $\Bun_{\cG}$ is flat. 
Therefore, $\cD'_{\Bun_\cG}$ is flat over $U({\overline{\fj}})$. 
Moreover, since $\pi^{-1}(0)$ is Lagrangian, $\Aut_{\nabla_{\cg,\mu}}$ is holonomic and its singular support is contained in the nilpotent cone. Assertion (i) then follows from \cite[Corollary 9]{Zhu}. 

(ii) Since the homomorphism $U({\overline{\fj}})\to \Gamma(\Bun_{\cG},D')$ comes from the action of $\overline{J}$ on $\Bun_{\cG}$, the $\cD$-module $\Aut_{\nabla_{\cg,\mu}}$ is $(\overline{J},\mu)$-equivariant.
By Theorem \ref{t:rigid}(ii), the relevant orbit of $\Bun_{\cG}$ is isomorphic to $\overline{J}/Z_{G,\phi}$. Hence the line bundle $\omega_{\Bun_{\cG}}^{1/2}$ is trivialized on this relevant orbit.
Thus, the restriction of $\Aut_{\nabla_{\cg,\mu}}$ to this orbit is isomorphic to the character sheaf $\cL_{\mu}^{\dR}$ defined by $\mu$.
By \cite[Lemma 2.7.9]{YunCDM}, $\Aut_{\nabla_{\cg,\mu}}$ is a clean extension, establishing the claim. 
\end{proof}

\begin{rem}
By repeating the argument of \cite[\S 4.3]{XuZhu}, we can remove the simple connectedness assumption on $G$ in the above result.
\end{rem}

\begin{cor}
Suppose moreover that $G$ is a classical group group and keep the assumption of Theorem \ref{th:oper Hecke eigensheaf}(ii).  
Then the $\hG$-oper $\nabla_{\cg,\mu}$ is isomorphic to the Hecke eigenvalue $\cE^{\dR}_{\hG}(0,\mu)$.  
\end{cor}

\subsection{Functoriality of oper spaces}\label{sss:functoriality}
Let $d$ be a positive integer.
For an adjoint type almost simple group $\hG$, we denote by $\Op_{\cg}(\bGm)_{(0,\varpi(0)),(\infty,\frac{1}{d})}$ the space of $\cg$-opers on $\bGm$ that
\begin{itemize}
\item has a regular singularity at $0$ with residue equals to $0$;
		
\item has a possibly irregular singularity of maximal formal slope $\le 1/d$ at $\infty$. 
\end{itemize}

When $d>\frac{h_G}{2}$ \eqref{eq:d and m}, the oper \eqref{eq:global oper} simplifies to 
\begin{equation} \label{eq:global oper d}
\nabla_{\cg}(\lambda_m,\dots,\lambda_n)=\partial_t+t^{-1}p_{-1}+\sum_{d_i\geq d} \lambda_i p_i, \quad \lambda_i\in \bC. 
\end{equation}
In this case, the space $\Op_{\cg}(\bP^1)_{\cG}$ identifies with $\Op_{\cg}(\bGm)_{(0,\varpi(0)),(\infty,\frac{1}{d})}$ in view of the following lemma. 

\begin{lem} \label{l:oper le 1/d}
An oper of $\Op_{\cg}(\bGm)_{(0,\varpi(0)),(\infty,\frac{1}{d})}$ can be written written as 
\begin{equation} \label{eq:oper le 1/d}
\nabla_{\cg}(\lambda_i)=
\partial_t+t^{-1}p_{-1}+
\sum_{d_i\geq d}\lambda_i(t)p_i,\quad \lambda_i(t)=\sum_{j=0}^{\lfloor d_i/d\rfloor-1}t^j \lambda_{ij},\quad \lambda_{ij}\in\bC.
\end{equation}
\end{lem}
\begin{proof}
As $\nabla|_{D^\times_0}\in\Op_{\cg,\fI}$, an oper $\nabla\in\Op_{\cg}(\bGm)_{(0,\varpi(0)),(\infty,\frac{1}{d})}$ can be written as
$$
\nabla=\partial_t+t^{-1}p_{-1}+\sum_{i=1}^n \lambda_i(t)p_i,\quad \lambda_i(t)\in \bC[t].
$$
Changing the coordinate to $s=t^{-1}$, $\nabla$ becomes
$$
\nabla=\partial_s+s^{-1}p_{-1}+s^{-2}\sum_{i=1}^n \lambda_i(s^{-1})p_i.
$$
After applying gauge transformation by $s^{\rho}$ and $\exp(\frac{-p_1}{2t})$, the above connection is equivalent to
$$
\nabla=\partial_s+p_{-1}-s^{-2}\frac{p_1}{4}+\sum_{i=1}^n s^{-d_i-1}\lambda_i(s^{-1})p_i.
$$ 
Then we conclude the expression \eqref{eq:oper le 1/d} by \cite[Proposition 3]{CK}. 	
\end{proof}

We fix a basis $p_i$ of $V_{\Sl_{2n},i}$ (resp. $V_{\Sl_{2n+1},i}$) as in \cite[\S 4.2.4]{Frenkel}. 
Then, $p_{2i}$ is also a basis of $V_{\spp_{2n},2i}$ (resp. $V_{\so_{2n+1},2i}$) for $1\le i\le n$. 
Moreover, via the inclusion $\so_{2n+1}\to \so_{2n+2}$, one can identify $V_{\so_{2n+1},2i}$ with $V_{\so_{2n+2},2i}$ for $i>n$. 

Then, we obtain the following functoriality between oper spaces given by inclusion:
\begin{prop}\label{p:oper functoriality}
\textnormal{(i)} Via inclusion $\spp_{2n}\to \Sl_{2n}$, we have
\begin{eqnarray*}
\Op_{\spp_{2n}}(\bGm)_{(0,\varpi(0)),(\infty,\frac{1}{2m})} &\hookrightarrow&
\Op_{\Sl_{2n}}(\bGm)_{(0,\varpi(0)),(\infty,\frac{1}{2m})} \\
\nabla_{\spp_{2n}}(\lambda_{m},\lambda_{m+1},\dots,\lambda_{n}) &\mapsto &
\nabla_{\Sl_{2n}}(\lambda_{m},0,\lambda_{m+1},\dots,0,\lambda_{n}).
\end{eqnarray*}

\textnormal{(ii)} Via inclusion $\so_{2n+1}\to \Sl_{2n+1}$, we have
\begin{eqnarray*}
\Op_{\so_{2n+1}}(\bGm)_{(0,\varpi(0)),(\infty,\frac{1}{2m})} &\hookrightarrow&
\Op_{\Sl_{2n+1}}(\bGm)_{(0,\varpi(0)),(\infty,\frac{1}{2m})} \\
\nabla_{\so_{2n+1}}(\lambda_{m},\lambda_{m+1},\dots,\lambda_{n}) &\mapsto &
\nabla_{\Sl_{2n+1}}(\lambda_{m},0,\lambda_{m+1},\dots,\lambda_{n},0).
\end{eqnarray*}

\textnormal{(iii)} Assume $d=2m>n+1$. Via inclusion $\so_{2n+1}\to \so_{2n+2}$, we have
\begin{eqnarray*}
\Op_{\so_{2n+1}}(\bGm)_{(0,\varpi(0)),(\infty,\frac{1}{2m})} &\hookrightarrow&
\Op_{\so_{2n+2}}(\bGm)_{(0,\varpi(0)),(\infty,\frac{1}{2m})} \\
\nabla_{\so_{2n+1}}(\lambda_{m},\lambda_{m+1},\dots,\lambda_{n}) &\mapsto &
\nabla_{\so_{2n+2}}(\lambda_{m},\lambda_{m+1},\dots,\lambda_{n}).
\end{eqnarray*}
\end{prop}

\subsection{Hypergeometric differential equations and $\Op_{\cg}(\bP^1)_{\cG}$} \label{ss:hyp opers}
We present a description of opers in $\Op_{\Sl_n}(\bP^1)_{\cG}$ in terms of certain hypergeometric differential equations \eqref{eq:hyp}. 

When $\hG=\SL_n$, we have $d=m$. We consider the following hypergeometric equation of type $(n,n-m)$ defined by $\underline{\beta}\in \mathbb{C}^{n-m}$, $\lambda\in \mathbb{C}$, $\delta=t\frac{d}{dt}$:
\begin{equation}
	\Hyp_{\lambda}(0;\underline{\beta})=\delta^n-(-1)^{m}\lambda t \prod_{j=1}^{n-m}(\delta-\beta_j)=\delta^n+t( u_m\delta^{n-m}+\dots+u_n),
\end{equation}
where $u_m,\dots,u_n\in \mathbb{C}$ are expressed in terms of $\underline{\beta}, \lambda$.
The above equation defines a $\SL_n$-oper on the trivial $\SL_n$-bundle on $\mathbb{G}_m$. 

For $\hG=\Sp_{2n}$ (resp. $\hG=\SO_{2n+1}$) and type $(2n,2(n-m))$ (resp. $(2n+1,2(n-m)+1)$), we take $\underline{\beta}=(\beta_1,-\beta_1,\dots,\beta_{n-m},-\beta_{n-m})$ (resp. $(\beta_1,-\beta_1,\dots,\beta_{n-m},-\beta_{n-m},\frac{1}{2})$). 
Then $u_{i}=0$ if $i$ is odd (resp. $u_{2i+1}=2u_{2i}$) and the connection associated to the above equation has a $\hG$-symmetry. 

In the following, we consider differential equations parametrised by $u_m,\dots,u_n\in \mathbb{C}$:
\begin{equation}
\Hyp(u_m,\dots,u_n)=
\delta^n+t\biggl( u_m\delta^{n-m}+\dots+u_n\biggr). 
\label{eq:Hyp eq}
\end{equation}
We denote by $\Hyp^{(n,n-m)}$ the affine space $\Spec(\mathbb{C}[u_m,\dots,u_n])$ parametrizing coefficients of the above equation.
The following proposition is proved in \S~\ref{ss:pf Hyp opers}. 

\begin{prop} \label{p:Hyp opers}
\textnormal{(i)} There exists an invertible lower triangular matrix A of rank $n-m+1$ such that under the linear transformation 
\[(u_m,\dots,u_n)=(\lambda_m,\dots,\lambda_n)A,\]
$\Hyp(u_m,\dots,u_n)$ is isomorphic to the following $\SL_n$-oper on the trivial $\SL_n$-bundle over $\bGm$:
\begin{equation}\label{eq:Hyp oper form}
\nabla=
\partial_t+t^{-1}p_{-1}+
\sum_{d_i\geq d}(\lambda_i+\sum_{j=1}^{\lfloor d_i/d\rfloor-1}t^j P_{ij})p_i,
\end{equation}
where $P_{ij}$ are polynomials in $\lambda_m,\dots,\lambda_n$. 

\textnormal{(ii)} There exists an isomorphism 
\[
\psi: \Op_{\Sl_n}(\bP^1)_{\cG} \xrightarrow{\sim} \Hyp^{(n,n-m)},
\]
such that an oper $\nabla$ of $\Op_{\Sl_n}(\bP^1)_{\cG}(\mathbb{C})$ is isomorphic to the oper \eqref{eq:Hyp oper form} corresponding to $\psi(\nabla)$ in (i). 
\end{prop}

\begin{rem}
The above proposition allows us to identify opers of $\Op_{\Sl_n}(\bP^1)_{\cG}$ with opers produced by hypergeometric differential equations. 
In particular, we can determine the polynomial $P_{ij}$ in Proposition \ref{l:global oper} for $\cg=\Sl_n$. 
We call $\psi(\nabla)$ \textit{hypergeometric coordinate} of the $\Sl_n$-oper $\nabla$. 
\end{rem}

Combined with results in \S~\ref{sss:functoriality}, we conclude the following corollary. 

\begin{cor} \label{c:opers hyp}
Let $\hG$ be $\SO_{2n+1}$ and $d>\frac{h_G}{2}$ (resp. $\Sp_{2n}$ and $d=h_G$ or $h_G-2$) and $\nabla$ an oper of $\Op_{\cg}(\bGm)_{\cG}$ \eqref{eq:global oper d}. 
By considering $\nabla$ as a $\hG$-local system, the connection $\nabla(\Std)$ associated to the standard representation of $\hG$ is isomorphic to a hypergeometric differential equation. 
\end{cor}

In particular, this allows us to deduce Proposition \ref{p:opers hyp}.

\section{Toric pairs for classical groups: Proof of Theorem \ref{t:toric}} \label{s:toric}

Except in \S\ref{ss:wt sp decomp}, we assume $G$ is a classical group and $M$ is an admissible Levi subgroup of $G$ (Definition \ref{d:admissible}).

\subsection{Root system of $G_0$}\label{ss:G_0 simple roots}
We now study the decomposition of the root system of $G_0$ (\S~\ref{sss:gradingG}) into simple root systems. 
Recall that we have a grading of $\fg$ defined by $\check{\rho}_G/d$, $d=h_M$. The relationship between $d$ and $m=|\Delta_M|$ is given in \eqref{eq:d and m}. Let $\Phi(G_0)$ be the set of roots of $G_0$ (with respect to $T$). Let \begin{equation}\label{eq:Phi(G_0) decomp}
\Phi(G_0)=\bigsqcup_i\Phi_i(G_0), \quad \Delta(G_0)=\bigsqcup_i\Delta_i(G_0)
\end{equation}
denote the decomposition of $\Phi(G_0)$ into irreducible root systems and the corresponding decomposition of simple roots. The aim of this subsection is to study the above decompositions. As a first step, we give an explicit realization of the root system of $G_0$.

\subsubsection{Type A} 
Suppose $G$ is of type $A_n$, $n\in \bZ_{\geq 1}$. Then $d\in \{2,...,n+1\}$. In this case, we have 
\begin{equation}\label{eq:A L-simple roots}
\Delta(G_0)=\{\alpha\in\Phi_G|\Ht(\alpha)=d\}=\{\chi_i-\chi_{i+d}=\sum_{j=i}^{i+d-1}\alpha_j|1\leq i\leq n+1-d \}.
\end{equation}
Thus, $|\Delta(G_0)|=n+1-d=n-m$.
For each $1\leq i\leq d$, the root subsystem $\Phi_i(G_0)$  is of Type A. More precisely, let $n+1=\tau d+\sigma$ where $0\leq \sigma\leq d-1$. Then we have: 
\begin{equation}
\Phi_i(G_0)=
\begin{cases}
A_\tau ,\quad 1\leq i\leq \sigma;\\
A_{\tau-1},\quad \sigma+1\leq i\leq d.
\end{cases}
\end{equation}

\subsubsection{Type B} 
Suppose $G$ is  of type $B_n$, $n\in \bZ_{\geq 2}$. Then $d=2m \in \{4,6,...,2n\}$. In this case, we have
\begin{equation}\label{eq:B L-simple roots}
\Delta(G_0)=
\begin{cases}
\chi_i-\chi_{i+d},\ \Ht=d, \quad 1\leq i\leq n-d,\ \text{when}\  n\geq d+1;\\
\chi_{n+1-d},\ \Ht=d, \quad \text{when}\  n\geq d;\\
\chi_i+\chi_{\bar{i}},\ \Ht=d, \quad n+2-d\leq i\leq n-m,\ \bar{i}:=2n+2-d-i>i;\\
\chi_{n+1-d-m}+\chi_{n+1-m},\ \Ht=2d,  \quad \text{when}\ n\geq \frac{3}{2}d.
\end{cases}
\end{equation}
It follows that 
\[
|\Delta(G_0)|=
\begin{cases}
n-m, \quad & n<\frac{3}{2}d;\\
n-m+1, \quad & n\geq\frac{3}{2}d.
\end{cases}
\]
We now consider three cases: 
\begin{enumerate} 
\item  $n<d$: 
For each $1\leq i\leq n-m$, let $\Delta_i(G_0):=\{\chi_i+\chi_{\bar{i}}\}$, where $\bar{i}$ is as defined in \eqref{eq:B L-simple roots}. Note that for all $i$, the root system $\Phi_i(G_0)$ is of type $A_1$. Thus, $\Phi(G_0)$ is direct sum of $n-m$ copies of $A_1$.\\

\item $d\leq n<\frac{3}{2}d$:
For each $1\leq i\leq m-1$, let $i^*:=i+n+1-d$. Then 
\begin{align*}
\Delta_i(G_0)=&\{\chi_{i^*}+\chi_{\bar{i^*}}\}
\sqcup\{\chi_{\bar{i^*}-kd}-\chi_{\bar{i^*}-kd+d}\,|\, 1\leq k\leq\lfloor\frac{\bar{i^*}-1}{d}\rfloor\}\\
&\sqcup\{\chi_{i^*-kd}-\chi_{i^*-kd+d}\,|\, 1\leq k\leq\lfloor\frac{i^*-1}{d}\rfloor\}.
\end{align*}
Moreover, $\Delta_m(G_0):=\{\chi_{n+1-d}\}$.
One readily verifies that the for all $i$, the root system $\Phi_i(G_0)$ is of type A.\\

\item $n\geq\frac{3}{2}d$: 
For $1\leq i\leq m-1$,  $\Delta_i(G_0)$ is as in the previous case. For $i=m$, we have 
$$
\Delta_m(G_0)=\{\chi_{n+1-d}\}\sqcup\{\chi_{n+1-d-kd}-\chi_{n+1-kd}\,|\, 1\leq k\leq\lfloor\frac{n}{d}-1\rfloor \}.
$$
For $i=m+1$, we have
$$
\Delta_{m+1}(G_0):=\{\chi_{n+1-m-d}+\chi_{n+1-m}\}\sqcup\{\chi_{n+1-m-kd}-\chi_{n+1-m-kd+d}\,|\,1\leq k\leq\lfloor\frac{n-m}{d}\rfloor\}.
$$
The types of irreducible subsystems  $\Phi_i(G_0)$ generated by $\Delta_i(G_0)$ is as follows:
\begin{equation}
\Phi_i(G_0)=
\begin{cases}
\text{Type A},\quad 1\leq i\leq m-1;\\
\text{Type B},\quad i=m;\\
\text{Type D},\quad i=m+1.
\end{cases}
\end{equation}
\end{enumerate}

\subsubsection{Type C} 
Suppose $G$ is  of type $C_n$, $n\in \bZ_{\geq 3}$. Then $d=2m \in \{4,6,...,2n\}$. In this case, we have
\begin{equation}\label{eq:C L-simple roots}
\Delta(G_0)=\{\alpha\in\Phi(G)|\Ht(\alpha)=d \}=
\begin{cases}
\chi_i-\chi_{i+d}, \quad 1\leq i\leq n-d;\\
\chi_i+\chi_{\bar{i}}, \quad n+1-d\leq i\leq n-m, \ \bar{i}=2n+1-d-i>i.
\end{cases}
\end{equation}
Thus, 
$|\Delta(G_0)|=n-m$. 

To study the decomposition of $\Phi(G_0)$ into simple systems, we need to consider two cases: 
\begin{enumerate} 
\item  Case $n<d$:  
For each $1\leq i\leq n-m$, $\Delta_i(G_0):=\{\chi_i+\chi_{\bar{i}}\}$, where $\bar{i}$ is as defined in \eqref{eq:C L-simple roots}. Moreover, $\Phi_i(G_0)\simeq\Phi(\GL_2)$; thus, $\Phi(G_0)$ is direct sum of $n-m$ copies of $A_1$. 

\item
Case $n\geq d$:
For each $1\leq i\leq m$, denote $i^*:=n-d+i$. Then 
\begin{align*}
\Delta_i(G_0)=&\{\chi_{i^*}+\chi_{\bar{i^*}}\}
\sqcup\{\chi_{\bar{i^*}-kd}-\chi_{\bar{i^*}-kd+d}\,|\, 1\leq k\leq\lfloor\frac{\bar{i^*}-1}{d}\rfloor\}\\
&\sqcup\{\chi_{i^*-kd}-\chi_{i^*-kd+d}\,|\, 1\leq k\leq\lfloor\frac{i^*-1}{d}\rfloor\}.
\end{align*}
For each $1\leq i\leq m$, the sub root system $\Phi_i(G_0)$ spanned by above $\Delta_i(G_0)$ is a type A root system. More precisely, let $n=\tau m+\sigma$ where $0\leq\sigma\leq m-1$, there are $\sigma$ terms of $\Phi_i(G_0)$, which are isomorphic to $A_\tau$, while $m-\sigma$ terms of $\Phi_i(G_0)$ are isomorphic to $A_{\tau-1}$.
\end{enumerate}

\subsubsection{Type D} 
Suppose $G$ is of type $D_n$, $n\in \bZ_{\geq 4}$.  Then $d=2m-2\in \{6,8,...,2(n-1)\}$. In this case, we have 
\begin{equation}\label{eq:D L-simple roots}
\Delta(G_0)=
\begin{cases}
\chi_i-\chi_{i+d},\ \Ht=d, \quad 1\leq i\leq n-d,\ \text{when}\ n\geq d+1;\\
\chi_i+\chi_{\bar{i}},\ \Ht=d, \quad n+1-d\leq i\leq n-m,\ \bar{i}=2n-d-i>i;\\
\chi_{n-d}+\chi_n,\ \Ht=d, \quad \text{when}\  n\geq d+1;\\
\chi_{n+1-d-m}+\chi_{n+1-m},\ \Ht=2d,  \quad \text{when}\ n\geq \frac{3}{2}d+1.
\end{cases}
\end{equation}
Thus, 
$|\Delta(G_0)|=
\begin{cases}
n-m, \quad & n\leq d;\\
n-m+1, \quad & \frac{3}{2}d\geq n\geq d+1;\\
n-m+2, \quad & n\geq\frac{3}{2}d+1.
\end{cases}$
We consider three cases: 
\begin{enumerate} 
\item 
 $n\leq d$: 
For each $1\leq i\leq n-m$, let $\Delta_i(G_0):=\{\chi_i+\chi_{\bar{i}}\}$, where $\bar{i}$ is as defined in \eqref{eq:D L-simple roots}. Moreover, $\Phi_i(G_0)\simeq A_1$; thus, $\Phi(G_0)$ is direct sum of $n-m$ copies of $A_1$.

\item 
 $d+1\leq n\leq\frac{3}{2}d$: 
For each $1\leq i\leq m-2$, denote $i^*:=i+n-d$, we define
\begin{align*}
\Delta_i(G_0):=&\{\chi_{i^*}+\chi_{\bar{i^*}}\}
\sqcup\{\chi_{\bar{i^*}-kd}-\chi_{\bar{i^*}-kd+d}\,|\, 1\leq k\leq\lfloor\frac{\bar{i^*}-1}{d}\rfloor\}\\
&\sqcup\{\chi_{i^*-kd}-\chi_{i^*-kd+d}\,|\, 1\leq k\leq\lfloor\frac{i^*-1}{d}\rfloor\}.
\end{align*}
Note that here last part is empty, but will become nonempty when $n\geq\frac{3}{2}d+1$. We also define
$$
\Delta_{m-1}(G_0):=\{\chi_{n-d}+\chi_n\},\quad \Delta_m(G_0):=\{\chi_{n-d}-\chi_n\}.
$$
 The type of sub root system $\Phi_i(G_0)$ generated by $\Delta_i(G_0)$ is as follows:
\begin{equation}
\Phi_i(G_0)=\text{Type A},\quad 1\leq i\leq m.
\end{equation}

\item
 $n\geq\frac{3}{2}d+1$: 
For $1\leq i\leq m-2$, we define $\Delta_i(G_0)$ as in the $d+1\leq n\leq\frac{3}{2}d$ case. For $i=m-1$, we define
$$
\Delta_{m-1}(G_0):=\{\chi_{n-d}+\chi_n\}\sqcup\{\chi_{n-kd}-\chi_{n-kd+d}\,|\,1\leq k\leq\lfloor\frac{n-1}{d}\rfloor\}
$$
For $i=m$, we define
\begin{align*}
\Delta_m(G_0):=&\{\chi_{n+1-m-d}+\chi_{n+1-m}\}\sqcup\{\chi_{n+1-m-d}+\chi_{n+1-m}\}\\
&\sqcup\{\chi_{n+1-m-kd}-\chi_{n+1-m-kd+d}\,|\,1\leq k\leq\lfloor\frac{n+1-m}{d}\rfloor\}.
\end{align*}
The type of sub root system $\Phi_i(G_0)$ generated by $\Delta_i(G_0)$ is as follows:
\begin{equation}
\Phi_i(G_0)=
\begin{cases}
\text{Type A},\quad 1\leq i\leq m-2;\\
\text{Type D},\quad i=m-1,m.
\end{cases}
\end{equation}
\end{enumerate}

\subsubsection{Decomposition of $\fg_0$ into Lie subalgebras}
With above notations, for $\alpha\in\Phi(G)$, we set 
$$
\chi(\alpha):=
\begin{cases}
\{\chi_i,\chi_j\},\quad \alpha\in\{\pm\chi_i\pm\chi_j\},\ i\neq j;\\
\{\chi_i\},\ \qquad \alpha\in\{\pm\chi_i, \pm2\chi_i\}.
\end{cases}
$$
Moreover, for a set of roots $S$, we set 
$$
\chi(S)=\bigcup_{\alpha\in S}\chi(\alpha).
$$
Then we have $\chi(\Phi_i(G_0))=\chi(\Delta_i(G_0))$ and $\chi(\Phi_i(G_0))\cap\chi(\Phi_j(G_0))=\varnothing$, $\forall i\neq j$. 

Now we define a Lie subalgebra $\fl_i$ of $\fg_0$ by
$$
\fl_i=\bigoplus_{\chi_j\in \chi(\Phi_i(G_0))} \Lie(T_j) \bigoplus_{\alpha\in \Phi_i(G_0)} \fg_{\alpha}. 
$$
Then we obtain a decomposition 
\begin{equation}\label{eq:G_0 block decomp}
\fg_0=\bigoplus_{i} \fl_i \bigoplus_{\chi_j \not\in \chi(\Phi(G_0))} \Lie(T_j). 
\end{equation}

\subsubsection{}\label{s:Levi nonzero intersects subrepns}
With above preparations, we obtain the following result. 
\begin{prop}\label{p:Levi nonzero intersects subrepns}
Let $G$ be a classical group and $M\subseteq G$ an admissible Levi. Then for all irreducible $G_0$-submodules $V_i\subseteq \fg_1$, we have $V_i\cap \fm_1\neq 0$. 
\end{prop} 
\begin{proof}
For $\alpha\in\Phi(\fm_1)$, let $\mathrm{hw}(\alpha)$ denote the highest weight of the subrepresentation $V_i$ that contains $\fg_\alpha$. In order to verify Proposition \ref{p:Levi nonzero intersects subrepns}, it suffices to show that $\{\mathrm{hw}(\alpha)|\alpha\in\Phi(\fm_1)\}=\mathrm{wt}^+(\fg_1)$ give all the highest weights of $\fg_1$ as $G_0$-representations. 
In view of \cite[Proposition 27]{KY}, the number of highest weights is $n+1-|\Delta(G_0)|$. Thus it suffices to verify that there are $n+1-|\Delta(G_0)|$ distinct $\mathrm{hw}(\alpha)$'s. Recall $\Phi(\fm_1)=\{\alpha_{n-m+1},...,\alpha_n,-\theta_M \}$.

(i) Type A. We have $\mathrm{hw}(\alpha_i)=\sum_{j=i'}^i\alpha_i$, where $1\leq i'\leq d$, $n-m+1\leq i\le n$ and $d|(i-i')$; $\mathrm{hw}(-\theta_M)=\sum_{j=i'}^i\alpha_{n-m}$, where $1\leq i'\leq d$ and $d|(n-m-i')$. Thus these $i'$'s are all distinct module $d$, giving $n+1-(n-m)=m+1=d$ many highest weights.

(ii) Type B, C, D. The highest weight $\mathrm{hw}(\alpha)$ for $\alpha\in\Phi(\fm_1)$ is of the form $\chi_{i_{\alpha}}+\chi_{j_{\alpha}}$ for some integers $i_{\alpha}\leq j_{\alpha}$.
Similar to type A, we can check the map from $\Phi(\fm_1)$ to the set of highest weights of $\fg_1$ is surjective is still surjective. Although, it may fail to be injective in some situations for type B and D. 
\end{proof}

\subsection{Weight space decomposition of $\fu_{0}$ and $\fu_\phi$}\label{ss:wt sp decomp}
In this subsection, we keep the assumptions and notations of \S~\ref{ss:toric}. 
To prove theorem \ref{t:toric}, we need to study the $T_\phi$-action on the unipotent group $U_0$. 
We first introduce some general notations about this action and prove Proposition \ref{p:stabiliser}(4).

We denote the Lie algebras of $U_\phi\subset U_0$ by $\fu_\phi\subset\fu_0$, and the weights of the adjoint action of $T_\phi$ on $\fu_0$ by $\Phi_{T_\phi}(\fu_0)$. We have the weight subspace decomposition: 
\begin{equation} 
\fu_0=\bigoplus_{\bgamma\in\Phi_{T_\phi}(\fu_0)}\fu_{\bgamma}.\footnote{Note that, a priori, the action of $T_\phi$ on $\fu_0$ may have $0$ as an eigenvalue, though we prove that this cannot happen.}
\end{equation} 
On the other hand, we have the weight decomposition with respect to maximal torus $T$: 
\begin{equation}\label{eq:T_phi decomp of u_L}
\fu_0=\bigoplus_{\gamma\in\Phi(\fu_0)}\fg_\gamma,
\end{equation}

The inclusion $T_{\phi}\to T$ induces a restriction map 
\begin{equation} \label{eq:Tphi weight}
\pi:\Phi_G\rightarrow X^*(T_\phi).
\end{equation}
More explicitly, if we set $\Delta_M^c:=\Delta_G-\Delta_M$, then for $\gamma=\sum_{\alpha\in\Delta_G}n_\alpha\alpha\in\Phi_G$ we have
\begin{equation}\label{eq:pi explicit}
\pi(\sum_{\alpha\in\Delta_G}n_\alpha\alpha)=\sum_{\alpha\in\Delta_M^c}n_\alpha\alpha.
\end{equation}
Moreover, we have
\begin{equation}\label{eq:T_phi wt subspace}
\fu_{\bgamma}=\bigoplus_{\gamma\in\pi^{-1}(\bgamma)}\fg_\gamma, \quad \forall \bgamma\in\Phi_{T_\phi}(\fu_0).
\end{equation}

Note that $\fu_\phi$ is a $T_\phi$-invariant subspace of $\fu_0$. We have a decomposition
\begin{equation}\label{eq:Tphi decomp of uphi}
\fu_\phi=\bigoplus_{\bgamma\in\Phi_{T_\phi}(\fu_\phi)}\fu_{\phi,\bgamma},\quad \textnormal{where } \fu_{\phi,\bgamma}=\fu_{\bgamma}\cap \fu_{\phi}.
\end{equation}

In view of Proposition \ref{p:stabiliser}(i-ii), for every $\bgamma\in\Phi_{T_\phi}(\fu_\phi)$, we have $\dim(\fu_{\phi,\bgamma})=1$.

\begin{lem}\label{l:nonzero initial}
If $\displaystyle \gamma=\sum_{\alpha\in\Delta_M^c}n_\alpha\alpha+\sum_{\alpha\in\Delta_M}m_\alpha\alpha\in\Phi(\fu_0)$, then $\displaystyle \sum_{\alpha\in\Delta_M^c}n_\alpha\alpha\neq0$. 
In particular,  $0\not\in\Phi_{T_\phi}(\fu_0)$ and $\pi$ induces a surjection 
$$\pi:\Phi(\fu_0)\tra\Phi_{T_\phi}(\fu_0).$$
\end{lem}
\begin{proof}
Recall that $\gamma\in\Phi(\fu_0)$ satisfies $d\mid\Ht(\gamma)$.
If $\pi(\gamma)=0$, then we have $\gamma\in\Phi_M$ and $\Ht(\gamma)<h_M=d$, which implies $\gamma=0$.
\end{proof}

\subsubsection{Proof of Proposition \ref{p:stabiliser}(iv)}\label{ss:stabiliser} 
Let $\fl_\phi$, $\ft$, and $\ft_\phi$ denote the Lie algebras of $L_\phi$, $T$ and $T_\phi$ respectively. 
Consider the $T_\phi$-eigenspace decomposition:
$$
\fg_0=\ft\oplus\bigoplus_{\bgamma\in\Phi_{T_\phi}(\fg_0)}\fg_{0,\bgamma}, \quad \fl_\phi=\ft_{\phi}\oplus\bigoplus_{\bgamma\in\Phi_{T_\phi}(\fl_\phi)}\fl_{\phi,\bgamma}.
$$
Recall $U_\phi=U_0\cap L_\phi$. 
It suffices to show that $B_1=T_\phi U_\phi$ is a Borel of $L_\phi$ as $B_\phi\supset B_1$. Then we reduce to show that
\begin{equation}\label{eq:1}
\Phi_{T_\phi}(\fl_\phi)=\Phi_{T_\phi}(U_\phi)\sqcup(-\Phi_{T_\phi}(U_\phi))
\end{equation}
The projection $\pi:\Phi(\fg_0)\ra\Phi_{T_\phi}(\fg_0)$ is surjective (Lemma \ref{l:nonzero initial}). 
Given any $\bgamma\in\Phi_{T_\phi}(\fl_\phi)$, we have
$$
\fl_{\phi,\bgamma}\subset\fl_{\bgamma}=\bigoplus_{\gamma\in\pi^{-1}(\bgamma)}\fg_\gamma.
$$

If a weight $\gamma=\sum_{\alpha\in\Delta_G}n_\alpha\alpha\in\Phi_G\in\pi^{-1}(\bgamma)$ is positive, 
then by Lemma \ref{l:nonzero initial}, $n_\alpha>0$ for some $\alpha\in\Delta_M^c$. 
By \eqref{eq:pi explicit}, this implies that $\gamma'>0$ for every $\gamma'\in\pi^{-1}(\bgamma)$. 
Therefore, we have $\fl_{\phi,\bgamma}\subset\fu_0 \cap\fl_\phi$ and $\bgamma\in\Phi_{T_\phi}(U_\phi)$. 
Similarly, when $\gamma<0$ for some $\gamma\in\pi^{-1}(\bgamma)$, we deduce $\bgamma\in -\Phi_{T_\phi}(U_\phi)$. 
This gives the inclusion $\subset$ of \eqref{eq:1}.

On the other hand, as $\pi(-\gamma)=-\pi(\gamma)$, the right hand side of \eqref{eq:1} is a disjoint union. 
Moreover, as $L_\phi$ is reductive, $\bgamma\in\Phi_{T_\phi}(\fl_\phi)$ if and only if $-\bgamma\in\Phi_{T_\phi}(\fl_\phi)$. This provides the opposite inclusion and completes the proof. \hfill\qed

\subsection{Classification of $\Phi_{T_\phi}(\fu_0)$ via $\pi$}
\label{sss:gamma_1 and gamma_2}
Let $\bgamma \in \Phi_{T_\phi}(\fu_0)$ and write 
$\bgamma=\sum_{\alpha\in \Delta_M^c } n_\alpha \alpha$. 

%In the following, we assume $G$ is a classical group and $M$ is admissible (Definition \ref{d:admissible}). 

\begin{lem} \label{l:fibresOfpi}
	There are three possibilities for $\pi^{-1}(\bgamma)$:
\begin{itemize}
\item [(I)] $\pi^{-1}(\bgamma)=\{\gamma\}$ is a singleton;
\item [(II)] $|\pi^{-1}(\bgamma)|\ge 2$ and for any $\gamma_1,\gamma_2\in\pi^{-1}(\bgamma)$, $\Ht(\gamma_1)=\Ht(\gamma_2)$;
\item [(III)] $G$ is of type B or D and $\pi^{-1}(\bgamma)$ contains roots of unequal heights; more precisely, \[
\{\gamma_1=\sum_{\alpha\in\Delta_M^c} n_\alpha\alpha,\gamma_2=\gamma_1+\alpha_{n-m+1}+\theta_M\}\subset \pi^{-1}(\bgamma).
\] 
\end{itemize}
\end{lem} 

\begin{proof}
Let $\gamma_1$ and $\gamma_2$ be elements of $\pi^{-1}(\bgamma)$ and write 
\[
\gamma_i=\sum_{\alpha\in\Delta_M^c}n_\alpha\alpha+\sum_{\alpha\in\Delta_M}m^i_\alpha\alpha, \quad m^i_\alpha\in\mathbb{N},\ i=1,2.
\]
Note that $\Ht(\gamma_i)$ is divisible by $d$; thus, $d=h_M$ divides 
$$
|\Ht(\gamma_1)-\Ht(\gamma_2)|=|\sum_{\alpha\in\Delta_M}m^1_\alpha-\sum_{\alpha\in\Delta_M}m^2_\alpha|\leq\max_{i\in\{1,2\}}\sum_{\alpha\in\Delta_M}m^i_\alpha.
$$
There are two possibilities of the maximal value of $\sum_{\alpha\in\Delta_M}m^i_\alpha$:

(i) If $G$ is of type A and type C, then $\sum_{\alpha\in\Delta_M}m^i_\alpha\leq \Ht(\theta_M)=h_M-1<h_M$, so $\Ht(\gamma_1)=\Ht(\gamma_2)$.

(ii) If $G$ is of type B and type D, then $\sum_{\alpha\in\Delta_M}m^i_\alpha\leq \Ht(\theta_M+\alpha_{n-m+1})=h_M$, where $\Delta_M=\{\alpha_{n-m+1},...,\alpha_n\}$. 
Therefore, we have $\Ht(\gamma_1)=\Ht(\gamma_2)$ or $|\Ht(\gamma_1)-\Ht(\gamma_2)|=h_M$. 
The second case happens only if one of $\gamma_i$ satisfies $m_\alpha^i=0$ for all $i$, while the other one satisfies all of $m_\alpha^i$ achieve the maximal possible value. Note that in this case, there is no weight other than $\gamma_1, \gamma_2$ in $\pi^{-1}(\bgamma)$.
\end{proof}

We say a root $\gamma\in\Phi(\fu_0)$ or a $T_\phi$-weight $\bgamma=\pi(\gamma)$ is of type (I) (resp. (II), (III)) if $\bgamma=\pi(\gamma)$ satisfies above condition (I) (resp. (II), (III)). The above lemma allows us to give an explicit description of the weight spaces $\Phi_{T_\phi}(\fu_0)$: 

\begin{prop}\label{p:G type v.s. weight type}
\textnormal{(i)} Suppose $G$ is of type A or C. Then only type (I) roots appear. In other words, $\pi$ defines a bijection $\Phi(\fu_0)\xrightarrow{\sim}\Phi_{T_\phi}(\fu_0)$.
	
\textnormal{(ii)} Suppose $G$ is of type B. 
\begin{itemize} 
\item If $n<\frac{3}{2}d$, only type (I) roots appear; 
\item if $n\geq\frac{3}{2}d$, both type (I) and (III) roots appear. 
\end{itemize} 
Moreover, for any type (III) weight $\bgamma\in \Phi_{T_\phi}(\fu_0)$, we have $|\pi^{-1}(\bgamma)|=2$. 

\textnormal{(iii)} Suppose $G$ is of type D.
\begin{itemize} 
\item If $n\leq d$, only type (I) roots appear; 
\item if $d+1\leq n\leq\frac{3}{2}d$, both type (I) and (II) roots appear; 
\item if $n\geq\frac{3}{2}d+1$, all of type (I),(II),(III) roots appear. 
\end{itemize} 
Moreover, for any type (II), (III) weight $\bgamma \in \Phi_{T_\phi}(\fu_0)$, we have $|\pi^{-1}(\bgamma)|=2$.
\end{prop}
\begin{proof}
We denote the set of type (I) (resp. (II), resp. (III)) roots of $\fu_0$ by $\Phi_{T_\phi}^{\I}(\fu_0)$ (resp. $\Phi_{T_\phi}^{\II}(\fu_0)$, resp. $\Phi_{T_\phi}^{\III}(\fu_0)$) and explicitly study the decomposition 
$$
\Phi(\fu_0)=\Phi_{T_\phi}^{\I}(\fu_0)\sqcup \Phi_{T_\phi}^{\II}(\fu_0)\sqcup \Phi_{T_\phi}^{\III}(\fu_0).
$$ 

With the notations of \S~\ref{ss:classical}, we have
$$\Delta_G=\Delta_M^c\sqcup\Delta_M=\{\alpha_1,\dots,\alpha_{n-m}\}\sqcup\{\alpha_{n-m+1},...,\alpha_n\}.$$ 
Suppose $\gamma_1=\sum_{j=i_1}^n n_j\alpha_j,\gamma_2=\sum_{j=i_2}^n m_j\alpha_j\in\Phi(\fu_0)$ satisfy $\pi(\gamma_1)=\pi(\gamma_2)$ and $n_{i_1} > 0, m_{i_2}> 0$. 
A simple observation is that $i_1=i_2$. Indeed, from Lemma \ref{l:nonzero initial}, we have $i_1,i_2\leq n-m$. Then we conclude the claim from \eqref{eq:pi explicit}.
\subsubsection{Type A} 
Given $1\leq d=m+1 \leq n+1$,  we have
\begin{equation}
\Phi(\fu_0)=\{\sum_{j=i}^{i+td-1}\alpha_j|1\leq i\leq i+td-1\leq n\}.
\end{equation} 

It follows from the observation that type (II) roots do not exist. 
We have 
\begin{equation}\label{eq:A T_phi-wts in fu_L}
\Phi(\fu_0)=\Phi_{T_\phi}^{\I}(\fu_0), \quad
\pi:\Phi(\fu_0)\xrightarrow{\sim}\Phi_{T_\phi}(\fu_0)~ \textnormal{is a bijection}.
\end{equation}

\subsubsection{Type B} \label{sss:type B uL}  
Given $2\leq d=2m\leq 2n$, we have
\begin{equation}
\begin{split}
\Phi(\fu_0)=&
\{\sum_{j=i}^{i+td-1}\alpha_j|1\leq i<i+td-1\leq n\}\\
&\sqcup\{\sum_{j=i}^{i+k-1}\alpha_j+2\sum_{j=i+k}^n\alpha_j|i+k=2n-i-td+2>i,\ 1\leq t\leq\lfloor\frac{2n-1}{d}\rfloor\}.
\end{split}
\end{equation}

We first show that type (II) roots do not exist. 
Suppose $\gamma_1,\gamma_2\in\Phi(\fu_0)$ satisfy $\pi(\gamma_1)=\pi(\gamma_2)$ and $\Ht(\gamma_1)=\Ht(\gamma_2)$.
If $\gamma_1$ and $\gamma_2$ are both of the form $\sum_{j=i}^{i+td-1}\alpha_j$ or $\sum_{j=i}^{i+k-1}\alpha_j+2\sum_{j=i+k}^n\alpha_j$, then $\Ht(\gamma_1)=\Ht(\gamma_2)$ implies $\gamma_1=\gamma_2$. 
If $\gamma_1=\sum_{j=i}^{i+td-1}\alpha_j$ while $\gamma_2=\sum_{j=i}^{i+k-1}\alpha_j+2\sum_{j=i+k}^n\alpha_j$, $2n-2i-k+2=qd$, then $\Ht(\gamma_1)=td\leq n+1-i$ while $\Ht(\gamma_2)=2n-i-(i+k)+2\geq n+2-i>\Ht(\gamma_1)$. Therefore type (II) roots would not appear.

We have $\Phi(\fu_0)=\Phi_{T_\phi}^{\I}(\fu_0)\sqcup \Phi_{T_\phi}^{\III}(\fu_0)$. These roots are explicitly given by:
\begin{equation}\label{eq:B T_phi-wts in fu_L}
\begin{split}
\Phi_{T_\phi}^{\I}(\fu_0)=&\{\sum_{j=i}^{i+k-1}\alpha_j\in\Phi(\fu_0)|i+k-1\neq n-m\}\sqcup\{\sum_{j=i}^{i+k-1}\alpha_j+2\sum_{j=i+k}^n\alpha_j\in\Phi(\fu_0)|i+k-1\neq n-m\};\\
\Phi_{T_\phi}^{\III}(\fu_0)=&\bigsqcup_{1\leq t\leq\lfloor\frac{n-m}{d}\rfloor}\{\sum_{j=i}^{n-m}\alpha_j, \sum_{j=i}^{n-m}\alpha_j+2\sum_{j=n-m+1}^n\alpha_j|n-m-i+1=td\}.
\end{split}
\end{equation}

\subsubsection{Type C} \label{sss:type C uL}
Given $2\leq d=2m\leq 2n$, we have
\begin{equation}
\begin{split}
\Phi(\fu_0)=&
\{\sum_{j=i}^{i+td-1}\alpha_j|1\leq i<i+td-1\leq n\}\\
&\sqcup\{\sum_{j=i}^{i+k-1}\alpha_j+2\sum_{j=i+k}^{n-1}\alpha_j+\alpha_n|i+k=2n-i-td+1\geq i,\ 1\leq t\leq\lfloor\frac{2n-1}{d}\rfloor\}.
\end{split}
\end{equation}

One can check that type (II) roots do not exist as in \S\ref{sss:type B uL}. 
We have 
\begin{equation}\label{eq:C T_phi-wts in fu_L}
\Phi(\fu_0)=\Phi_{T_\phi}^{\I}(\fu_0), \quad
\pi:\Phi(\fu_0)\xrightarrow{\sim}\Phi_{T_\phi}(\fu_0)~ \textnormal{is a bijection}.
\end{equation}

\subsubsection{Type D} \label{sss:type D uL} 
Given $2\leq d=2m\leq 2n-2$, we have
\[
\begin{split}
\Phi(\fu_0)=&
\{\sum_{j=i}^{i+td-1}\alpha_j|1\leq i<i+td-1\leq n\}
\sqcup\{\sum_{j=i}^{n-2}\alpha_j+\alpha_n|i=n-td, 1\leq t\leq\lfloor\frac{n-1}{d}\rfloor \}\\
&\sqcup\{\sum_{j=i}^{i+k-1}\alpha_j+2\sum_{j=i+k}^{n-2}\alpha_j+\alpha_{n-1}+\alpha_n|i+k=2n-i-td>i,\ 1\leq t\leq\lfloor\frac{2n-3}{d}\rfloor\}.
\end{split}
\]

Both type (I), (II) and (III) roots exist. We have $
\Phi(\fu_0)=\Phi_{T_\phi}^{\I}(\fu_0)\sqcup\Phi_{T_\phi}^{\II}(\fu_0)\sqcup\Phi_{T_\phi}^{\III}(\fu_0)$. 
By a similar argument of \S\ref{sss:type B uL}, we deduce that these roots are given by:
\begin{equation}\label{eq:D T_phi-wts in fu_L}
\begin{split}
\Phi_{T_\phi}^{\II}(\fu_0)=&\{\sum_{j=i}^{i+k-1}\alpha_j\in\Phi(\fu_0)|i+k-1\neq n-m,n-1\}\\
&\sqcup\{\sum_{j=i}^{i+k-1}\alpha_j+2\sum_{j=i+k}^{n-2}\alpha_j+\alpha_{n-1}+\alpha_n\in\Phi(\fu_0)|i+k-1\neq n-m\};\\
\Phi_{T_\phi}^{\II}(\fu_0)=&\bigsqcup_{1\leq t\leq\lfloor\frac{n-1}{d}\rfloor}\{\sum_{j=i}^{n-1}\alpha_j, \sum_{j=i}^{n-2}\alpha_j+\alpha_n|n-i=td\};\\
\Phi_{T_\phi}^{\III}(\fu_0)=&\bigsqcup_{1\leq t\leq\lfloor\frac{n-m}{d}\rfloor}\{\sum_{j=i}^{n-m}\alpha_j, \sum_{j=i}^{n-m}\alpha_j+2\sum_{j=n-m+1}^{n-2}\alpha_j+\alpha_{n-1}+\alpha_n|n-m-i+1=td\}.
\end{split}
\end{equation}	
\end{proof}

\subsection{Structure of $\fu_\phi$ and a complement of $\fu_\phi$ in $\fu_0$} \label{ss:structure Uphi}
The goal of this subsection is to study $\fu_{\phi}$ and define an explicit $T_\phi$-invariant complement $\fu^c$ of $\fu_\phi$ in $\fu_0$, i.e., a decomposition of $T_\phi$-modules 
\[
\fu_0= \fu_\phi \oplus \fu^c.
\] 
We start with an observation: 
\begin{lem} \label{l:fu_phi=z(phi)} 
$\fu_\phi=\{v\in\fu_0\,|\,[v,\phi]=0\}$.
\end{lem} 
\begin{proof} 
Note that the exponential map on $\fu_0$ is well-defined and induces an isomorphism $\fu_0\simeq U_0$. 
Then we conclude \eqref{l:fu_phi=z(phi)} from the fact that $U_\phi$ is the stabilizer of $\phi$ in $U_0$. 	
\end{proof} 

Recall the root space decomposition of $\fu_\phi$ \eqref{eq:Tphi decomp of uphi}:
\begin{equation}
\fu_\phi=\bigoplus_{\bgamma\in\Phi_{T_\phi}(\fu_\phi)}\fu_{\phi,\bgamma},\qquad \textnormal{where} \qquad \fu_{\phi,\bgamma}:=\fu_{\bgamma}\cap \fu_{\phi},\quad \dim(\fu_{\phi,\bgamma})=1.
\end{equation}

Proposition \ref{p:G type v.s. weight type} then implies that  $\fu_{\bgamma}$ is either a root subspace $\fg_\gamma$ for some $\gamma\in\Phi(\fu_0)\subseteq\Phi(G)$, or the direct sum of two root subspaces. 
The following result provides a criterion for whether a root subspace is contained in $\fu_\phi$. 
\begin{lem}\label{l:roots in U_phi}
For $\gamma\in\Phi(\fu_0)$, the inclusion $\fu_\gamma\subseteq \fu_\phi$ holds if and only if for every root $\alpha\in\Phi(\fm_{-1})$, we have $\gamma+\alpha\not\in\Phi_G$.
\end{lem}
\begin{proof}
Recall that $\displaystyle \phi=\sum_{\alpha\in\Phi(\fm_{-1})} \lambda_{\alpha} E_{\alpha}$. By \eqref{l:fu_phi=z(phi)}, $E_\gamma$ belongs to $\fu_\phi$ if and only if 
$$
[E_\gamma,\phi]=\sum_{\alpha\in\Phi(\fm_{-1})}\lambda_\alpha [E_\gamma,E_\alpha]=0.
$$
The latter is, in turn, equivalent to $[E_\gamma,E_\alpha]=0$, i.e. $\gamma+\alpha\not\in\Phi_G$, for every $\alpha\in\Phi(\fm_{-1})$. 
\end{proof}

We now describe the subspace $\fu_{\phi,\bgamma}$ for every $T_\phi$-weight $\bgamma\in \Phi_{T_\phi}(\fu_0)$ and define a complement $\fu^c_{\bgamma}$ of $\fu_{\phi,\bgamma}$ in $\fu_{\bgamma}$. When $\bgamma$ is of type (I) (i.e. $\pi^{-1}(\bgamma)=\{\gamma\}$ is a singleton), the above corollary provides a criterion for $\fu_{\phi,\bgamma}=\fg_{\bgamma}$ or $\fu_{\phi,\bgamma}=0$ and we define $\fu^c_{\bgamma}$ in an obvious way. Namely, denote by $\Phi_{T_\phi}^{\I}(\fu_\phi)$ (resp. $\Phi_{T_\phi}^{\I}(\fu^c)$) the $T_\phi$-weights in $\fu_\phi$ (resp. $\fu^c$) of type (I), which can be uniquely lifted to roots in $\Phi(\fu_0)$. Then we have 
\[
\Phi^{\I}(\fu_0)=\Phi_{T_\phi}^{\I}(\fu_\phi)\sqcup\Phi_{T_\phi}^{\I}(\fu^c).
\]
When $\bgamma$ is of type (II) or (III) (which can appear only in types B and D), $\fu_{\phi,\bgamma}$ is no longer a root subspace $\fg_{\gamma}$. 
We will describe the set of $T_\phi$-weights $\Phi_{T_\phi}^{\II}(\fu_\phi)$ and $\Phi_{T_\phi}^{\III}(\fu_\phi)$ of type (II) and (III),
and define $\fu^c_{\bgamma}$. Below, we give an explicit description of these root spaces  in each classical type.

\subsubsection{Type $A_n$} \label{sss:A wts in fu_phi and fu^c} 
Given $\displaystyle \gamma=\sum_{j=i-td+1}^i\alpha_j\in\Phi(\fu_0)$, by Lemma \ref{l:roots in U_phi}, we have
\begin{align*}
\gamma\in\Phi_{T_\phi}^{\I}(\fu^c)&\Leftrightarrow\exists\alpha\in\Phi(\fm_{-1})=\{-\alpha_{n-m+1},-\alpha_{n-m+2},...,-\alpha_n,\sum_{j=n-m+1}^n\alpha_j \}, \gamma+\alpha\in\Phi_G.\\
&\Leftrightarrow i\in\{n-m,n-m+1,...,n\}.
\end{align*}

Then we deduce
\begin{equation}\label{eq:A wts in fu_phi and fu^c}
\begin{split}
\Phi_{T_\phi}^{\I}(\fu^c)=&\{\sum_{j=i-td+1}^i\alpha_j|n-m\leq i\leq n,1\leq i-td+1\leq i\};\\
\Phi_{T_\phi}^{\I}(\fu_\phi)=&\{\sum_{j=i-td+1}^i\alpha_j|1\leq i-td+1\leq i<n-m\}.
\end{split}
\end{equation}

\subsubsection{Type $B_n$}
We have
\begin{equation}\label{eq:B (I)-wts in fu_phi and fu^c}
\begin{split}
\Phi_{T_\phi}^{\I}(\fu^c)=&\{\sum_{j=i-td+1}^i\alpha_j| n-m+1\leq i\leq n\}\\
&\sqcup\{\sum_{j=i-k}^{i-1}\alpha_j+2\sum_{j=i}^n\alpha_j| n-m+2\leq i\leq n, i-k=2n-i-td+2\};\\
\Phi_{T_\phi}^{\I}(\fu_\phi)=&\{\sum_{j=i-td+1}^i\alpha_j| 1\leq i\leq n-m-1\}\\
&\sqcup\{\sum_{j=i-k}^{i-1}\alpha_j+2\sum_{j=i}^n\alpha_j| 2\leq i\leq n-m, i-k=2n-i-td+2\}.
\end{split}
\end{equation}

It remains to study $\fu_{\phi,\bgamma}$ for a type (III) weight $\bgamma\in \Phi_{T_\phi}(\fu_0)$. 
Recall (cf. \eqref{eq:T_phi wt subspace} and \eqref{eq:B T_phi-wts in fu_L}) that
$$
\fu_{\bgamma}=\fg_{\gamma_1}\oplus\fg_{\gamma_2},\ \gamma_1=\sum_{j=i}^{n-m}\alpha_j, \ \gamma_2=\sum_{j=i}^{n-m}\alpha_j+2\sum_{j=n-m+1}^n\alpha_j,\ n-m-i+1=td.
$$
A vector $v=\lambda_1 E_{\gamma_1}+\lambda_2 E_{\gamma_2}$ belongs to $\fu_{\phi,\bgamma}$ if and only if
$$
0=[v,\phi]
=[\lambda_1 E_{\gamma_1}+\lambda_2 E_{\gamma_2},\sum_{\alpha\in\Delta_M} E_{-\alpha}+E_{\theta_M}]
=\lambda_1 E_{\gamma_1+\theta_M}+\lambda_2 E_{\gamma_2-\alpha_{n-m+1}}=(\lambda_1+\lambda_2)E_{\gamma_1+\theta_M}.
$$

Thus $E_{\phi,\bgamma}=E_{\gamma_1}-E_{\gamma_2}$ is a basis of $\fu_{\phi,\bgamma}$ and we define $\fu^c_{\bgamma}$ to be the subspace of $\fu_{\bgamma}$ generated by $E_{c,\bgamma}=E_{\gamma_1}+E_{\gamma_2}$:
\begin{equation}\label{eq:B (III)-wts in fu_phi and fu^c}
\begin{split}
&\Phi_{T_\phi}^{\III}(\fu_\phi)=\Phi_{T_\phi}^{\III}(\fu^c)=\pi(\Phi_{T_\phi}^{\III}(\fu_0));\\
&\fu_{\phi,\bgamma}=k(E_{\gamma_1}-E_{\gamma_2}), \ \fu^c_{\bgamma}=k(E_{\gamma_1}+E_{\gamma_2}),\ \pi^{-1}(\bgamma)=\{\gamma_1,\gamma_2\}, \ \forall \bgamma\in\Phi_{T_\phi}^{\III}(\fu_0).
\end{split}
\end{equation}

\subsubsection{Type C} 
We have
\begin{equation}\label{eq:C wts in fu_phi and fu^c}
\begin{split}
\Phi_{T_\phi}^{\I}(\fu^c)=&\{\sum_{j=i-td+1}^i\alpha_j| n-m\leq i\leq n\}\\
&\sqcup\{\sum_{j=i-k}^{i-1}\alpha_j+2\sum_{j=i}^{n-1}\alpha_j+\alpha_n| n-m+1\leq i\leq n-1, i-k=2n-i-td+1\};\\
\Phi_{T_\phi}^{\I}(\fu_\phi)=&\{\sum_{j=i-td+1}^i\alpha_j| 1\leq i\leq n-m-1\}\\
&\sqcup\{\sum_{j=i-k}^{i-1}\alpha_j+2\sum_{j=i}^{n-1}\alpha_j+\alpha_n| 1\leq i\leq n-m, i-k=2n-i-td+1\}.
\end{split}
\end{equation}

\subsubsection{Type D} 
 \label{sss:D wts in fu_phi and fu^c} 	We have	
\begin{equation}\label{eq:D (I)-wts in fu_phi and fu^c}
\begin{split}
\Phi_{T_\phi}^{\I}(\fu^c)=&\{\sum_{j=i-td+1}^i\alpha_j| n-m+1\leq i\leq n-2\ \&\ i=n\}\\
&\sqcup\{\sum_{j=i-k}^{i-1}\alpha_j+2\sum_{j=i}^{n-2}\alpha_j+\alpha_{n-1}+\alpha_n| n-m+2\leq i\leq n-2, i-k=2n-i-td\};\\
\Phi_{T_\phi}^{\I}(\fu_\phi)=&\{\sum_{j=i-td+1}^i\alpha_j| 1\leq i\leq n-m-1\}\\
&\sqcup\{\sum_{j=i-k}^{i-1}\alpha_j+2\sum_{j=i}^{n-2}\alpha_j+\alpha_{n-1}+\alpha_n| 2\leq i\leq n-m, i-k=2n-i-td\}.
\end{split}
\end{equation}

It remains to study $\fu_{\phi,\bgamma}$ for a type (II) or (III) weight $\bgamma\in \Phi_{T_\phi}(\fu_0)$. 
Recall that  
$$
\fu_{\bgamma}=\fg_{\gamma_1}\oplus\fg_{\gamma_2},\quad
(\gamma_1,\gamma_2)=
\begin{cases} 
\displaystyle 
(\sum_{j=i}^{n-1}\alpha_j,\sum_{j=i}^{n-2}\alpha_j+\alpha_n),& \textnormal{type (II)},\\
\displaystyle (\sum_{j=i}^{n-m}\alpha_j,\sum_{j=i}^{n-m}\alpha_j+2\sum_{j=n-m+1}^{n-2}\alpha_j+\alpha_{n-1}+\alpha_n), & \textnormal{type (III)}.
\end{cases}
$$

By a similar computation as in type B, $E_{\phi,\bgamma}=E_{\gamma_1}-E_{\gamma_2}$ is a basis of $\fu_{\phi,\bgamma}$. 
We define $\fu^c_{\bgamma}$ to be the subspace of $\fu_{\bgamma}$ generated by $E_{c,\bgamma}=E_{\gamma_1}+E_{\gamma_2}$:
\begin{equation}\label{eq:D (II),(III)-wts in fu_phi and fu^c}
\begin{split}
&\Phi_{T_\phi}^{\II}(\fu_\phi)=\Phi_{T_\phi}^{\II}(\fu^c)=\pi(\Phi_{T_\phi}^{\II}(\fu_0)),\quad 
\Phi_{T_\phi}^{\III}(\fu_\phi)=\Phi_{T_\phi}^{\III}(\fu^c)=\pi(\Phi_{T_\phi}^{\III}(\fu_0));\\
&\fu_{\phi,\bgamma}=k(E_{\gamma_1}-E_{\gamma_2}), \ \fu^c_{\bgamma}=k(E_{\gamma_1}+E_{\gamma_2}),\ \pi^{-1}(\bgamma)=\{\gamma_1,\gamma_2\}, \ \forall \bgamma\in \pi(\Phi_{T_\phi}^{\II}(\fu_0)\sqcup\Phi_{T_\phi}^{\III}(\fu_0)).
\end{split}
\end{equation}

\subsubsection{} 
We define $\fu^c \subseteq \fu_0$ by 
\begin{equation} \fu^c:=\bigoplus_{\bgamma\in \Phi_{T_\phi}(\fu_\phi)} \fu^c_{\bgamma}.
\end{equation} 
Thus, we have a decomposition of $T_\phi$-modules $\fu_0=\fu_\phi\oplus \fu^c$.

We collect an observation for future use. We consider roots of $\fg$ as characters of $T_\phi$ by restriction. From discussions on $\Phi_{T_\phi}(\fu^c)$ and Proposition \ref{p:G type v.s. weight type}, we can see:

\begin{prop}\label{p:U^c wts=T_phi basis} 
We can label elements of $\Phi_{T_\phi}(\fu^c)$ by $\Phi_{T_\phi}(\fu^c)=\{\overline{\gamma}_i\}_{i=1}^{n-m}$ such that for $1\le i\le n-m$, $\gamma_i\in \pi^{-1}(\overline{\gamma}_i)\subset \Phi(\fu_0)$ can be written as 
\[
\gamma_i=\chi_i +u_i\chi_{r_i}
\]
for an integer $r_i\in [n-m+1,n+1]$ and $u_i\in \{-1,0,1\}$.  
\end{prop}

Note that in the case of type B,C,D, $\chi_{r_i}$ is trivial on $T_{\phi}$ and that when $G=\SL_{n+1}$, $\chi_{r_i}|_{T_\phi}$ is nontrivial, because $T_\phi=\{(a_1,...,a_{n-m},a,...,a)\in T\}$.

\subsection{A complement for $U_\phi$ inside $U_0$} In this section, we give a group theoretic analogue of the above result and construct a complement $U^c$ for $U_\phi$ inside $U_0$. The main challenge is that $U_\phi$ is not, in general, normal; thus, $U_0/U_\phi$ is not a group.

\subsubsection{Root space decomposition of $U_\phi$}	  For each weight $\bgamma \in \Phi_{T_\phi}(\fu_\phi)$, let $U_{\bgamma}=\exp(\fu_{\bgamma})$ and $U_{\phi,\bgamma}=\exp(\fu_{\phi,\bgamma})$.
Then, we have an inclusion $U_{\phi,\bgamma}\subseteq U_{\bgamma}$. Let $U_{\phi,\bgamma}:=U_\phi\cap U_{\bgamma}$.  In view of root space decomposition of $\fu_\phi$  \eqref{eq:Tphi decomp of uphi}, we obtain an isomorphism of schemes
\begin{equation} 
\displaystyle U_\phi\simeq \prod_{\bgamma\in\Phi_{T_\phi}(\fu_\phi)}U_{\phi,\bgamma},
\end{equation}
where the product is defined with respect to some order on $\Phi_{T_\phi}(\fu_\phi)$.
In particular, we see that the resulting scheme is, up to isomorphism, independent of the chosen order.

\subsubsection{}  
Let $U^c_{\bgamma}$ the subscheme $\exp(\fu^c_{\bgamma})$ of $U_0$. Note that $T_\phi$ acts on $\fu^c_{\bgamma}$; thus, it also acts on $U^c_{\bgamma}$. 
Let $\prec$ be an order on $\Phi_{T_\phi}(\fu^c)$ and 
define
\[
U^c:=\prod_{\bgamma \in \Phi_{T_\phi}(\fu_c)} U^c_{\bgamma} \subset U_0,
\]
where the product is taken with respect to the order $\prec$. 
In general, $U^c$ is not a subgroup of $U_0$, but only a subscheme (depending on $\prec$), equipped with an action of $T_\phi$. 

\begin{prop}\label{l:U^cU_phi=U_L}\mbox{}
\textnormal{(i)}
The multiplication map on $G$ induces an isomorphism of schemes 
$U^c\times U_\phi\xrightarrow{\sim} U_0$. 

\textnormal{(ii)} 
We have a $T_\phi$-equivariant isomorphism of schemes $U^c\simeq U_0/U_\phi$. Note that $U_0/U_{\phi}$ is independent of the choice of $\prec$.  
\end{prop}
\begin{proof}
Let $V=U^c\times U_\phi$. Let $e$ be the unit in $U_0$, and by abuse of notation also the product of units $(e,e)\in U^c\times U_\phi$. 
In view of their definition, the differential map
$$
d_e\mathrm{m}:T_e V\rightarrow T_e U_0
$$
is an isomorphism. To prove the assertion, we use Luna map to show there are isomorphisms $V\simeq T_e V$ and $U_0\simeq T_e U_0$, whose composition with $d_e\mathrm{m}$ coincides with $\mathrm{m}$. By \cite[Theorem 13.41]{Milne}, it suffices to find an action of torus $\bGm$ on $U^c,U_\phi$ and $U_0$,  compatible with the multiplication map $\mathrm{m}$, fixing $e$, such that all of its weights on $\fu^c,\fu_\phi$ and $\fu_0$ are positive.

When $G$ is of type A or C, only type (I) $T_\phi$-weights appear \eqref{eq:A T_phi-wts in fu_L} and \eqref{eq:C T_phi-wts in fu_L} and $U^c,U_\phi$ and $U_0$ are all products of root subgroups for positive roots of $G$. 
In particular, they are normalized by $T$. Consider the adjoint action given by $\check{\rho}_G(\bGm)$, which has positive weights on $\fu^c,\fu_\phi$ and $\fu_0$. This provides the desired $\bGm$-action.

When $G$ is of type B and D, we have seen from \eqref{eq:B T_phi-wts in fu_L} and \eqref{eq:D T_phi-wts in fu_L} that type (I), (II), and (III) $T_\phi$-weights all appear, and $U^c,U_\phi$ and $U_0$ are only normalized by $T_\phi$ rather than $T$. However we can modify $\check{\rho}_G$ into a $T_\phi$-cocharacter that satisfies our requirement. 
Since we only concern properties on subgroups of $U_0$, we may assume $G$ is either $\SO_{2n+1}$ or $\SO_{2n}$. Observe that $T=\prod_{i=1}^n T_i$ and $T_\phi=\prod_{i=1}^{n-m}T_i$ in notations of \S\ref{ss:classical}, which allow us to define a projection $p:T\tra T_\phi$. 
We set $\lambda=\check{\rho}_G\circ p\in X_*(T_\phi)$. If $G=\SO_{2n+1}$, then $\lambda=\sum_{i=1}^{n-m}(n-i+1)\lambda_i$. Any $\gamma\in\Phi(\fu_0)$ is either $\chi_i$, or $\chi_i\pm\chi_{i+k}$. From Lemma \ref{l:nonzero initial}, we always have $i\leq n-m$ and therefore $\gamma(\lambda)>0$. 
Similarly, if $G=\SO_{2n}$, then $\lambda=\sum_{i=1}^{n-m}(n-i)\lambda_i$. Any $\gamma\in\Phi(\fu_0)$ is of the form $\chi_i\pm\chi_{i+k}$, where $i\leq n-m$. Thus we have $\gamma(\lambda)>0$ as well. Thus $\lambda$ provides the desired $\bGm$-action, which completes the proof.	
\end{proof}

\subsection{Orbits of the action of $T_\phi$}
Recall that we have a basis $E_{\bgamma}^c$ for the one-dimensional vector space $\fu^c_{\bgamma}$. 
For each subset $A\subseteq \Phi_{T_\phi}(\fu^c)$, define an element 
\[
u_A=\prod_{\bgamma\in A}\exp(E_{\bgamma}^c) \in U^c,
\] 
where the product is taken with respect to the order $\prec$. Let $U_A\subseteq U^c$ denote the conjugacy class of $u_A$ in $U^c$. 

\begin{prop} \label{p:T_phi conjugacy class in U^c}
\textnormal{(i)}
The images of $u_{A}$ and $U_A$ in $U_0/U_\phi$ are independent of $\prec$. 

\textnormal{(ii)} The map $A\mapsto U_{A}U_\phi$ induces a bijection 
\[
\textrm{\{subsets of $\Phi_{T_\phi}(\fu^c)$\}} \longleftrightarrow \textrm{\{$T_\phi$-orbits on $U^c$\}}. 
\]
\end{prop}
To prove the proposition, we need a lemma first. 
\begin{lem}\label{l:u^c semicomm} 
\textnormal{(i)} For any type (II), (III) weight $\bgamma\in\Phi_{T_\phi}(\fu_0)$ where $\pi^{-1}(\bgamma)=\{\gamma_1,\gamma_2\}$, $\gamma_1+\gamma_2$ is not a root. In particular, $U_{\gamma_1}$ and $U_{\gamma_2}$ commute, and $U_{\bgamma}=U_{\gamma_1}U_{\gamma_2}$.

\textnormal{(ii)} For any $\bgamma_1,\bgamma_2\in\Phi_{T_\phi}(\fu^c)$, if $\bgamma_1+\bgamma_2\in\Phi_{T_\phi}(\fu_0)$, then $\bgamma_1+\bgamma_2\in\Phi_{T_\phi}(\fu_\phi)-\Phi_{T_\phi}(\fu^c)$. In this case, $\bgamma_1+\bgamma_2+\bgamma\not\in\Phi_{T_\phi}(\fu_0)$ for any $\bgamma\in\Phi_{T_\phi}(\fu^c)$.
\end{lem}
\begin{proof}
(i) It follows from Proposition \ref{p:G type v.s. weight type} that type (II),(III) weights only occur in type B, D. 
Moreover, if $\pi(\gamma_1)=\pi(\gamma_2)$, then $\gamma_i=\alpha_k+\sum_{j=k+1}^n n_j^i\alpha_j$ for some $1\leq k\leq n-m$. Thus $\gamma_1+\gamma_2=2\alpha_k+\sum_{j=k+1}^n (n_j^1+n_j^2)\alpha_j$ is not a root in type B or D root system.

(ii) Since here we essentially only concern properties of root systems, we may assume $G$ is either $\SL_{n+1}, \SO_{2n+1}, \Sp_{2n}, \SO_{2n}$. From Proposition \ref{p:U^c wts=T_phi basis}, we may assume $\bgamma_1=\chi_p$, $\bgamma_2=\chi_q$ for $1\leq p,q\leq n-m$. Thus $\bgamma_1+\bgamma_2=\chi_p+\chi_q\not\in\Phi_{T_\phi}(\fu^c)$ as a character of $T_\phi$. Moreover, since any root is of the form $\chi_i$ or $\chi_i\pm\chi_j$, so is any weight in $\Phi_{T_\phi}(\fu_0)$. Thus for any $\bgamma=\chi_r$, $\bgamma_1+\bgamma_2+\bgamma=\chi_p+\chi_q+\chi_r\not\in\Phi_{T_\phi}(\fu_0)$.
\end{proof}

\begin{proof}
We are now ready to prove the proposition. It suffices to prove (i). For any $\bgamma_i\in\Phi_{T_\phi}(\fu^c)$ and $v_i\in\fu^c_{\chi_i}$, we set $u_i=\exp(v_i)$.
If $\bgamma_i+ \bgamma_j \not\in\Phi_{T_\phi}(\fu_0)$, then $u_i$ and $u_j$ already commute. 
Otherwise, from Lemma \ref{l:u^c semicomm}(ii) we have $\bgamma=\bgamma_i+\bgamma_j\in\Phi_{T_\phi}(\fu_\phi)$. Thus $v:=[v_i,v_j]\in\fu_{\phi,\bgamma}$. 
Moreover, we have $[v',v]=0$ for any $v'\in\fu^c$ and that $\exp(v)$ commutes with any $u\in U^c$. 
We obtain
$$
u_iu_j=(u_iu_ju_i^{-1})u_i=\exp(v_j+[v_i,v_j])u_i=u_j\exp(v)u_i=u_ju_i\exp(v)
$$
where $\exp(v)$ commutes with any element of $U^c$. 
By induction, if $\{j_1,\dots,j_k\}$ is a permutation of $\{i_1,\dots,i_k\}$, we deduce that 
$$
u_{i_1}u_{i_2}\cdots u_{i_k}=u_{j_1}u_{j_2}\cdots u_{j_k}u_\phi
$$
for some $u_\phi\in U_\phi$ that commutes with any element of $U^c$. This finishes the proof. 	
\end{proof}

\subsection{Proof of Theorem \ref{t:toric}} 
Our goal is to show that the action of $T_\phi$ on $U_0/U_\phi$ has an open dense orbit with finite stabilizer. In the previous subsections, we constructed a subscheme $U^c\subseteq U_0$ and a $T_\phi$-equivariant isomorphism of schemes $U^c\simeq U_0/U_\phi$. Moreover, we constructed a bijection between subsets $A\subseteq \Phi_{T_\phi}(\fu^c)$ and $T_\phi$-orbits on $U^c$. 

Let $\mathring{U}^c\subseteq U^c$ denote the $T_\phi$-orbit corresponding to the the maximal subset (i.e. $A=\Phi_{T_\phi}(\fu^c)$). We claim that $\mathring{U}^c$ is open dense with finite stabilizer. Indeed, as noted above, $U^c$ is a product of one dimensional subschemes $U^c_{\bgamma}$. Thus, $\mathring{U}^c$ then identifies with the subvariety of $U^c$ consisting of elements whose restriction to each of these one-dimensional subschemes is nontrivial. In other words, 
$\mathring{U}^c$ consists of elements of the form $ \prod_{\bgamma\in \Phi_{T_\phi}(\fu_\phi)} \exp(c_{\bgamma} E_{\bgamma}^c)$, where $c_{\bgamma}$ ranges over elements of $k^\times$. 
This is a dense open subset with trivial stabilizer. 
\hfill \qed

\section{Rigidity of hypergeometric data for classical groups: Proof of Theorem \ref{t:rigid}} 
\label{s:rigid}

\subsection{Stratification of the moduli stack $\Bun_{\cG'}$}
We keep the notation and assumption of Theorem \ref{t:rigid}. 
Recall that $\Bun_{\cG'}$ denotes the moduli stack of $\cG'$-bundles, defined by hypergeometric data \eqref{eq:gp schemes}.
Let $\tW=N_{G(F_\infty)}(T(F_\infty))/T(\cO_\infty)$ be the Iwahori-Weyl group and $\Omega$ the normalizer of $I$ which is a subgroup of $\tW$. 
The group $J=B_\phi P(1)\subset P$ is contained in the Iwahori $I=B_0 P(1)$. Let $s$ be a coordinate around $\infty$ and set $I^-=I^\opp\cap G(\overline{k}[s,s^{-1}])$. 
By \cite[Proposition 1.1]{HNY} and Birkhoff decomposition, we have a decomposition
$$
\Bun_{\cG'}(\overline{k})=I^-\backslash G(\overline{k}(\!(s)\!))/J=\bigsqcup_{\tw\in\tW}I^-\backslash I^-\tw I/J.
$$
As $I=B_0 P(1)=T U_0 P(1)$, $J=B_\phi P(1)=T_\phi U_\phi P(1)$, above decomposition can be written as
$$
\Bun_{\cG'}(\overline{k})=\bigsqcup_{\tw\in\tW}I^-\backslash I^-\tw U_0 J/J.
$$

Note that $T_\phi$ is contained in both $I^-$ and $J$. Thus $I^-\backslash I^-\tw uJ/J=I^-\backslash I^-\tw (\Ad_t u)J/J$ for any $u\in U_0$, $t\in T_\phi$. 
In view of Proposition \ref{p:T_phi conjugacy class in U^c}, we obtain
\begin{equation}\label{eq:Bun_cG' decomp}
\Bun_{\cG'}(\overline{k})=\bigsqcup_{\tw\in\tW}\bigcup_{A\subset\Phi_{T_\phi}(\fu^c)}
I^-\backslash I^-\tw u_{A} J/J.
\end{equation}
We obtain a surjection $\tW\times 2^{\Phi_{T_\phi}(\fu^c)}\tra\Bun_{\cG'}(\overline{k})$ given by $(\tw,A)\mapsto I^-\tw u_AJ$.
In particular, $\Bun_{\cG'}(\overline{k})$ has countably many objects up to isomorphism.

\subsubsection{A family of subtori} \label{ss:subtori}
We introduce a family of subtori in $T_\phi$, which would be contained in the automorphism group of certain $\cG'$-bundles. 

Given $\bgamma\in\Phi_{T_\phi}(\fu^c)$, we define
\begin{equation}\label{eq:T_bgamma}
T_{\bgamma}:=\biggl(\bigcap_{\alpha\in\Phi_{T_\phi}(\fu^c)-\{\bgamma\}}\ker\alpha\biggr)^\circ\subset T_\phi.
\end{equation}
It is clear that these subtori belong to the family $\mathcal{S}$ defined in \eqref{eq:S subtori}.

From Proposition \ref{p:U^c wts=T_phi basis}, we can see that $T_{\bgamma}\simeq\bGm$. More specifically, if $G=\SO_{2n+1}, \Sp_{2n}, \SO_{2n}$ and $\bgamma=\chi_i+u_i\chi_{r_i} \in \Phi_{T_\phi}(\fu^c)$, then $T_{\bgamma}=T_i$; if $G=\SL_{n+1}$ and $\bgamma=\chi_i-\chi_{r_i}\in\Phi_{T_\phi}(\fu^c)$, then $T_{\bgamma}=\{(a,...,a,a^{-n},a,...,a)|a\in k^\times\}$, where $a^{-n}$ is the $i$-th diagonal entry.

\subsection{Automorphisms of $\cG'$-bundles} \label{ss:AutE}
We first prove Theorem \ref{t:rigid}(i). 
To do so, let $x_I$ be the barycenter of alcove $C_I$ associated to the Iwahori $I=B_0P(1)$. By \cite[Proposition 27]{KY}, the set of simple affine roots $\Delta^{\aff}(I)$ associated to $C_I$ is a disjoint union of the set $\Delta(G_0)$ of simple roots of $L$ and the set $\wt^-(V)$ of lowest weights in $L$-representation $V$ with respect to $B_I$. Thus $\tw\in\Omega$ if and only if $(\tw\alpha)(x_I)>0$ for any $\alpha\in\Delta(G_0)\sqcup\wt^-(V)$.

Note that the subgroup $Z_{G,\phi}$ \eqref{sss:stabiliser} of the center is contained in the automorphism group of each $\mathcal{G}'$-bundle. Theorem \ref{t:rigid}(i) follows from the following proposition. 

\begin{prop}\label{p:wildAut}
For $\cG'$-bundle $\cE=I^-\tw u_A J$ where $A\subset\Phi_{T_\phi}(\fu^c)$, its automorphism group $\Aut(\cE)=J\cap\Ad_{(\tw u_A)^{-1}}I^-$ contains following subschemes:

\textnormal{(i)} If $A\neq\Phi_{T_\phi}(\fu^c)$, then for any $\bgamma\in\Phi_{T_\phi}(\fu^c)-A$, $T_{\bgamma}\subset T_\phi\cap\Aut(\cE)$.

\textnormal{(ii)} If $A=\Phi_{T_\phi}(\fu^c)$, $\tw\not\in\Omega$ and there exists $\alpha\in\wt^-(V)$ such that $(\tw\alpha)(x_I)<0$, then $\Ad_{\ru^{-1}}U_\alpha\subset P(1)\cap\Aut(\cE)$.

\textnormal{(iii)} If $A=\Phi_{T_\phi}(\fu^c)$, $\tw\not\in\Omega$ and there exists $\alpha\in\Delta(G_0)$ such that $(\tw\alpha)(x_I)<0$, then there exists $\bgamma\in\Phi_{T_\phi}(\fu^c)$ and a connected subscheme of the form $H_{\tw}=\{tu(t)|t\in T_{\bgamma}, u(t)\in U_\phi, u(1)=1\}\simeq T_{\bgamma}$ such that $H_{\tw}\subset B_\phi\cap\Aut(\cE)$.

\textnormal{(iv)} If $A=\Phi_{T_\phi}(\fu^c)$ and $\tw\in\Omega$, then $\Aut(\cE)=Z_{G,\phi}$.
\end{prop}
\begin{proof}
(i): Note that 
$$
T_\phi\cap\Ad_{u_A^{-1}}T=T_\phi\cap\Ad_{(\tw u_A)^{-1}}T\subset J\cap\Ad_{(\tw u_A)^{-1}}I^-=\Aut(\cE).
$$
	
If $A\neq\Phi_{T_\phi}(\fu^c)$, let $\bgamma\in\Phi_{T_\phi}(\fu^c)-A$.
From \eqref{eq:T_bgamma}, any $t\in T_{\bgamma}\subset T_\phi$ satisfies $\alpha(t)=1$, $\forall\alpha\in A$, i.e. $u_A t u_A^{-1}=t$. We obtain
$$
T_{\bgamma}\subset T_\phi\cap\Ad_{u_A^{-1}}T\subset\Aut(\cE).
$$
	
(ii): In this case $u_A=\ru$. Since $(\tw\alpha)(x_I)<0$ for $\alpha\in\wt^-(V)$, we have $\Ad_{\tw}U_\alpha\subset\Ad_{\tw}P(1)\cap I^-=\Ad_{\tw\ru}P(1)\cap I^-$. Thus
$$
\Ad_{\ru^{-1}}U_\alpha\subset P(1)\cap\Ad_{(\tw\ru)^{-1}}I^-\subset\Aut(\cE).
$$
	
(iii): For the $\alpha\in\Delta(G_0)$ that satisfies $(\tw\alpha)(x_I)<0$, we have $U_\alpha\subset\Ad_{\tw^{-1}}I^-$. Thus
\begin{equation} \label{Bphicap Ad Aut}
B_\phi\cap\Ad_{\ru^{-1}}(TU_\alpha)\subset J\cap\Ad_{(\tw\ru)^{-1}}I^-=\Aut(\cE).
\end{equation}

On the other hand, in view of \ref{ss:G_0 simple roots}, we can see that $\alpha\in\Delta(G_0)$ is of the form $\alpha=\chi_i$ or $\chi_i\pm\chi_j$ where $i<j$, $1\leq i\leq n-m$. From Proposition \ref{p:U^c wts=T_phi basis}, there is a unique $\bgamma\in\Phi_{T_\phi}(\fu^c)$ such that $\bgamma=\chi_i$ or $\chi_i-\chi_{r_i}$. 
In the following, we construct a morphism $u:T_{\bgamma}\ra U_\phi$ and define $H_{\tw}$ by
$$H_{\tw}=\{tu(t)|t\in T_{\bgamma}, u(t)\in U_\phi, u(1)=1\}\subset B_\phi\cap\Aut(\cE).$$
By \eqref{Bphicap Ad Aut}, it suffices to require $H_{\tw}\subset\Ad_{\ru^{-1}}(TU_\alpha)$, i.e. for $\forall t\in T_{\bgamma}$, $\exists u_\alpha\in U_\alpha$ such that
\begin{equation}\label{eq:u(t)}
t^{-1}\ru t=u_\alpha\ru u(t)^{-1}.
\end{equation}
The construction of $u(t)$ is based on a discussion on the restriction of $\alpha$ to $T_\phi$, i.e. the $T_\phi$-weight $\pi(\alpha)\in\Phi_{T_\phi}(\fu_0)$ \eqref{eq:Tphi weight}.
	
Case a): when $\pi(\alpha)\in\Phi_{T_\phi}(\fu^c)$.
	
From Proposition \ref{p:U^c wts=T_phi basis}, we have $\pi(\alpha)=\alpha|_{T_\phi}=\bgamma$ for some $\bgamma\in\Phi_{T_\phi}(\fu^c)$. 	
We choose an order $\prec$ on $\Phi_{T_\phi}(\fu^c)$ such that $\bgamma$ is the maximal one. Then the LHS of \eqref{eq:u(t)} is
$$
t^{-1}\ru t=\biggl(\prod_{\bgamma'\in\Phi_{T_\phi}(\fu^c)-\{\bgamma\}}\exp(E_{c,\bgamma'})\biggr)\exp(\bgamma(t^{-1})E_{c,\bgamma}).
$$
	
For the RHS of \eqref{eq:u(t)}, we set $u_\alpha=\exp(\lambda E_\alpha)$ for some $\lambda$ to be determined. 
From Lemma \ref{l:u^c semicomm}(ii), we know that for any
$\bgamma'\in\Phi_{T_\phi}(\fu^c)$ and $u\in U_{\bgamma'}$, we have $u_\alpha u=uu_\alpha u'$ for some $u'\in U_\phi$ that commutes with $U_{\bar{\alpha}}$ for every $\bar{\alpha}\in\Phi_{T_\phi}(\fu^c)$. We obtain
$$
u_\alpha \ru=\ru u_\alpha u_\phi=\biggl(\prod_{\bgamma'\in\Phi_{T_\phi}(\fu^c)-\{\bgamma\}}\exp(E_{c,\bgamma'})\biggr)\exp(E_{c,\bgamma})u_\alpha u_\phi,
$$
for some $u_\phi=u_\phi(t)\in U_\phi$ that commutes with $U_{\bar{\alpha}}$ for any $\bar{\alpha}\in\Phi_{T_\phi}(\fu^c)$.
	
Comparing above two formulas, the equality \eqref{eq:u(t)} amounts to requiring
$$
\exp(\bgamma(t^{-1})E_{c,\bgamma})=\exp(E_{c,\bgamma})u_\alpha u_\phi u(t)^{-1},
$$
for some $u(t)\in U_\phi$. This is the same as 
$$
\exp((1-\bgamma(t^{-1}))E_{c,\bgamma})\exp(\lambda E_\alpha)=u(t)u_\phi^{-1}\in U_\phi.
$$
Moreover, if $\alpha$ is type (I) root then $E_{c,\bgamma}=E_\alpha$; if $\alpha$ is type (II) or (III) root, Lemma \eqref{l:u^c semicomm}.(1) implies that $[E_{c,\bgamma},E_\alpha]=0$. Thus we always have $[E_{c,\bgamma},E_\alpha]=0$. The question reduces to find $\lambda=\lambda(t)$ such that
\begin{equation}\label{eq:mu}
(1-\bgamma(t^{-1}))E_{c,\bgamma}+\lambda E_\alpha\in\fu_\phi.
\end{equation}
	
When $\alpha$ is a type (I) root, $E_{c,\bgamma}=E_\alpha$. We set $\lambda=\bgamma(t^{-1})-1$ and 
$u(t)=u_\phi(t)$.
Then the equation \eqref{eq:u(t)} is satisfied, and $H_{\tw}$ defined by $T_{\bgamma}$ and $u(t)$ gives the desired subscheme of $\Aut(\cE)$.
	
When $\alpha$ is of type (II) or (III), assume $\pi^{-1}(\bgamma)=\{\alpha,\alpha'\}$. Recall from \eqref{eq:B (III)-wts in fu_phi and fu^c} and \eqref{eq:D (II),(III)-wts in fu_phi and fu^c} that $E_{c,\bgamma}=E_{\alpha}+E_{\alpha'}$ and $E_{\phi,\bgamma}=E_{\alpha}-E_{\alpha'}\in\fu_\phi$. We let $\lambda=2(\bgamma(t^{-1})-1)$, then \eqref{eq:mu} becomes
$$
(1-\bgamma(t^{-1}))E_{c,\bgamma}+\lambda E_\alpha=(\bgamma(t^{-1})-1)(E_\alpha-E_{\alpha'})=(\bgamma(t^{-1})-1)E_{\phi,\bgamma}\in\fu_\phi.
$$
If we set
$u(t)=\exp((\chi_i(t^{-1})-1)E_{\phi,\bgamma})u_\phi(t)$, then equation \eqref{eq:u(t)} is satisfied.
Thus $H_{\tw}$ defined from $T_{\bgamma}$ and $u(t)$ gives the desired subscheme of $\Aut(\cE)$.
	
Case b): when $\pi(\alpha)\in\Phi_{T_\phi}(\fu_\phi)-\Phi_{T_\phi}(\fu^c)$.
	
In this case, $\alpha\in\Phi_{T_\phi}(\fu^c)$ is a type (I) root.
In view of \ref{sss:A wts in fu_phi and fu^c}-\ref{sss:D wts in fu_phi and fu^c} and \ref{ss:G_0 simple roots}, any root $\alpha\in\Phi(\fu_\phi)\cap\Delta(G_0)$ is of the form $\alpha=\sum_{j=i}^{i+d-1}\alpha_j=\chi_i-\chi_{i+d}$ where $i+d\leq n-m$.
From Proposition \ref{p:U^c wts=T_phi basis}, $\chi_{i+d}$ (resp. $\chi_{i+d}-\chi_{r_{i+d}}$ if $G=\SL_{n+1}$) is the unique $T_\phi$-weight $\bar{\beta}\in\Phi_{T_\phi}(\fu^c)$ satisfying $\alpha+\bar{\beta}\in\Phi_{T_\phi}(\fu_0)$. Let $u_\alpha=\exp(\lambda E_\alpha)$, above discussion implies that
$$
u_\alpha \ru u_\alpha^{-1}=(\prod_{\bgamma'\neq \bgamma,\bar{\beta}}\exp(E_{c,\bgamma'}))\exp(E_{c,\bgamma})\exp(E_{c,\bar{\beta}}+\lambda[E_\alpha,E_{c,\bar{\beta}}]).
$$
Here we choose an order on $\Phi_{T_\phi}(\fu^c)$ such that $\bgamma,\bar{\beta}$ are the last two of them. 
Any $\beta\in\pi^{-1}(\bar{\beta})$ is either $\beta=\chi_{i+d}\pm\chi_j$ for some $j>n-m$, or $\chi_{i+d}$. Thus $\alpha+\beta\in\Phi(\fu_0)$ is still a $T$-root, and $[E_\alpha,E_{c,\bar{\beta}}]=E_{c,\bgamma}$.
By Lemma \ref{l:u^c semicomm}.(2), $[E_{c,\bgamma},E_{c,\bar{\beta}}]$ is either $0$, or an element of $\fu_\phi$ that commutes with $\fu^c$. 
Thus we can use Baker–Campbell–Hausdorff formula\footnote{Here all the Lie algebra elements appear are nilpotent, contained in the same Borel subalgebra of $\fg_0$. Thus exponential map is a regular morphism, and the Baker-Campbell-Hausdorff formula is an equality of regular functions, thus applicable.} to get
\begin{align*}
\exp(\lambda E_{c,\bgamma})\exp(E_{c,\bar{\beta}})
&=\exp(\lambda E_{c,\bgamma}+E_{c,\bar{\beta}}+\frac{1}{2}[\lambda E_{c,\bgamma},E_{c,\bar{\beta}}])\\
&=\exp(E_{c,\bar{\beta}}+\lambda E_{c,\bgamma})\exp(\frac{\lambda}{2}[E_{c,\bgamma},E_{c,\bar{\beta}}]),\\
\Rightarrow\exp(E_{c,\bar{\beta}}+\lambda E_{c,\bgamma})=&\exp(\lambda E_{c,\bgamma})\exp(E_{c,\bar{\beta}})\exp(-\frac{\lambda}{2}[E_{c,\bgamma},E_{c,\bar{\beta}}]).
\end{align*}
We set $u_\phi=\exp(-\frac{\lambda}{2}[E_{c,\bgamma},\lambda E_{c,\bar{\beta}}])\in U_\phi$. In summary, we have 
$$
u_\alpha \ru u_\alpha^{-1}=(\prod_{\bgamma'\neq \bgamma,\bar{\beta}}\exp(E_{c,\bgamma'}))\exp((1+\lambda)E_{c,\bgamma})\exp(E_{c,\bar{\beta}})u_\phi.
$$
	
Recall that our goal is to construct $u(t)$ satisfying \eqref{eq:u(t)}. We set $\lambda=\bgamma(t^{-1})-1$ for $t\in T_{\bgamma}$, and 
$$
u(t)=u_\phi u_\alpha=\exp(-\frac{\lambda}{2}[E_{c,\bgamma},\lambda E_{c,\chi_{i+d}}])\exp(\lambda E_\alpha).
$$
In this way, we obtain identity \eqref{eq:u(t)}:
$$
u_\alpha\ru u(t)^{-1}=(\prod_{\bgamma'\neq \bgamma,\bar{\beta}}\exp(E_{c,\bgamma'}))\exp(\bgamma(t^{-1})E_{c,\bgamma})\exp(E_{c,\bar{\beta}})=t^{-1}\ru t.
$$
Thus $H_{\tw}$, defined by $T_{\bgamma}$ and $u(t)$, gives rise to the desired subscheme of $\Aut(\cE)$. This completes the proof of Proposition \ref{p:wildAut}(iii). 
	
(iv): Since $\tw\in\Omega$ normalize $I^-$, it suffices to show $\Ad_{\ru}J\cap I^-=Z_{G,\phi}$. Note that $J=B_\phi P(1)\subset I$ and $\ru\in U^c\subset I$. We have
$$
\Ad_{\ru}J\cap I^-\subset I\cap I^-=I\cap I^-=T.
$$
	
Then, we deduce
$$
\Ad_{\ru}J\cap I^-=(\Ad_{\ru}B_\phi)P(1)\cap T=\Ad_{\ru}B_\phi\cap T.
$$
	
For any $\Ad_{\ru}tu=t((t^{-1}\ru t)u\ru^{-1})\in \Ad_{\ru}B_\phi\cap T$ with $t\in T_\phi, u\in U_\phi$, we deduce $(t^{-1}\ru t)u=\ru$. 
Since $U^c$ is normalized by $T_\phi$, we deduce from Proposition \ref{l:U^cU_phi=U_L} that $u=1$, $t^{-1}\ru t=\ru$. From the definition of $\ru$, this tells us that $\bar{\gamma}(t)=1$ for all $\bar{\gamma}\in\Phi_{T_\phi}(\fu^c)$. Thus Proposition \ref{p:U^c wts=T_phi basis} implies $t\in Z_{G,\phi}$. 
This shows $\Aut(\cE)=\Ad_{\ru}B_\phi\cap T=Z_{G,\phi}$.
\end{proof}

\subsection{Proof of Theorem \ref{t:rigid}(ii)}
We now conclude the proof by showing that a $\cG'$-bundle is relevant if and only if its stabilizer is finite.
Let $\cE=I^-\tw u_A J$ be any $\cG'$-bundle. The automorphism group of $\cE$ is given by 
$$
\Aut(\cE)\simeq \Ad_g J\cap I^-\simeq J\cap \Ad_{g^{-1}}I^-.
$$ 
Its restrictions to $\cO_0$ and $\cO_\infty$ define an embedding 
\begin{equation}\label{eq:res auto grp}
\Aut(\cE)\simeq J\cap \Ad_{g^{-1}}I^-\hra I^\opp\times J=K_S,\ a\mapsto (\Ad_g a,a).
\end{equation}

Recall that $\cE$ is called \textit{relevant} (Definition \ref{d:relevant}) if the pullback of $\gamma_S=\cL_\delta\boxtimes\cL_\mu$ to $\Aut(\cE)^\circ$ via the above map is a constant sheaf.

\subsubsection{If $A\neq\Phi_{T_\phi}(\fu^c)$}	
Suppose $\cE$ is relevant. From Proposition \ref{p:wildAut}(i), the Frobenius trace of $\gamma_S$ has to be constant on $T_{\bgamma}\subset T_\phi\cap\Aut(\cE)$ via \eqref{eq:res auto grp}:
$$
\delta(\tw t\tw^{-1})\rho(t)=1, \quad \forall\ t\in T_{\bgamma}.
$$
This contradicts to the assumption that $\delta$ and $\rho$ are in general position (Definition \ref{d:general position}).
Therefore the bundle $\cE$ is irrelevant.

\subsubsection{If $A=\Phi_{T_\phi}(\fu^c)$ and $(\tw\alpha)(x_I)<0$ for some $\alpha\in\wt^-(V)$} 
Suppose $\cE$ is relevant. From Proposition \ref{p:wildAut}(ii), the Frobenius trace of $\gamma_S$ on $\Ad_{\ru^{-1}}U_\alpha\subset P(1)\cap\Aut(\cE)$ is constant, i.e.
$$
\psi(\phi(\ru^{-1}v\ru))=1,\quad \forall v\in U_\alpha
$$

Since $\psi$ is nontrivial, we deduce that $\phi(\ru^{-1} v\ru)=0$, $\forall v\in\fg_\alpha =\Lie(U_\alpha) \subset \fg_1$. The same argument implies that for any $u$ in the $T_\phi$-orbit of $\ru$, we have $\phi(u^{-1} v u)=0$. Recall that the $T_\phi$-orbits $\mathring{U}:=\Ad_{T_\phi}\ru$ is open in $U_\prec$.
Moreover, since $\phi$ is stabilized by $U_\phi$, we have $\phi(u^{-1} v u)=0$, $\forall v\in\fg_\alpha,\ u\in \mathring{U}U_\phi$. 
Since $\mathring{U}U_\phi$ is an open dense subgroup of $U_0=U^c U_\phi$ (Proposition \ref{l:U^cU_phi=U_L}), we have 
$$
\phi(u^{-1}vu)=0,\quad \forall v\in\fg_\alpha, u\in U_0,\quad \textnormal{and}\quad  \phi(\ad_u v)=0, \quad \forall u\in\fu_0, v\in\fg_\alpha.
$$ 
Suppose this $\alpha$ is the lowest weight of subrepresentation $V_i\subset V$. 
Then we have $\phi(v)=0$, $\forall v\in V_i$. 
Since $\phi$ is generic on $\fm_1$, we obtain a contradiction with Proposition \ref{p:Levi nonzero intersects subrepns}. Therefore the bundle $\cE$ is irrelevant.

\subsubsection{If $A=\Phi_{T_\phi}(\fu^c)$ and $(\tw\alpha)(x_I)<0$ for some $\alpha\in\Delta(G_0)$} 
Suppose $\cE$ is relevant. From Proposition \ref{p:wildAut}(iii), the Frobenius trace of the restriction of $\gamma_S$ to $H_{\tw}$ is constant:
$$
\delta(\tw\ru tu(t)\ru^{-1}\tw^{-1})\rho(tu(t))=\delta(\tw t\tw^{-1})\rho(t)=1,\quad \forall t\in T_{\bgamma}.
$$
This contradicts to the assumption that $\delta$ and $\rho$ are in general position. Therefore the bundle $\cE$ is irrelevant.

\subsubsection{If $A=\Phi_{T_\phi}(\fu^c)$ and $\tw\in\Omega$}
From Proposition \ref{p:wildAut}, $\Aut(\cE)^{\circ}=1$. Thus $\gamma_S$ is constant on $\Aut(\cE)^{\circ}$ and $\cE$ is relevant. This completes the proof of Theorem \ref{t:rigid}. \hfill \qed

\section{Study of local and global Hitchin maps: Proof of Propositions \ref{p:j^+ image} and \ref{p:global Hitchin diagram}}\label{s:pfs II}
\subsection{Proof of Proposition \ref{l:J}} \label{ss:J}
(i) Note that $\fg_1=\fm_1\oplus \fk$ and that the $T_\phi$-action preserves $\fk$ and is trivial on $\fm_1$. 
It suffices to show that the action of $U_\phi\subset L$ (resp. $\fu_\phi$) on $V$ preserves $\fk\subset\fg_1\simeq V$. 

Recall that $\fu_\phi$ has a decomposition \eqref{eq:Tphi decomp of uphi} into $T_\phi$-weight subspaces. 
By assumption, only type (I) roots exist (Proposition \ref{p:G type v.s. weight type}). 
For $\bgamma\in\Phi_{T_\phi}(\fu_\phi)$, $\bgamma$ is the restriction of of a unique root $\gamma\in\Phi(\fu_0)$ to $T_\phi$ and the space $\fu_{\phi,\bgamma}=\fg_{\gamma}$ is generated by $E_{\phi,\bgamma}=E_{\gamma}$. 
For any $\beta\in\Phi(\fg_1)$, if $\gamma+\beta\in\Phi_G$, then $\gamma+\beta\in\Phi(\fg_1)=\Phi(\fk)\sqcup\Phi(\fm_1)$. 
Therefore, if $\gamma+\beta=\alpha\in\Phi(\fm_1)$, then $\gamma-\alpha=-\beta$ is a root. However, this fact contradicts to Lemma \ref{l:roots in U_phi}.
Therefore the action of $E_{\gamma}$ sends $\fg_1$ into $\fk$ and preserves $\fk$. Assertion (i) follows.

Assertion (ii) follows from (i) and the definition. 
For $\mu=\rho\times \psi \phi$ \eqref{eq:defmu}, $\rho$ (resp. $\phi$) is trivial on $U_\phi$ (resp. $\fk$). Then assertion (iii) follows. \hfill \qed

\subsection{Proof of Proposition \ref{p:good stack}}\label{ss:good stack}
Recall \cite[\S~1.1.1]{BD} that a stack $\mathcal{Y}$ is good if
\begin{equation}\label{eq:good stack}
\mathrm{codim}\{y\in\mathcal{Y}|\dim\Aut(y)=n\}\geq n,\quad \forall n>0.
\end{equation}

Since $\Bun_\cG$ is $\overline{J}$-torsor over $\Bun_{\cG'}$, it suffices to show the assertion for $\Bun_{\cG'}$. 
Since $\dim\Bun_{\cG'}=0$ \eqref{r:numerical criteria}, it suffices to show the following fact:

Claim: for any $n>0$, there are finitely many isomorphism classes of $\cG'$-bundles $\cE\in\Bun_{\cG'}$ satisfying $\dim\Aut(\cE)=n$. 

Let $y\in\Bun_{\cG'}(\overline{k})$ be a $\cG'$-bundles and $I^-\backslash I^-\tw u_A J/J$ the associated double quotient in \eqref{eq:Bun_cG' decomp}. Its automorphism group is
$$
\Aut(y)\simeq I^-\cap \tw u_A J u_A^{-1}\tw^{-1}.
$$
Since $u_A\in U^c\subset U_0$ and $J=B_\phi P(1)$, we have $P(1)\subset u_A J u_A^{-1}\subset I=B_0 P(1)$. Thus we have
$$
I^-\cap \tw P(1)\tw^{-1}\subset \Aut(y)\subset I^-\cap \tw I\tw^{-1}=T(I^-\cap\tw I(1)\tw^{-1}),
$$

If $\ell(\tw)$ denotes the length of $\tw$ in $\tW$, then $\dim(I^-\cap \tw I(1)\tw^{-1})=\ell(\tw)$ \cite[ Theorem 7]{Fal05}. We deduce from above relation that 
$$
\dim\Aut(y)\leq \dim T+\ell(\tw).
$$

On the other hand, $I^-\cap\tw I(1)\tw^{-1}=I^-\cap\tw U_0P(1)\tw^{-1}$, $I^-\cap\tw U_0\tw^{-1}$ and $I^-\cap\tw P(1)\tw^{-1}$ are finite dimensional groups normalized by $T$. Thus they are products of affine root subgroups, and we have
$$
I^-\cap\tw I(1)\tw^{-1}=(I^-\cap\tw U_0\tw^{-1})(I^-\cap\tw P(1)\tw^{-1})\subset U_0(I^-\cap\tw P(1)\tw^{-1}),
$$
which implies
$
\dim(I^-\cap\tw P(1)\tw^{-1})\geq \dim(I^-\cap\tw I(1)\tw^{-1})-\dim U_0=\ell(\tw)-\dim U_0.
$

From the above discussion, we obtain
$$
\ell(\tw)-\dim U_0\leq \dim\Aut(y)\leq \ell(\tw)+\dim T.
$$

Recall \eqref{eq:Bun_cG' decomp} that for a fixed $\tw$, there are at most $2^{|\Phi_{T_\phi}(\fu^c)|}$ isomorphism classes of $\cG'$-bundles of the form $I^-\backslash I^-\tw u_A J/J$. Together with the above inequality, we obtain the claim and finishes the proof. \hfill \qed

\subsection{Proof of Propositions \ref{p:j^+ image}}
\label{ss:ProofHitchin}
We first prove the inclusion in one direction of assertion (i). 

\begin{lem}\label{l:j^+ image range}
We have $\cA_{\fj^+}\subseteq\Hit_{\fj^+}$. 
\end{lem}
\begin{proof}
Let $K'/K$ be a finite extension of degree $d$ and $u$ a uniformizer of $K'$ such that $s=u^d$. 
We consider the morphism $DK:\fg^*\otimes\omega_K\ra\fg^*\otimes\omega_{K'}$ defined in \cite[Proposition 10]{Zhu} and denote by $DK^*$ the composition of $DK$ with the isomorphism $\fg^*\otimes\omega_{K'}\xrightarrow{\sim} \fg\otimes\omega_{K'}$ induced by the Killing form.
We have a decomposition:
\begin{equation}\label{eq:decomposition j+}
\fj^+=\fu_\phi\oplus\fk\oplus\fp(2),
\end{equation}
where $\fk=\oplus_{\alpha\in \Phi(\fg_1)-\Phi(\fm_1)} \fg_{\alpha}$ satisfies $\fm_1\oplus\fk=V$. 
Then we deduce
\begin{equation}\label{eq:DK* of j+perp}
DK^*((\fj^{+})^{\perp})=(\fm_{-1}u^{-1}\oplus\fu_\phi^\perp\oplus\prod_{j\geq 1}\fg_j u^j)\frac{\td u}{u},
\end{equation}
where $\fu_\phi^\perp$ is the orthogonal complement of $\fu_\phi$ in $\fg_0$ with respect to Killing form.

Inside $\omega_{K'}^{d_i}$, for every $n\in \mathbb{Z}$ we know from \cite[Lemma 11]{Zhu} \footnote{There is a typo in \textit{loc.cit}.} that
\[
\omega_{\cO_{K'}}^{d_i}(n)\bigcap \omega_{K}^{d_i}=\omega_{\cO_K}^{d_i}\biggl(d_i- \lceil \frac{d_i-n}{d}\rceil\biggr).
\]
We reduce to show that for $X\frac{\td u}{u}\in DK^*((\fj^{+})^{\perp})$,
\begin{equation} \label{eq:local Hitchin in K'}
h^{cl}(X\frac{\td u}{u})\in \bigoplus_{d_i<d}\omega_{\cO_{K'}}^{d_i}(d_i+d-1)\oplus\bigoplus_{d_i\geq d}\omega_{\cO_{K'}}^{d_i}(d_i+d). 
\end{equation}

Recall that $h^{cl}(X\frac{\td u}{u})$ is calculated by coefficients of the characteristic polynomial, and that $d_i$-th component of $h^{cl}(X\frac{\td u}{u})$ is a sum of products of $d_i$ matrix entries in $X\frac{\td u}{u}$. 

For such a $d_i$ product, entries of $\fm_{-1}u^{-1}$ would contribute at most $d$ terms in the product. 
Then we deduce that the degree of $u$ in this product is $\le -2d_i$ if $d_i<d$ and $\le -d_i-d$ if $d_i\ge d$. The inclusion \eqref{eq:local Hitchin in K'} follows. 
\end{proof}

\subsubsection{} \label{sss:global Hitchin Gm}
Following \cite[\S6]{BK}, we use the global Hitchin map to prove Proposition \ref{p:j^+ image}. 
Let $\ell$ be an integer and consider a set $S$ of $\ell+1$ distinct points on the projective line: $x:=\infty$, $y_1:=0$, $y_2,...,y_\ell\in \bP^1-\{0,\infty\}$. Let $I$ be an Iwahoric subgroup. We consider a group scheme $\cG_\ell$ over $\bP^1$ defined by
$$
\cG_\ell|_{\bP^1-S}=G\times(\bP^1-S);\ \cG_\ell(\cO_x)=J^+;\ \cG_\ell(\cO_{y_i})=I,\ 1\leq i\leq \ell.
$$ 
We denote the global Hitchin base over $\bP^1-S$ by $\Hit(\bP^1-S)$ and the global Hitchin map by 
$$
H^{cl}:T^*\Bun_{\cG_\ell}\ra\Hit(\bP^1-S).
$$
Let $\cA_{\cG_\ell}$ be the Zariski closure of the image $H^{cl}(T^*\Bun_{\cG_\ell})$ in $\Hit(\bP^1-S)$. 
We define a closed subscheme $\Hit(\bP^1)_{\cG_\ell}$ of $\Hit(\bP^1-S)$ by
\begin{align*}
&\Hit(\bP^1)_{\cG_\ell}= 
\Hit(\bP^1-S)\times_{D_\infty^\times}\Hit_{\fj^+}\times_{D_{y_1}^\times}\Hit_\fI\times\cdots\times_{D_{y_\ell}^\times}\Hit_\fI & \\
&\simeq\bigoplus_{d_i<d}\Gamma(\bP^1,\Omega^{d_i}(d_i x+(d_i-1)(y_1+\cdots+y_\ell)))\oplus\bigoplus_{d_i\geq d}\Gamma(\bP^1,\Omega^{d_i}((d_i+1) x+(d_i-1)(y_1+\cdots+y_\ell))). &
\end{align*}

\begin{lem} \label{eq:global Hitcin image}
For every $\ell\geq 1$, we have $\cA_{\cG_\ell}=\Hit(\bP^1)_{\cG_\ell}$.
\end{lem}
\begin{proof}
Recall that $\cA_\fI=\Hit_\fI$ \eqref{Hitchin map I}. Then we have
\[
\cA_{\cG_\ell}\simeq \Hit(\bP^1-S)\times_{D_\infty^\times}\cA_{\fj^+}\times_{D_{y_1}^\times}\Hit_\fI\times\cdots\times_{D_{y_\ell}^\times}\Hit_\fI.
\]
By Lemma \ref{l:j^+ image range} and the above formula, we have $\cA_{\cG_\ell}\subseteq\Hit(\bP^1)_{\cG_\ell}$. To obtain another direction, we first prove the case $\ell=1$. 
We show that the subspace $Z$ \eqref{eq:def of Z} of $\Hit_{\fj^+}$ is contained in $\cA_{\fj^+}$ in subsection \ref{l:Z dense}.
Then we have closed embeddings 
\begin{equation}\label{eq:cG_1 case}
\Hit(\bGm)\times_{D_\infty^\times}Z\times_{D_0^\times}\Hit(D_0)_\fI
\hookrightarrow\cA_{\cG_1}
\hookrightarrow\Hit(\bP^1)_{\cG_1}.
\end{equation}
Since $\sharp \{d_i\geq d\}=n-m+1$, we have
\[
\Hit(\bP^1)_{\cG_1}\simeq \bigoplus_{d_i\geq d}\Gamma(\bP^1,\Omega^{d_i}((d_i+1) \infty+(d_i-1)0))\simeq \bA^{n-m+1}. 
\]
For each $d\leq d_i\leq h_G$, let $t=s^{-1}$, we have
$
\bC t(\frac{\td t}{t})^{d_i}=\bC s^{-1}(\frac{\td s}{s})^{d_i} \subset Z.
$ 
Therefore, $\bC s^{-1}(\frac{\td s}{s})^{d_i}$ is contained in the LHS of \eqref{eq:cG_1 case}. By comparing dimension, we deduce that the closed embeddings \eqref{eq:cG_1 case} are isomorphisms, which proves the case $\ell=1$. 

In the general case, the proposition follows from the following isomorphisms:
$$
\cA_{\cG_\ell}
\simeq \cA_{\cG_1}\times_{D_{y_1}^\times}\Hit_\fI\times\cdots\times_{D_{y_\ell}^\times}\Hit_\fI
\xrightarrow{\sim} \Hit(\bP^1)_{\cG_1}\times_{D_{y_1}^\times}\Hit_\fI\times\cdots\times_{D_{y_\ell}^\times}\Hit_\fI
\simeq \Hit(\bP^1)_{\cG_\ell}.
$$
\end{proof}

\subsubsection{Proof of Proposition \ref{p:j^+ image}.(i).}
By Lemma \ref{l:j^+ image range}, it remains to show the image $h^{cl}((\fj^{+})^{\perp})$ is dense in $\Hit_{\fj^+}$.
Recall Example \ref{Hitchin map I}(i) that $\Hit(D)=\fc^*\times^{\bGm}\omega_{\cO}^\times$. It suffices to show that for any positive integer $N>0$, the image is dense in the quotient $\Hit_{\fj^+}/s^N\Hit(D)$. 
Consider the following commutative diagram
$$
\begin{tikzcd}
\cA_{\cG_\ell} \arrow{r} \arrow{d}{\wr} &\cA_{\fj^+} \arrow[hookrightarrow]{r} &\Hit_{\fj^+}\arrow{d}\\
\Hit(\bP^1)_{\cG_\ell} \arrow{rr}{\theta} &   &\Hit_{\fj^+}/s^N\Hit(D)
\end{tikzcd}
$$
We will show that for any $N>0$, we can choose $\ell\gg N$ such that $\theta$ is surjective. 
We do this by comparing dimensions, i.e. we verify the equlity
\begin{equation}\label{eq:Hitchin dim equality}
\dim \Hit(\bP^1)_{\cG_\ell}=\dim\ker(\theta)+\dim\Hit_{\fj^+}/s^N\Hit(D).
\end{equation}
Indeed, we have
\begin{align*}
\dim \Hit(\bP^1)_{\cG_\ell}
=&\sum_{d_i<d}\dim\Gamma(\bP^1,\Omega^{d_i}(d_i x+(d_i-1)(y_1+\cdots+y_\ell)))\\
&+\sum_{d_i\geq d}\dim\Gamma(\bP^1,\Omega^{d_i}((d_i+1) x+(d_i-1)(y_1+\cdots+y_\ell)))\\
=&m\sum_i(d_i-1)-\sum_i d_i+n+\sum_{d_i\geq d}1.
\end{align*}
\begin{align*}
\dim\Hit_{\fj^+}/s^N
&=\sum_{d_i<d}\dim(s^{-d_i}\cO/s^N\cO)+\sum_{d_i\geq d}\dim(s^{-d_i-1}\cO/s^N\cO)\\
&=Nn+\sum_i d_i+\sum_{d_i\geq d}1.
\end{align*}

If we take $\ell$ large enough such that $\ell(d_i-1)-N-2d_i>0$ for every $d_i$, then we also have
\begin{align*}
\dim\ker(\theta)
=&\sum_{d_i<d}\dim\Gamma(\bP^1,\Omega^{d_i}(-N x+(d_i-1)(y_1+\cdots+y_\ell)))\\
&+\sum_{d_i\geq d}\dim\Gamma(\bP^1,\Omega^{d_i}(-N x+(d_i-1)(y_1+\cdots+y_\ell)))\\
=&-Nn+m\sum_i(d_i-1)-2\sum_i d_i+n.
\end{align*}
Combining above, we get \eqref{eq:Hitchin dim equality}, which completes the proof of Part (i) of the proposition. \hfill\qed

\subsubsection{Proof of Proposition \ref{p:j^+ image}.(ii)} \label{sss:pf j^+ image ii}
The morphism $p$ in \eqref{eq:local Hitchin diagram} is a quotient by $\Hit_{\fj}$. In order to obtain the right vertical map and the commutative diagram, we need to show
$$
(p\circ h^{cl})(X+Y)=(p\circ h^{cl})(X),\qquad \forall X\in(\fj^{+})^{\perp},\ Y\in\fj^\perp.
$$

To prove this assertion, we may assume $\cg=\gl_{n+1}$ in the type A case.  
We set $\fu^c_-=\oplus_{\alpha\in \Phi(\fu^c)}\fg_{-\alpha}$, $\ft_{\fm}=\prod_{j=n-m+1}^n \Lie(T_j)$ in type B, C, D or $\ft_{\fm}=\prod_{j=n-m+1}^{n+1} \Lie(T_j)$ in type A. 
Similar to \eqref{eq:DK* of j+perp}, we have
$$
DK^*(\fj^\perp)=(\ft_{\fm}\oplus\fu_0\oplus\fu^c_-\oplus\prod_{j\geq 1}\fg_j u^j)\frac{\td u}{u}.
$$
As $\fu_\phi^{\perp}=\ft\oplus \fu_0\oplus\fu^c_-$, we have a decomposition
$$
DK^*((\fj^{+})^{\perp})=DK^*(\fj^\perp)\oplus(\ft_\phi\oplus\fm_{-1}u^{-1})\frac{\td u}{u}. 
$$

After taking $s=u^d$, we have $Z=\bigoplus_{d_i\geq d}\bC u^{-d}(\frac{\td u}{u})^{d_i}$.
Denote by $t_{d_i,d+d_i}$ the coefficients of $u^{-d-d_i}\lambda^{d_i}$ in the determinant $\det(\lambda I+X)$. For $X\in DK^*((\fj^{+})^{\perp})$, $(p\circ h^{cl})(X)$ is calculated by $t_{d_i,d+d_i}$ for $d_i\geq d$.
Each $t_{d_i,d+d_i}$ is a sum of products $x_1x_2\cdots x_N$, where $x_i$'s are entries of some principal minor $X_I$ of $X$ belonging to different columns and rows. 

The proposition amounts to show that each $x_i$ in $t_{d_i,d+d_i}$ belongs to $(\ft_\phi\oplus\fm_{-1}u^{-1})\frac{\td u}{u}$. 

There are at most $d$ non-zero entries in the matrix of $\fm_{-1}$ and these entries are in different columns and rows. 
If $x_1x_2\cdots x_N$ is a summand of $t_{d_i,d+d_i}$, there are $d$ terms of $x_i$'s belonging to $\fm_{-1}u^{-1}\frac{\td u}{u}$, and the rest of $x_i$ are entries of $DK^*((\fj^{+})^{\perp})\cap\fg\frac{\td u}{u}=(\fb_0\oplus\fu^c_-)\frac{\td u}{u}$ and lie in different columns and rows to those of $\fm_{-1}$. 
By Lemma \ref{l:roots in U_phi}, $-\gamma\in\Phi(\fu^c_-)$ if and only if there exists $\alpha\in\Phi(\fm_{-1})$ such that $\gamma+\alpha\in\Phi_G$.
This implies and any element of $\fu^c_-$ has a common column or row with $\fm_{-1}$, thus cannot appear as $x_i$. 
So are elements of $\ft_{\fm}$ and of $\fu^c\subset\fb_0=\ft\oplus\fu^c\oplus\fu_\phi$. Thus the rest of $x_i$ can only come from $\ft_\phi\oplus\fu_\phi$, and their product is a summand of the determinant of $\ft_\phi\oplus\fu_\phi$. Therefore, they all belong to $\ft_\phi$. This completes the proof. \hfill \qed

\subsection{Proof of $Z\subseteq \cA_{\fj^+}$} \label{l:Z dense}
Indeed, we show a slightly stronger result; namely, we shall prove that 
\begin{equation}
Z\subset\overline{h^{cl}(\fm_{-1}u^{-1}\oplus\fu_\phi^\perp)},
\end{equation}
where $\fm_{-1}u^{-1}\oplus\fu_\phi^\perp$ are defined in \eqref{eq:DK* of j+perp}. 
Recall \S \ref{sss:pf j^+ image ii} that $\fu_\phi^\perp=\fb_0\oplus\fu^c_-$. For $X\in \fm_{-1}u^{-1}\oplus\fu_\phi^\perp$, we write
\[
X=X_1+X_2+X_3, \quad \textnormal{where}~ X_1\in \fm_{-1}u^{-1}, X_2\in\fb_0, X_3\in\fu^c_-.
\]
Moreover, we set $X_2=X_2'+X_2''$, where $X_2'=(x_1,x_2,...,x_n)\in\ft$ and $X_2''\in\fu_0$.

By Lemma \ref{l:roots in U_phi}, a root $\gamma\in\Phi(\fu_0)$ is a weight of $\fu^c$ if and only if there exists $\alpha\in\Phi(\fm_{-1})$ such that $\gamma+\alpha\in\Phi_G$. This means entries of $\fu^c_-$ has common row or column with those of $\fm_{-1}$. 
Since the Hitchin map of classical group is given by coefficients of characteristic polynomials, above discussion shows
$$
\det(\lambda I+X)=\det(\lambda I+X_2+X_3)+u^{-d}\det(X_1)\prod_{i=1}^{n-m}(\lambda+x_i)
$$ 
where for $X_1=u^{-1}\sum_{\gamma\in\Phi(\fm_{-1})}c_\alpha E_\alpha$, $\det(X_1)=c_{\theta_M}\prod_{-\alpha\in\Delta_M}c_\alpha^{a_\alpha}$, $a_\alpha$ the coefficient of $-\alpha$ in $\theta_M$.

We keep the notation and assumption of \ref{sss:pf j^+ image ii}. Recall that $\ft_\phi=\oplus_{i=1}^{n-m}\Lie(T_i)$. Thus we have $Z=p(\overline{h^{cl}(u^{-1}\fm_{-1}\oplus\ft_\phi)})$. 
It remains to show that fixing the coefficients $(x_i)_{i=1}^{n-m}$ of $X_2$ in $\ft_\phi$, we can choose coefficients $(x_i)$ of $X_2$ in $\ft_\fm$ and of $X_3$, such that $\det(\lambda I+X_2+X_3)=\lambda^N$, where $N=n+1,2n+1,2n,2n$ for type A, B, C, D respectively. 

We can decompose $\fg_0$ as simple Lie subalgebras $\fl_i$ of type A \eqref{eq:G_0 block decomp}. 
Let $X_{2,i}$ and $X_{3,i}$ be the restriction of $X_2$ and $X_3$ to $\fl_i$ and $N_i$ the size of $\fl_i$. 
We reduce to prove the following claim: there exist coefficients of $X_{2,i}$ in $\ft_\fm$ and of $X_{3,i}$ such that
\begin{equation}\label{eq:claim det}
\det(\lambda I+X_{2,i}+X_{3,i})=\lambda^{N_i}.
\end{equation}

Label the simple roots of $\fl_i$ by $\beta_1,\beta_2,...,\beta_{n_i}$ so that they give roots from left to right in Dynkin diagram of $L_i$. We set $\beta_{i,j}=\sum_{k=i}^{j-1}\beta_k$. 

With the notation of \S~\ref{ss:G_0 simple roots}, the following lemma follows from a straightforward case-by-case verification. 
\begin{lem}\label{l:u^c in blocks}
\textnormal{(i)} 
Except the case where $G$ is of type B when $d\leq n<\frac{3}{2}d$ and $i=m$, we have $\ft_{\fm}\cap \fl_i=\Lie(T_{u_i})$ \eqref{sss:pf j^+ image ii} for a unique integer $u_i$. 
	
\textnormal{(ii)} Since there exist only type I weights (Proposition \ref{p:G type v.s. weight type}), we consider $\Phi_{T_\phi}(\fu^c)$ as a subset of $\Phi_G$. 
There exists an integer $1\leq v_i\leq n_i$ such that 
$$
\Phi_{T_\phi}(\fu^c)\cap\Phi(\fl_i)=\{\beta_{1,v_i},\beta_{2,v_i},...,\beta_{v_i-1,v_i} \}\sqcup\{\beta_{v_i,v_i+1},\beta_{v_i,v_i+2},...,\beta_{v_i,n_i}\}.
$$
Moreover, $\beta_{v_i-1,v_i}=\chi_j\pm\chi_{i'}$ for some $j<i'$ or $\beta_{v_i,v_i+1}=\chi_{i'}\pm\chi_j$ for some $i'<j$.
\end{lem}

If we take $X_{3,i}=\sum_{\gamma\in\Delta(L_i)}E_\gamma$, the claim \eqref{eq:claim det} follows from Lemmas \ref{l:u^c in blocks} and the calculation of a characteristic polynomial in Lemma \ref{l:image of type A block}. 
The exceptional case in type B can also be checked in a similar way. This complete the proof of $Z\subseteq \cA_{\fj^+}$.
\hfill \qed

\subsubsection{Calculation of a characteristic polynomial} 
Given positive integers $N>1$ and $i\in [1,N]$, consider the following matrix:
\begin{equation}
X:=\sum_{j=1}^N x_jE_{j,j}+\sum_{j=1}^{n-1}E_{j,j+1}+\sum_{j=1}^{i-1}y_jE_{i,j}+\sum_{j=i}^{n-1} y_jE_{j+1,i},
\end{equation}
$$
X=\begin{pmatrix}
x_1& 1 &      &       &       &       &      &   \\
   &x_2& 1    &       &       &       &      &   \\
   &   &\ddots&\ddots &       &       &      &   \\
   &   &      &x_{i-1}&1      &       &      &   \\
y_1&y_2&\cdots&y_{i-1}&x_i    &1      &      &   \\
   &   &      &       &y_i    &x_{i+1}&\ddots&   \\
   &   &      &       &\vdots &       &\ddots&1  \\
   &   &      &       &y_{N-1}&       &      &x_N  
\end{pmatrix}
$$

\begin{lem}\label{l:image of type A block}
For any $1\leq i\leq N$, given any complex numbers $x_1,x_2,...,x_{i-1},x_{i+1},...,x_N$ where $x_j\neq x_{j'}, \forall1\leq j\leq i-1$, $i+1\leq j'\leq N$, there exist $x_i,y_1,y_2,...,y_{N-1}$ such that $\det(\lambda I_N+X)=\lambda^N$.
\end{lem}
\begin{proof}
One can write
\begin{equation}
\det(\lambda I_N+X)=\lambda^N+(\sum_{j=1}^N x_j)\lambda^{N-1}+\sum_{k=1}^{N-1}(\sum_{j=1}^{N-1}A_{kj}y_j+C_k)\lambda^{N-1-k},
\end{equation}
where $A_{kj}$ is constant coefficient polynomial in $x_j$ for $j\neq i$; $C_k$ is constant coefficient polynomial in all of $x_j$. Then $\det(\lambda I_N+X)=\lambda^N$ is equivalent to equations
$$
x_i=-\sum_{j\neq i}x_j; \quad \sum_{j=1}^{N-1}A_{kj}y_j=-C_k,\quad 1\leq k\leq N-1.
$$

Consider the $(N-1)\times(N-1)$ matrix $A=(A_{kj})_{1\le k,j\le N-1}$, where the entries $A_{kj}$ are polynomials in $x_j$, $j\neq i$. 
One can show by induction on $i$ that
\begin{equation}\label{eq:det of coeff}
\det(A)=\pm\prod_{j=1}^{i-1}\prod_{j'=i+1}^N(x_j-x_{j'}).
\end{equation}
The assumption of lemma guarantee that $\det(A)\neq 0$. Thus above equations always have solutions $x_i,y_1,...,y_{N-1}$.
\end{proof}

\subsection{Proof of Proposition \ref{p:global Hitchin diagram}}\label{ss:pf global Hitchin}
With the notation of \S \ref{sss:global Hitchin Gm} for $\cG=\cG_1$ and $y_1=0$, the first claim follows from Lemma \ref{eq:global Hitcin image}. 
Since the global Hitchin map is compatible with the local one, the commutativity of diagram \eqref{eq:global Hitchin diagram} follows from that of \eqref{eq:local Hitchin diagram}. 
In view of the isomorphism (\S \ref{sss:global Hitchin Gm}) 
$$
\Hit(\bP^1)_\cG\simeq \bigoplus_{d_i\geq d}\Gamma(\bP^1,\Omega^{d_i}((d_i-1)0+(d_i+1) \infty)),
$$
we deduce that $\Hit(\bP^1)_{\cG}\to Z$ is an isomorphism. \hfill \qed

\section{Study of the space of opers $\Op_{\cg}(\bP^1)_{\cG}$} \label{s:opers proof}

\subsection{Proof of Proposition \ref{p:local quant}} 
\label{ss:pf local quant} 
We first make some preparations. 
For every root $\alpha\in\Phi_G$, we set
$$
r_\alpha=\min\{r\in\bZ\,|\,s^r\fg_\alpha\subset\fj^+\}, \quad \textnormal{and} \quad r_0=1.
$$
We fix an order on a basis $\{E_{\alpha}[d]=s^d E_{\alpha}\}$ of $\fg(\!(s)\!)$:  
\begin{eqnarray} \nonumber
\cdots <E_{-\theta}[r_{-\theta}-1]<E_{-(\theta-\alpha_2)}[r_{-(\theta-\alpha_2)}-1]<\cdots < E_{\theta-\alpha_2-\alpha_3}[r_{\theta-\alpha_2-\alpha_3}-1]< E_{\theta-\alpha_2}[r_{\theta-\alpha_2}-1]\\
<E_{\theta}[r_{\theta}-1] < E_{-\theta}[r_{-\theta}]<\cdots < E_{\theta}[r_{\theta}]< \cdots \label{eq:order roots}
\end{eqnarray}
according to $r_{\alpha}$, heights and the relationship $\alpha_i>\alpha_{i+1}$. 
By Poincaré--Birkhoff--Witt theorem, there is a basis $\mathcal{B}$ of $\Vac_{\fj^+}=\Ug/\Ug(\fj^+ + \bC\bone)$ consisting of tensors $E_1\otimes E_2\otimes\cdots\otimes E_k$ where $E_j\in \{E_\alpha[r_\alpha-r]\,|\,r\in\bN^+,\alpha\in\Phi_G\}$ and $E_1\leq E_2\leq\cdots\leq E_k$. 

A Segal-Sugawara vector $S_i$ can be written as a linear combination of tensors $E_{\beta_1}[n_1]E_{\beta_2}[n_2]\cdots E_{\beta_k}[n_k]$ in $\mathcal{B}$ satisfying (c.f. \cite[Lemma10]{CK})
\begin{equation} \label{eq:SS operators}
\textnormal{(i)}.\ k\leq d_i;\qquad\qquad\quad
\textnormal{(ii)}.\ \sum_{b=1}^k\beta_b=0;\qquad\qquad\quad
\textnormal{(iii)}.\ \sum_{b=1}^k n_b=-d_i. 
\footnote{We use $d_i$ to denote fundamental degree, while in \emph{loc.cit.} $d_i$ stands for exponent and is different to fundamental degree by $1$. }
\end{equation}
Then the associated Segal-Sugawara operator $S_{i,j}$ is a linear combination of tensors 
$$X=E_{\beta_1}[n_1]E_{\beta_2}[n_2]\cdots E_{\beta_k}[n_k],$$
where the set of $\beta_j$ is same as that of a summand of $S_i$. 
Thus $X$ still satisfies conditions (i),(ii) above, while condition (iii) would be replaced by $\sum_{b=1}^k n_b=j-d_i+1$.

For a tensor $X=E_{\beta_1}[n_1]E_{\beta_2}[n_2]\cdots E_{\beta_k}[n_k]$ of $\Ug$, we rewrite its image $X\cdot\bone$ in $\Vac_{\fj^+}$ as a linear combination of tensors in $\mathcal{B}$.
In this process, we move some factor $E_{\beta_b}[n_b]\in\fj^+$ to the right using Lie bracket, and rearrange rest of factors. 
This process preserves $\sum_b\beta_b$ and $\sum_b n_b$, while $k$ may decrease. 
We collect above discussion in the following lemma:

\begin{lem}\label{l:S_ij tensor}
Let $S_i=\sum_{\underline{\gamma},\underline{m}} \lambda_{\underline{\gamma},\underline{m}} E_{\gamma_1}[m_1]E_{\gamma_2}[m_2]\cdots E_{\gamma_k}[m_k]$ 
be the decomposition of a Segal-Sugawara vector with respect to $\mathcal{B}$, indexed by pairs $(\underline{\gamma},\underline{m})$. 
Then each non-zero component
	
$E_{\beta_1}[n_1]E_{\beta_2}[n_2]\cdots E_{\beta_k}[n_k]\cdot\bone\in \mathcal{B}$ of $S_{i,j}\cdot\bone \in \Vac_{\fj^+}$ satisfies
\begin{itemize}
\item [(i)] $k\leq d_i$;
\item [(ii)] There exists an index $(\underline{\gamma},\underline{m})$ and a partition of $\underline{\lambda}=\{\gamma_1,\dots,\gamma_{m_c}\}$ such that each $\beta_b$ is a sum of roots in one of these subsets. 
In particular, $\sum_{b=1}^k\beta_b=0$;
\item [(iii)] $\sum_{b=1}^k n_b=j-d_i+1$, and each $n_b\leq r_{\beta_b}-1$. 
\end{itemize}
\end{lem}

\subsubsection{} \label{l:S_ij vanish}
Next, recall \eqref{eq:decomposition j+} that $\fj^+=\fu_\phi\oplus\fk\oplus\fp(2)$. For any subspace $\fh\subset\fj^+$, we denote the set of finite part of the affine root subspaces contained in $\fh$ by $\Phi(\fh)\subset\Phi_G$. Note that $\Phi(\fu_\phi)\subset\Phi(\fg_0)$ is disjoint from $\Phi(\fk)\sqcup\Phi(\fm_1)=\Phi(\fg_1)$.
With this notation, $r_\alpha$ is equal to
$$
r_\alpha=
\begin{cases}
1-\frac{\Ht\alpha-1}{d},\quad &\alpha\in\Phi(\fm_1);\\
1-\lceil\frac{\Ht\alpha}{d}\rceil,\quad &\alpha\in\Phi_G-\Phi(\fm_1)-\Phi(\fu_\phi);\\
-\frac{\Ht\alpha}{d},\quad\quad\ \ &\alpha\in\Phi(\fu_\phi);\\
1,\quad\quad\qquad\ \ &\alpha=0.
\end{cases}
$$

If $S_{i,j}\cdot\bone\neq0$, then it has non-zero coefficient in at least one basis $X=E_{\beta_1}[n_1]E_{\beta_2}[n_2]\cdots E_{\beta_k}[n_k]\cdot \bone$ described in Lemma \ref{l:S_ij tensor}. For this tensor $X$, we have an inequality
\begin{equation}\label{eq:m_1 roots bound}
\begin{split}
j-d_i+1&=\sum_{b=1}^k n_b\leq \sum_{b=1}^k (r_{\beta_b}-1)
\leq\sum_{\beta_b\not\in\Phi(\fm_1)}(-\lceil\frac{\Ht\beta_b}{d}\rceil)+\sum_{\beta_b\in\Phi(\fm_1)}(-\frac{\Ht\beta_b-1}{d})\\
&\leq\sum_{b=1}^k(-\frac{\Ht\beta_b}{d})+\frac{|\{\beta_b\in\Phi(\fm_1)\}|}{d}
=\frac{|\{\beta_b\in\Phi(\fm_1)\}|}{d}.
\end{split}
\end{equation}
In particular, we have $j\leq d_i-1+\frac{d_i}{d}$. 
Then we conclude Lemma \ref{l:ker(fZ to fZ_j^+)}(i).

\subsubsection{Conclusion of the proof} \label{sss:proof local quant}
We simultaneously prove Proposition \ref{p:local quant}(i) and Lemma \ref{l:ker(fZ to fZ_j^+)}(ii). 
Proposition \ref{p:local quant}(ii) follows from \ref{p:local quant}(i). 

We first show that the image of $S_{i,d_i}$ lies in $U({\overline{\fj}})$ for $d_i\ge d$. 
Note that these operators all have degree $1$. 
Consider the decomposition of $S_{i, d_i}\cdot\bone$ in Lemma \ref{l:S_ij tensor}. For any non-zero component $X=E_{\beta_1}[n_1]\cdots E_{\beta_k}[n_k]\cdot \bone$ in $S_{i,d_i}\cdot \bone$, we need to show that $E_{\beta_b}[n_b]\in\fj$ for every $b$. 

For each root $\beta$, we set $E(\beta)=r_{\beta}-1+\frac{\Ht(\beta)}{d}$. Note that $-1<E(\beta)\le 1/d$. The equality holds when $\beta\in \Phi(\fm_1)$.
We will show that $n_b=r_b-1$. 
Then we deduce an equality
\begin{equation}
1=\sum_{b=1}^{k} E(\beta_b) = \sum_{b=1}^k E(\beta_b)=
\frac{|\{\beta_b|\beta_b\in \Phi(\fm_1)\}|}{d}+ \sum_{\beta_b\not \in \Phi(\fm_1)} E(\beta_b),
\label{equality Ebeta}
\end{equation}
and therefore $d\le |\{\beta_i|\beta_i\in \Phi(\fm_1)\}| \le d_i$. 

(a). In the case of type A (resp. type C), $S_i$ are constructed in \cite[Theorem 3.1]{MolevA} (resp. \cite[Theorem 4.4]{Yakimova}). 
The above operators $S_i$ are constructed from coefficients of certain characteristic polynomials. 
By Lemma \ref{l:S_ij tensor}(ii), we obtain $|\{\beta_b\in\Phi(\fm_1)\}|\leq d$ and therefore the equality $\sharp\{\beta_b\in\Phi(\fm_1)\}=d$. 
Thus $\beta_b=-\theta_M$ for a unique $b$, and the number of $b$ such that $\beta_b=\alpha_j\in\Delta_M$ equals to coefficient of $\alpha_j$ in $\theta_M$. This implies
$
\sum_{\beta_b\in\Phi(\fm_1)}\beta_b=0=\sum_{\beta_b\not\in\Phi(\fm_1)}\beta_b.
$ 

Then all of equalities in \eqref{eq:m_1 roots bound} hold. 
We deduce that $n_b=r_{\beta_b}-1$ for any $b$. 
Thus if $\beta_b\not\in\Phi(\fm_1)$, then we have $n_b=-\frac{\Ht(\beta)}{d}$ and $\beta_b\in\Phi(\fl)-\Phi(\fu_\phi)$.
Also, since $S_i$ is defined from coefficient of characteristic polynomials, and the root vector of any root in $\Phi(\fu^c)$ has common column or row with $\fm_1$, we also have $\beta_b\not\in\Phi(\fu^c)$ in this case, thus $\beta_b\in\Phi(\fb^-_L)=\Phi^-(\fl)\sqcup\{0\}$. 
From $\sum_{\beta_b\not\in\Phi(\fm_1)}\beta_b=0$, we deduce $\beta_b=0$ for any $\beta_b\not\in\Phi(\fm_1)$. 
Moreover, we have $E_{\beta_b}\in\ft_\phi$ since it cannot have common column or row with $\fm_1$. 
We obtain $E_{\beta_b}[n_b]\in\fj$ for every $b$, which completes the proof of assertion (i) in this case.

In view of $|\{\beta_b\in\Phi(\fm_1)\}|\leq d$ and inequality \eqref{equality Ebeta}, we conclude that $S_{i,j}\in \ker(\fZ\to \fZ_{\fj^+})$ for $j\ge d_i+1$. This finishes the proof of Lemma \ref{l:ker(fZ to fZ_j^+)}(ii). 

(b). In the case of type B or D and $d=h_G-2$ or $h_G$, all of equalities in \eqref{eq:m_1 roots bound} hold as $d_i\ge d$. 
We deduce that $n_b=r_{\beta_b}-1$ for any $b$. 

Then the length $k$ of $X$ may be $d,d+1,d+2$. 
When $k\le d+1$, in view of the equality \eqref{equality Ebeta}, 
there exists at most one root $\beta_b$ satisfying $E(\beta_b)=0$ and other roots belong to $\Phi(\fm_1)$.
In view of \eqref{eq:SS operators}(ii), we deduce that $\beta_b$ is a zero root. 
Moreover, since $S_{i,d_i}$ belongs to $\Vac_{\fj^+}^J$, the action of $E_{\alpha_i}$ on $S_{i,d_i}$ is trivial. Then we deduce that $E_{\beta_b}$ belongs to $\ft_{\phi}$ and the assertion in this case follows. 

When $k=d+2$, then $X$ is a component of $S_{n,d_n}$ and its image $\overline{X}$ in the associated graded $\gr A$ is a non-zero component of $h_{n,d_n}$. 
As $h_{n,d_n}$ is constructed from the invariant polynomials of $\cg$, then the assertion in this case follows from the same argument in (a). \hfill \qed

\subsection{Proof of Proposition \ref{l:global oper}} \label{ss:global oper}
The space $\Op_{\cg}(\bP^1)_{\cG}$ is a closed subscheme of $\Op_{\cg}(\bGm)_{(0,\varpi(0)),(\infty,\frac{1}{d})}$ \eqref{sss:functoriality}. 
By Lemma \ref{l:oper le 1/d}, we have
$$\mathbb{C}[\Op_{\cg}(\bGm)_{(0,\varpi(0)),(\infty,\frac{1}{d})}]\simeq \mathbb{C}[\lambda_{ij}]_{d_i\ge d, 0\le j\le \lfloor d_i/d\rfloor-1}.$$

We calculate the classical limit of $\Op_{\cg}(\bGm)_{(0,\varpi(0)),(\infty,\frac{1}{d})}$ following \cite[\S 3.1.14]{BD}. 
We consider a family of $\cg$-$\hbar$-opers on $\bGm$:
\[
\nabla=
\hbar \partial_t+t^{-1}p_{-1}+
\sum_{d_i\geq d}\lambda_i(t)p_i,\quad \lambda_i(t)=\sum_{j=0}^{\lfloor d_i/d\rfloor-1}t^j \lambda_{ij},\quad \lambda_{ij}\in\bC.
\]
These $\cg$-$\hbar$-opers form an affine space $\widetilde{\Op}_{\cg}(\bGm)_{(0,\varpi(0)),(\infty,\frac{1}{d})}$ flat over $\Spec(\mathbb{C}[\hbar])$, whose fiber at $\hbar=1$ is $\Op_{\cg}(\bGm)_{(0,\varpi(0)),(\infty,\frac{1}{d})}$. 
By a simple gauge transform, we see it is stable under the natural action of $\bGm$, sending $\nabla$ to $\lambda \cdot \nabla$, for $\lambda\in \mathbb{C}^*$. 
Its fiber at $\hbar=0$ is isomorphic to the associated Hitchin space. 

\begin{lem}\label{l:classical limit oper}
Assume moreover that $d>n$ when $\hG$ is of type $D_n$.  
The characteristic polynomials $\{c_{i}\}$ defined in \S \ref{ss:local Hitchin} induces an isomorphism
\begin{eqnarray*}
\biggl(\widetilde{\Op}_{\cg}(\bGm)_{(0,\varpi(0)),(\infty,\frac{1}{d})}\biggr)_0 &\xrightarrow{\sim}& \bigoplus_i \Gamma(\bP^1, \Omega^{d_i}( (d_i-1)0+ (d_i+\lfloor\frac{d_i}{d}\rfloor)\infty)),  \\
X=(t^{-1}p_{-1}+
\sum_{d_i\geq d}\lambda_i(t) p_i)
&\mapsto& (c_{d_i}(X) (dt)^{d_i})_{i}.
\end{eqnarray*}
Moreover, if $t$ denotes a coordinate of $\bP^1$ with a simple pole at $\infty$, and $\delta_{ij}$ denotes the coefficient of $t^{1+j} (\frac{dt}{t})^{d_i}$, then above isomorphism sends
\begin{equation} \label{eq:description classical limit}
\delta_{ij} \mapsto u_{ij} \lambda_{ij} + Q_{ij}(\lambda_{\mu}),
\end{equation}
where $u_{ij}$ is a constant and $Q_{ij}$ is a polynomial in variables $\lambda_{\mu\nu}$ with $\mu<i, \nu<j$. 
In particular, $\delta_{i0}$ is sent to a scalar multiple of $\lambda_i$($=\lambda_{i0}$). 
\end{lem}

\begin{proof}
It suffices to prove the relationship \eqref{eq:description classical limit}. 
By functoriality of canonical opers form \S \ref{sss:functoriality}, it suffices to show the lemma for $\cg=\Sl_n$. 
We set $\lambda_{i0}=\lambda_i$ for simplicity. 
We proceed the proof by induction on $n$. When $n=d$, the assertion is clear.  
	
We first consider $c_n(X)=\det(X)$, which is expressed as $(\sum_{j=0}^{\lfloor n/d \rfloor -1} t^j \lambda_{ij})t^{-(n-1)}$ plus a sum of products of determinants of certain principal minors of $X$. 
We can calculate these determinant by induction hypothesis. 
Then the assertion for $c_n(X)$ follows from the induction hypothesis and the following inequality that for $n=n_1+\dots+n_k$, we have:
\[
\lfloor \frac{n_1}{d} \rfloor+\dots + \lfloor \frac{n_k}{d} \rfloor \le \lfloor \frac{n}{d} \rfloor.
\]
	
For $i<n$, $c_i(X)$ is calculated by a sum of products of determinants of certain principal minors of $X$. 
We can conclude the assertion for $c_i(X)$ by the induction hypothesis and a similar argument as above. 
\end{proof}

Now, we can prove Proposition \ref{l:global oper}.
We consider the following filtered homomorphisms:
\begin{equation}
\mathbb{C}[\lambda_{i0}]_{d_i\ge d} \hookrightarrow \mathbb{C}[\lambda_{ij}] \simeq  \mathbb{C}[\Op_{\cg}(\bGm)_{(0,\varpi(0)),(\infty,\frac{1}{d})}] 
\twoheadrightarrow \mathbb{C}[\Op_{\cg}(\bP^1)_{\cG}].
\label{eq:composition opers}
\end{equation}
By Lemma \ref{l:classical limit oper}, the composition of the their associated graded is an isomorphism:
\begin{align*}
\Spec(\gr \mathbb{C}[\lambda_{i0}]_{d_i\ge d}) &\leftarrow \bigoplus_i \Gamma(\bP^1, \Omega^{d_i}( (d_i-1)0+ (d_i+\lfloor\frac{d_i}{d}\rfloor)\infty))\\ &\hookleftarrow 
\Hit(\bP^1)_{\mathcal{G}} \simeq \bigoplus_{d_i\ge d} \Gamma(\bP^1, \Omega^{d_i}( (d_i-1)0+ (d_i+1)\infty)).
\end{align*}
Then so is the composition of \eqref{eq:composition opers}. 
This finishes the proof. \hfill \qed 

\subsection{Proof of Proposition \ref{p:Hyp opers}} \label{ss:pf Hyp opers}
We will prove a more general result.
Consider the following differential equations:
\[
\Hyp(u_m(t),\dots,u_n(t))=\delta^{n}+t\biggl( u_m(t)\delta^{n-m}+\dots+u_n(t)\biggr),
\]
where $u_{i}(t)=\sum_{j\ge 0} u_{ij} t^j\in \mathbb{C}[t]$,
and the following $\Sl_n$-opers:
\begin{equation} \label{eq:opers principal nilpotent}
\nabla(\lambda_m(t),\dots,\lambda_n(t))=
\partial_t+t^{-1}p_{-1}+
\sum_{d_i\geq d}\lambda_i(t)p_i,
\end{equation}
where $\lambda_i(t)=\sum_{j\ge 0} \lambda_{ij} t^j\in \mathbb{C}[t]$.

\begin{lem} \label{l:isomorphism companion}
We set $\deg \lambda_{ij}=\deg u_{ij}= j+1$. 
There exists an isomorphism of homogeneous $\mathbb{C}$-algebras 
\begin{equation}
\mathbb{C}[\lambda_{ij}]\xrightarrow{\sim} \mathbb{C}[u_{ij}],\quad
u_{ij}= c_i \lambda_{ij}+ Q_{ij}(\underline{\lambda}),
\label{eq:formula Hypopers}
\end{equation}
where $c_i$ is the sum of coefficients of $p_{i}$ and $Q_{ij}$ is a polynomial in variables $\lambda_{\mu \nu}$ with $\mu < i, \nu \le j$, such that under this isomorphism, the connection $\Hyp(u_m(t),\dots,u_n(t))$ is isomorphic to $\nabla(\lambda_m(t),\dots,\lambda_n(t))$. 
\end{lem}

\begin{proof}
Let $(e_1,\dots,e_n)$ be a basis associated to the connection $\nabla(\lambda_m(t),\dots,\lambda_n(t))$. 
Then the companion connection form
\[
\nabla= d+
\begin{pmatrix}
0 & 1 & 0 & \dots & 0 \\
0 & 0 & 1 & \dots & 0 \\
\vdots & \vdots & \ddots & \ddots & \vdots \\
u_{m}(t)t &0 & \dots & \dots & \dots \\
\vdots & \vdots & \vdots & \vdots & \vdots \\
u_{n-1}(t)t & 0 & 0 & \dots & 1 \\
u_n(t) t & 0 & 0 & \dots & 0  \end{pmatrix}\frac{dt}{t}.
\]
associated to the basis $(e_1,\delta(e_1),\dots,\delta^{n-1}(e_1))$ allows us to define a homogeneous homomorphism:
\[
\mathbb{C}[\lambda_{ij}] \to \mathbb{C}[u_{ij}].
\]
It suffices to show that the above map satisfies the relationship in \eqref{eq:formula Hypopers}.
Indeed, if we know this relationship, then we prove \eqref{eq:formula Hypopers} by defining an inverse: 
\begin{equation} \label{eq:inverse companion}
\lambda_{ij}=\frac{1}{c_i} u_{ij}+R_{ij}(\underline{u}),
\end{equation}
where $R_{ij}$ is a polynomial in variables $u_{\mu\nu}$, with $\mu< i$ and $\nu\le j$. 

We can first determine $u_m(t)=c_m \lambda_m(t)$. 
Moreover, for every $i\ge m$, the gauge transform process shows that 
\[
tu_i(t)= c_i t\lambda_i(t)+ \sum_{\underline{\alpha}} c_{\underline{\alpha}} \prod_{m\le \ell \le i-1} \prod_{j=1}^{k_{\ell}} \delta^{\alpha_{\ell,j}}(t\lambda_{\ell}(t)),
\]
for certain pairs $\underline{\alpha}=(\alpha_{m,1},\dots, \alpha_{m,k_m},\dots,\alpha_{i-1,1},\dots,\alpha_{i-1,k_{i-1}})$ of positive integers and some constant $c_{\underline{\alpha}}$ associated to $\underline{\alpha}$.
Moreover, one can see that $\sum \alpha_{\ell,j}\le n$, and the above sum is therefore a finite sum. 
Then the assertion follows from the above formula. 
\end{proof}

\subsubsection{Proof of Proposition \ref{p:Hyp opers}(i)}
With the notation of Lemma \ref{l:isomorphism companion}, by \eqref{eq:formula Hypopers} and \eqref{eq:inverse companion}, if $u_{ij}=0$ for $j\ge 1$, then $\Hyp(u_m,\dots,u_n)$ is isomorphic to 
\[
\nabla=
\partial_t+t^{-1}p_{-1}+
\sum_{d_i\geq d}(\lambda_i+\sum_{j\ge 1} t^j P_{ij})p_i, 
\]
where $\lambda_i$ is a linear combination of $u_m,\dots,u_i$, and $P_{ij}$ is a polynomial in $\{\lambda_{\mu}\}_{\mu<i}$. 
The maximal slope of the above oper is $\le \frac{1}{m}$, which is the maximal slope of $\Hyp(u_m,\dots,u_n)$. Then we conclude assertion (i) by \cite[Proposition 3]{CK}. \hfill\qed

\subsubsection{}
To prove the assertion (ii), we consider the following companion form of $\Sl_n$-opers on $D_{\infty}^{\times}$: 
\begin{equation} \label{eq:companion partial}
\partial^n_s+v_2(s)\partial^{n-2}_s+\dots+v_n(s),
\end{equation}
where $s=t^{-1}$ is the coordinate at $\infty$ and $v_i(s)=\sum v_{i,j} s^{-j-1}\in \mathbb{C}(\!(s)\!)$. 
The coordinates $v_{i,j}$ form a set of topological generators of $\mathbb{C}[\Op_{\Sl_n}(D^{\times}_{\infty})]$. 
Under the isomorphism $\fZ\xrightarrow{\sim} \mathbb{C}[\Op_{\Sl_n}(D^{\times}_{\infty})]$, the Segal-Sugawara operator $S_{i,j}$ is sent to a scalar multiple of $v_{i,j}$ \cite[Corollary 4.9]{MasoudConductor} \footnote{Although \textit{loc.cit} shows the result for $\mathfrak{gl}_n$, one can deduce the corresponding result for $\Sl_n$ by letting $v_1(s)=0$ and $S_1=0$.}. 

\begin{lem}
\label{l:compare companions}
There exists a linear transformation given by 
\begin{eqnarray*}
v_{i,j}=u_{i,j-i}+ a_{i,j}^{i-1} u_{i-1,j-i}+\dots+a_{i,j}^{m}u_{m,j-i}, \quad m\le i\le n, j\ge i,\\
v_{i,j}=b_{i,j}, \quad m\le i\le n, j< i,
\end{eqnarray*}
where $a_{i,j}^{\ell},b_{i,j}$ are constants, 
such that $\Hyp(u_m(t),\dots,u_n(t))$ is isomorphic to \eqref{eq:companion partial} in $\Op_{\SL_n}(D^{\times}_{\infty})$.
\end{lem}

\begin{proof}
By taking $t=s^{-1}$, $\Hyp(u_m(t),\dots,u_n(t))$ is written as
\[
\partial_s+ \biggl(p_{-1} +u_m(s^{-1}) E_{n-m,n} +\dots + u_n(s^{-1})E_{n,n}\biggr)s^{-1}.
\]
By taking a gauge transform by $s^{\check{\rho}}$, the above connection is isomorphic to
\[
\partial_s + p_{-1}+ \check{\rho}s^{-1} + u_m(s^{-1})s^{-(m+1)} E_{n-m+1,n}+\dots+ u_n(s^{-1})s^{-n-1} E_{1,n}. 
\]
If $(e_1,\dots,e_n)$ denotes a basis for the above connection form, we set $f_i=\nabla^{i-1}_{\partial_s}(e_1)$. 
Then we have linear relationship 
\[
f_i= e_i+ c_{i,i-1} e_{i-1}s^{-1}+\dots+ c_{i,1} e_1 s^{-i+1}, \quad e_i=f_i+d_{i,i-1}f_{i-1}s^{-1}+\dots + d_{i,1} f_1 s^{-i+1},
\]
for certain constants $c_{i,j},d_{i,j}$. 
The connection matrix associated to the basis $(f_1,\dots,f_n)$ gives rise to the companion form \eqref{eq:companion partial}. 
Note that constants $c_{i,j}$'s and $d_{i,j}$'s would not contribute to $v_{i,j}$ if $j\ge i$. 
Then a simple calculation allows us to conclude the lemma. 
\end{proof}

\subsubsection{Proof of Proposition \ref{p:Hyp opers}(ii)}
Let $\Hyp$ be the affine space with coordinates $u_{i,j}$ with $m\le i\le n$ and $j\ge 0$. 
We can consider the following commutative diagram
\[
\xymatrix{
\Op_{\Sl_n}(\bP^1)_{\cG} \ar@{^{(}->}[r] \ar@{^{(}->}[d]_{Lemma~ \ref{l:isomorphism companion}} & \Op_{\fj^+}(D_{\infty}^{\times}) \ar@{^{(}->}[d] \\
\Hyp \ar[r]^-{Lemma~\ref{l:compare companions}} & \Op_{\Sl_n}(D_{\infty}^{\times})}
\]
where each arrow is a closed immersion. 
By Lemma \ref{l:ker(fZ to fZ_j^+)}(ii), the right vertical map sends $v_{i,j}$ to zero if $j\ge i+1$. 
In view of Lemma \ref{l:compare companions} and above diagram, the left vertical map sends $u_{i,j}$ to zero if $j\ge 1$. 
By Lemma \ref{l:isomorphism companion}, $\mathbb{C}[\Op_{\Sl_n}(\bP^1)_{\cG}]$ is generated by the image of $u_{m,0},\dots,u_{n,0}$. This finishes the proof of Proposition \ref{p:Hyp opers}(ii). \hfill\qed

\end{document}